\documentclass{amsart}
\pdfoutput=1

\usepackage{amsmath}
\usepackage{amssymb, color}
\usepackage[margin=1in]{geometry}

\usepackage{graphicx}
\usepackage{subfig}
\usepackage{float}
\usepackage{epsf}
\usepackage{hyperref}
\usepackage{titletoc}
\usepackage{eufrak}
\usepackage{ulem}
\usepackage[linesnumbered]{algorithm2e}
\usepackage{algpseudocode}
\usepackage{multirow}
\usepackage{bm}
\usepackage{arydshln}
\usepackage{pifont}
\usepackage{mathtools}
\DeclarePairedDelimiter{\ceil}{\lceil}{\rceil}

\numberwithin{equation}{section} 

\newtheorem{lemma}{Lemma}[section]
\newtheorem{theorem}{Theorem}[section]
\newtheorem{proposition}[theorem]{Proposition}
\newtheorem{corollary}[theorem]{Corollary}
\theoremstyle{remark}
\newtheorem{remark}{Remark}[section]
\newtheorem{definition}[theorem]{Definition}

\renewcommand{\phi}{\varphi}

\renewcommand{\tilde}{\widetilde}
\renewcommand{\hat}{\widehat}

\newcommand{\bd}{\boldsymbol}
\newcommand{\vct}[1]{\bd{#1}}

\newcommand{\econst}{\mathrm{e}}

\newcommand{\zerovct}{\bd{0}} % Zero vector
\newcommand{\R}{\mathbb{R}}

\newcommand{\N}{\mathbb{N}}

\newcommand{\CalH}{{\mathcal{H}}}
\newcommand{\CalK}{{\mathcal{K}}}
\newcommand{\CalL}{{\mathcal{L}}}
\newcommand{\CalM}{{\mathcal{M}}}
\newcommand{\CalP}{{\mathcal{P}}}

\newcommand{\dist}{{\mathrm{dist}}}
\newcommand{\rd}{\mathrm{d}}
\DeclareMathOperator*{\argmin}{arg\,min}

\DeclareMathOperator*{\esssup}{ess\,sup}

\newcommand{\revise}[1]{\textcolor{black}{#1}}
\newcommand{\referee}[1]{\textcolor{black}{#1}}

\title{Sparse operator compression of higher-order elliptic operators with rough coefficients}
\author{Thomas Y. Hou}
\thanks{Applied and Computational Mathematics, Caltech, Pasadena, CA 91125, USA. {\it Email: hou@cms.caltech.edu.}}
 \author{Pengchuan Zhang}
\thanks{Applied and Computational Mathematics, Caltech, Pasadena, CA 91125, USA. {\it Email: pzzhang@cms.caltech.edu.}}

\begin{document}
\maketitle
\begin{abstract}
We introduce the sparse operator compression to compress a self-adjoint higher-order elliptic operator with rough coefficients and various boundary conditions. The operator compression is achieved by using localized basis functions, which are energy-minimizing functions on local patches. On a regular mesh with mesh size $h$, the localized basis functions have supports of diameter $O(h\log(1/h))$ and give optimal compression rate of the solution operator. We show that by using localized basis functions with supports of diameter $O(h\log(1/h))$, our method achieves the optimal compression rate of the solution operator. From the perspective of the generalized finite element method to solve elliptic equations, the localized basis functions have the optimal convergence rate $O(h^k)$ for a $(2k)$th-order elliptic problem in the energy norm. From the perspective of the sparse PCA, our results show that a large set of Mat\'{e}rn covariance functions can be approximated by a rank-$n$ operator with a localized basis and with the optimal accuracy.
\end{abstract}

\section{Introduction}
\label{sec:introduction}
\subsection{Main objectives and the problem setting}

The main purpose of this paper is to develop a general strategy to compress a class of self-adjoint higher-order elliptic operators by localized basis functions that give optimal approximation property of the solution operator. To be more specific, suppose $\CalL$ is a self-adjoint elliptic operator in the divergence form 
\begin{equation}\label{elliptic}
	\CalL u = \sum_{0 \le |\sigma|, |\gamma| \le k} (-1)^{|\sigma|} D^{\sigma} (a_{\sigma \gamma}(x) D^{\gamma} u),
\end{equation}
where the coefficients $a_{\sigma \gamma}\in L^{\infty}(D) $, $D$ is a bounded domain in $\R^d$, $\sigma = (\sigma_1, \ldots, \sigma_d)$ is a $d$-dimensional multiindex. We ask the question: given an integer $n$, what is the best rank-$n$ compression of the operator $\CalL$ with localized basis functions? This question arises in many different contexts. 

Consider the elliptic equation with the homogeneous Dirichlet boundary conditions
\begin{equation}\label{intro:ellipticequation}
	\CalL u = f, \quad u \in H_0^k(D),
\end{equation}
where the load $f \in L^2(D)$. For a self-adjoint, positive definite operator $\CalL$, Eqn.~\eqref{intro:ellipticequation} has a unique weak solution, denoted as $\CalL^{-1} f$. 
%Under the framework of the Generalized Finite Element Method (GFEM)~\cite{babuvska1983generalized, hou_multiscale_1997, strouboulis2001generalized, efendiev_generalized_2013}, given $n$ basis functions $\mathit{\Psi} =[\psi_1, \ldots, \psi_n]$ in $H_0^k(D)$, we can solve the linear system~\eqref{intro:ellipticequation} approximately by projecting it onto the subspace $\Psi$ spanned by $\mathit{\Psi}$, i.e., $u_n := \mathit{\Psi} L_n^{-1} \mathit{\Psi}^T f$. Here, $\mathit{\Psi}^T f := [(\psi_1, f), \ldots, (\psi_n, f)]^T$ and $L_n $ is the positive definite stiffness matrix. The optimal basis $\mathit{\Psi}$ should achieve small error for all $f\in L^2(D)$, i.e., should minimize the following quantity
%\begin{equation}\label{intro:ellipticerror}
%	\|\CalL^{-1} - \mathit{\Psi} L_n^{-1} \mathit{\Psi}^T \|_2 \equiv \inf_{f\in L^2(D),~f\neq \vct{0}}\frac{\|\CalL^{-1} f - \mathit{\Psi} L_n^{-1} \mathit{\Psi}^T f\|_2}{\|f\|_2}.
%\end{equation}
%
We define the operator compression error of the basis $\mathit{\Psi}$ as follows:
\begin{equation}\label{intro:compresserror}
	E_{\mathrm{oc}}(\mathit{\Psi}; \CalL^{-1}) := \min_{K_n \in \R^{n\times n},~ K_n \succeq 0} \|\CalL^{-1} - \mathit{\Psi} K_n \mathit{\Psi}^T\|_2,
\end{equation}
which is the optimal approximation error of $\CalL^{-1}$ among all positive semidefinite operators with range space spanned by $\mathit{\Psi}$. Using $E_{\mathrm{oc}}(\mathit{\Psi}; (\CalL + \lambda_G)^{-1})$ for some $\lambda_G > 0$ to quantify the compression error is useful for operators that are not invertible, such as $-\Delta$ with periodic boundary conditions.

Without imposing the sparsity constraints on the basis $\mathit{\Psi}$, the compression error $E_{\mathrm{oc}}(\mathit{\Psi}; \CalL^{-1})$ achieves its minimum $\lambda_{n+1}(\CalL^{-1})$ if we use the first $n$ eigenfunctions of $\CalL^{-1}$ to form $\mathit{\Psi}$ ($\lambda_n$ is the $n$th eigenvalue arranged in a descending order). However, the eigenfunctions are expensive to compute and do not have localized support \cite{Zou_PCA_06, ozolins_compressed_2013, hou_LocalModes_2014}. In many cases, localized/sparse basis functions are preferred. For example, in the multiscale finite element method \cite{Houbook09}, localized basis functions lead to sparse linear systems, and thus result in more efficient algorithms; see, e.g.,~\cite{babuvska1983generalized, hou_multiscale_1997, strouboulis2001generalized, hou_PetrovGalerkin_2004, babuska_optimal_2011, efendiev2011multiscale, efendiev_generalized_2013, maalqvist2014localization, owhadi_polyharmonic_2014, owhadi2015multi,efendiev16jcp}. In quantum chemistry, localized basis functions like the Wannier functions have better interpretability of the local interactions between particles (see, e.g.,~\cite{marzari1997maximally, weinan2010localized, marzari2012maximally, ozolins_compressed_2013, LaiLuOsher_2014}), and also lead to more efficient algorithms~\cite{goedecker1999linear}. In statistics, the sparse principal component analysis (SPCA) looks for sparse vectors to span the eigenspace of the covariance matrix, which leads to better interpretability compared with the PCA; see, e.g.,~\cite{Jolliffe_2003, Zou_PCA_06, dAspremont_sparsePCA, witten2009penalized, vu2013fantope}.

\subsection{Summary of our main results}

In this paper, we study operator compression for higher-order elliptic operators. We assume that the self-adjoint elliptic operator $\CalL$ is coercive, bounded and strongly elliptic (to be made precise in Section 6.2). 
%These assumptions are very general, and cover nearly all elliptic operators of interest, such as all the second-order uniformly elliptic operators and all the higher-order Laplacian operator $(-\Delta)^k$ ($k\ge1$). The coercivity and boundedness are standard assumptions, which guarantee the existence of $\CalL^{-1}$. The strong ellipticity (see Definition~3.1 in Part II) is slightly stronger than the standard uniform ellipticity. We show that uniform ellipticity and strong ellipticity are equivalent in two cases: (1) $k=1$, (2) $d=1,2$. For small physical dimensions $d$ and differential orders $k$, strongly elliptic operators approximate uniformly elliptic operators well, and counter examples are difficult to construct. 
Under these assumptions, we construct $n$ basis functions $\mathit{\Psi}^{\mathrm{loc}} = [\psi_1^{\mathrm{loc}}, \ldots, \psi_n^{\mathrm{loc}}]$ that achieve nearly optimal performance on both ends in the accuracy--sparsity trade-off~\eqref{intro:OCL1}. 
\begin{enumerate}
\item[1.] They are optimally localized up to a logarithmic factor, i.e., 
\begin{equation}\label{intro:localization}
	\left|\text{supp}(\psi_i^{\mathrm{loc}})\right| \le \frac{C_l \log(n)}{n} \quad \forall 1 \le i \le n.
\end{equation} 
Here, $|\text{supp}(\psi_i^{\mathrm{loc}})|$ denotes the area/volume of the support of the localized function $\psi_i^{\mathrm{loc}}$ in $\R^d$, and the constant $C_l$ is independent of $n$.
\item[2.] If we use a generalized finite element method \cite{babuvska1983generalized, hou_multiscale_1997, strouboulis2001generalized, efendiev_generalized_2013} to solve the elliptic equations, we achieve the optimal convergence rate in the energy norm, i.e.,
\begin{equation}\label{intro:approximationerrorH}
	\|\CalL^{-1} f - \mathit{\Psi}^{\mathrm{loc}} L_n^{-1} (\mathit{\Psi}^{\mathrm{loc}})^T f \|_H \le C_e \sqrt{\lambda_{n}(\CalL^{-1})} \|f\|_2 \quad \forall f \in L^2(D),
\end{equation}
where $L_n$ is the stiffness matrix under the basis $\mathit{\Psi}^{\mathrm{loc}}$, $\|\cdot\|_H$ is the associated energy norm, and $C_e$ is independent of $n$. 
\item[3.] For the sparse operator compression problem, we achieve the optimal approximation error up to a constant, i.e.,
\begin{equation}\label{intro:approximationerrorL2}
	E_{\mathrm{oc}}(\mathit{\Psi}^{\mathrm{loc}}; \CalL^{-1}) \le C_e^2 \lambda_n(\CalL^{-1}),
\end{equation} 
where $E_{\mathrm{oc}}(\mathit{\Psi}^{\mathrm{loc}}; \CalL^{-1})$ is the operator compression error defined in Eqn.~\eqref{intro:compresserror}.
\end{enumerate}
We will focus on the theoretical analysis of the approximation accuracy~\eqref{intro:approximationerrorH} and the localization of the basis functions~\eqref{intro:localization}.

\subsection{Our construction}
\label{intro:ourmethods}
To construct such localized basis functions $\mathit{\Psi}^{\mathrm{loc}} = [\psi_1^{\mathrm{loc}}, \ldots, \psi_n^{\mathrm{loc}}]$, we first partition the physical domain $D$ using a regular partition $\{\tau_i\}_{i=1}^m$ with mesh size $h$. We pick $\{\phi_{i,q}\}_{q=1}^Q$ to be a set of orthogonal basis functions of $\CalP_{k-1}(\tau_i)$, which is the space of all $d$-variate polynomials of degree at most $k-1$ on the patch $\tau_i \subset D$ and $Q = \binom{k+d-1}{d}$ is the dimension of the space $\CalP_{k-1}(\tau_i)$. For $r > 0$, let $S_r$ be the union of the subdomains $\tau_j$ that intersect with $B(x_i, r)$ (for some $x_i \in \tau_i$) and let $\psi_{i,q}^{\mathrm{loc}}$ be the minimizer of the following quadratic problem:
\begin{equation}\label{intro:localVar}
\begin{split}
	\psi_{i,q}^{\mathrm{loc}} = \argmin_{\psi \in H} \quad & \|\psi\|_H^2 \\
	\text{s.t.} \quad & \int_{S_r} \psi \phi_{j,q'} = \delta_{iq,jq'}, \quad \forall 1 \le j \le m,\, 1 \le q' \le Q, \\
				& \psi(x) \equiv 0, \quad x \in D\backslash S_r.
\end{split}
\end{equation}
Here, the space $H = \{\CalL^{-1} f : f\in L^2(D)\}$ is the solution space of the operator $\CalL$, and $\|\cdot\|_H$ is the energy norm associated with $\CalL$ and the prescribed boundary condition. It is important to point out that the boundary condition of the elliptic problem is already incorporated in the above optimization problem through the solution space $H$ and the definition of the energy norm $\| \cdot\|_H$. This variational formulation is very general and can take into account lower-order terms very easily.

Collecting all the $\psi_{i,q}^{\mathrm{loc}}$ for $1\le i\le m$ and $1\le q\le Q$ together, we get our basis $\mathit{\Psi}^{\mathrm{loc}}$. We will prove that for $r = \mathcal{O}( h \log(1/h))$, 
\begin{enumerate}
\item they achieve the optimal convergence rate to solve the elliptic equation, i.e.,
\begin{equation}\label{intro:localizedMsFEM}
	\|\CalL^{-1} f - \mathit{\Psi}^{\mathrm{loc}} L_n^{-1} (\mathit{\Psi}^{\mathrm{loc}})^T f \|_H \le C_e h^k \|f\|_2 \quad \forall f \in L^2(D),
\end{equation}
where the constant $C_e$ is independent of $n$.
\item they achieve the optimal approximation error to approximate the elliptic operator, i.e.,
\begin{equation}\label{intro:approximationerrorL2local}
	E_{\mathrm{oc}}(\mathit{\Psi}^{\mathrm{loc}}; \CalL^{-1}) \le C_e^2 h^{2k}.
\end{equation} 
\end{enumerate}
For $n = m Q$, we can show that the $n$th largest eigenvalue of $\CalL^{-1}$ is of the order $h^{2k}$, i.e., $\lambda_n(\CalL^{-1}) = \mathcal{O}(h^{2k})$. Therefore, the optimality above is exactly the optimality described in Eqn.~\eqref{intro:approximationerrorH} and \eqref{intro:approximationerrorL2}.

\subsection{Comparison with other existing methods} 
\revise{
Our approach for operator compression originates at the MsFEM and numerical homogenization, where localized multiscale basis functions are constructed to approximate the solution space of some elliptic PDEs with multiscale coefficients; see~\cite{babuvska1983generalized, hou_multiscale_1997, strouboulis2001generalized, Houbook09, babuska_optimal_2011, efendiev_generalized_2013, maalqvist2014localization, owhadi_polyharmonic_2014, owhadi2015bayesian, owhadi2015multi, efendiev16jcp}. Specifically, our work is inspired by the work presented in~\cite{maalqvist2014localization,owhadi2015multi}, in which multiscale basis functions with support size $O(h\log(1/h))$ are constructed for second-order elliptic equations with rough coefficients and homogeneous Dirichlet boundary conditions. In this paper, we generalize the construction~\cite{owhadi2015multi} and propose a general framework to compress higher-order elliptic operators with optimal compression accuracy and optimal localization. 
}

\revise{
We remark that although we use the framework presented in \cite{owhadi2015multi} as the direct template for our method, to the best of our knowledge, the local orthogonal decomposition (LOD)~\cite{maalqvist2014localization}, in the context of multidimensional numerical homogenization, contains the first rigorous proof of optimal exponential decay rates with a priori estimates (leading to localization to subdomains of size $h \log(1/h)$, with basis functions derived from the Clement interpolation operator). The idea of using the preimage of some continuous or discontinuous finite element space under the partial differential operator to construct localized basis functions in Galerkin-type methods was even used earlier, e.g., in \cite{Grasedyck_ALbasis_2012}, although it did not provide a constructive local basis. In addition to establishing the exponential decay of the basis (for general nonconforming measurements of the solution, we will generalize the proof of this result to higher-order PDEs and measurements formed by local polynomials), a major contribution of \cite{owhadi2015multi} was to introduce a multiresolution operator decomposition for second-order elliptic PDEs with rough coefficients.}

\revise{There are several new ingredients in our analysis that are essential for us to obtain our results for higher-order elliptic operators with rough coefficients. First of all,  we prove an inverse energy estimate for functions in $\Psi$, which is crucial in proving the exponential decay. In particular, Lemma~\ref{lem:scaling} is an essential step to obtaining the inverse energy estimate for higher-order PDEs that is not found in \cite{maalqvist2014localization} nor \cite{owhadi2015multi}. We remark that Lemma 3.12 in~\cite{owhadi2015multi} provides such an estimate for second-order elliptic operators, by utilizing a relation between the Laplacian operator $\Delta$ and the $d$-dimensional Brownian motion. It is not straightforward to extend this probabilistic argument to higher-order cases. In contrast, our inverse energy estimate is valid for any $2k$th-order elliptic operators and is tighter than the estimation in~\cite{owhadi2015multi} for the second-order case. Secondly, we prove a projection-type polynomial approximation property in $H^k(D)$. This polynomial approximation property plays an essential role in both estimating the compression accuracy and in localizing the basis functions. Thirdly, we propose the notion of the strong ellipticity to analyze the higher-order elliptic operators and show that strong ellipticity is only slightly stronger than the standard uniform ellipticity. Very recently, the authors of ~\cite{owhadi2017universal} introduce the Gaussian cylinder measure and successfully generalize the probabilistic framework in~\cite{owhadi2015bayesian, owhadi2015multi} to a much broader class of operators, including higher-order elliptic operators without requiring the strong ellipticity.}

\revise{
As in \cite{maalqvist2014localization,owhadi2015multi}, the error bound in our convergence analysis blows up for fixed oversampling ratio $r/h$. To achieve the desired $O(h^k)$ accuracy in the energy norm, we require $r/h=O(\log(1/h))$. There has been some previous attempt to study the convergence of MsFEM using oversampling techniques with $r/h$ being fixed, see, e.g., \cite{henning2013_oversample,Peterseim2016}. In particular, the authors of \cite{henning2013_oversample,Peterseim2016} showed that if the oversampling ratio $r/h$ is fixed, the accuracy of the numerical solution will depend on the regularity of the solution and cannot be guaranteed for problems with rough coefficients. By imposing $r/h=O(\log(1/h))$, the authors of \cite{henning2013_oversample,Peterseim2016} proved that the the MsFEM with constrained oversampling converges with the desired accuracy $O(h)$. 
}

\revise{
There has been some previous work for second-order elliptic PDEs by using basis functions of support size $O(h)$, see, e.g., \cite{babuska_optimal_2011,houliu2016_DCDS}. However, they need to use $O(\log(1/h))$ basis functions associated with each coarse finite element to recover the $O(h)$ accuracy. The computational complexity of this approach is comparable to the one that we present in this paper. It is worth mentioning that the authors of \cite{houliu2016_DCDS} use a local oversampling operator to construct the optimal local boundary conditions for the nodal multiscale basis and enrich the nodal multiscale basis with optimal edge multiscale basis. Moreover, the method in \cite{houliu2016_DCDS} allows an explicit control of the approximation accuracy in the offline stage by truncating the SVD of the oversampling operator. In \cite{houliu2016_DCDS}, the authors demonstrated numerically that this method is robust to high-contrast problems and the number of basis functions per coarse element is typically small. We remark that the recently developed generalized multiscale finite element method (GMsFEM) \cite{efendiev_generalized_2013,efendiev16jcp} has provided another promising approach in constructing multiscale basis functions with support size $O(h)$.
}

Another popular way to formulate the operator compression problem is to solve the following $l^1$ penalized variational problem:
\begin{equation}\label{intro:OCL1}
\begin{split}
	\min_{\mathit{\Psi}} \quad& \sum_{i=1}^n \|\psi_i\|_{H}^2 + \lambda \sum_{i=1}^n \|\psi_i\|_1,\\
	\text{s.t.} \quad& (\psi_i, \psi_j) = \delta_{i,j} \quad \forall 1\le i,j \le n,
\end{split}
\end{equation}
where $\|\psi_i\|_{H}$ is the energy norm induced by the operator $\CalL$. In problem~\eqref{intro:OCL1}, enforcing $\|\psi_i\|_{H}$ to be small leads to a small compression error, enforcing $\|\psi_i\|_1$ to be small leads to a sparse basis function, and $\lambda>0$ is a parameter to control the trade-off between the accuracy and sparsity. 

The sparse PCA (SPCA) is closely related to the above $l^1$-based optimization problem. Given a covariance function $K(x,y)$, the SPCA solves a variational problem similar to Eqn.~\eqref{intro:OCL1}:
\begin{equation}\label{intro:SPCA}
\begin{split}
	\min_{\mathit{\Psi}} \quad& -\sum_{i=1}^n (\psi_i, \CalK \psi_i) + \lambda \sum_{i=1}^n \|\psi_i\|_1,\\
	\text{s.t.} \quad& (\psi_i, \psi_j) = \delta_{i,j} \quad \forall 1\le i,j \le n,
\end{split}
\end{equation}
where $(\psi_i, \CalK \psi_i) := \int_D \int_D K(x,y) \psi_i(x) \psi_i(y)\rd x\, \rd y$. In the SPCA~\eqref{intro:SPCA}, we have the minus sign in front the variational term because we are interested in the eigenspace corresponding to the largest $n$ eigenvalues. 
%Although Eqn.~\eqref{intro:OCL1} and Eqn.~\eqref{intro:SPCA} have similarities in both formulation and algorithms, they are very different in nature. The most significant difference is that the input of the $l^1$ operator compression~\eqref{intro:OCL1} is an \textit{unbounded} elliptic operator $\CalL$, while the input of the SPCA is the \textit{bounded} covariance operator $\CalK$, which may be viewed as $\CalL^{-1}$ in some cases. 
Although the $l^1$ approach performs well in practice, neither Problem~\eqref{intro:OCL1} nor the SPCA~\eqref{intro:SPCA} is convex, and one needs to use some sophisticated techniques to solve the non-convex optimization problem or its convex relaxation; see, e.g.,~\cite{Zou_PCA_06, dAspremont_sparsePCA, ozolins_compressed_2013, vu2013fantope, LaiLuOsher_2014}. 

%There is a trade-off between the compression accuracy and the sparsity of the basis functions. In general, there is no localized basis that can achieve the optimal compression accuracy $\lambda_{n+1}(\CalL^{-1})$. In other words, the principal eigenspace cannot be spanned by a localized basis. 

%We would like to point out some similarities and differences between our approach and the $l^1$-penalty approach. There are several important differences between our approach and the $l^1$-based optimization approach. First of all, instead of adding a $l^1$ penalty term, we obtain sparsity by enforcing $\psi_{i,q}$ to have zero moments up to the $(k-1)$-th order ($2k$ is the order of ${\CalL}$) on every patch except for $\tau_i$, as required in Eqn.~\eqref{intro:localVar}.\footnote{Recall that for any $1 \le j \le m$, $\{\phi_{j,q}\}_{q=1}^Q$ spans the space of polynomials with order no more than $k-1$ on patch $\tau_j$.} 
In comparison with the $l^1$-based optimization method or the SPCA, our approach has the advantage that this construction will guarantee that $\psi_{i,q}$ decays exponentially fast away from $\tau_i$. This exponential decay justifies the local construction of the basis functions in Eqn.~\eqref{intro:localVar}. Moroever, our construction~\eqref{intro:localVar} is a quadratic optimization with linear constraints, which can be solved as efficiently as solving an elliptic problem on the local domain $S_r$. \revise{The computational complexity to obtain all $n$ localized basis functions $\{\psi_i^{\mathrm{loc}}\}_{i=1}^n$ is only of order $N\log^{3d}(N)$ if a multilevel construction is employed, where $N$ is the degree of freedom in the discretization of $\CalL$; see~\cite{owhadi2015multi}.} In contrast, the orthogonality constraint in Eqn.~\eqref{intro:OCL1} is not convex, which introduces additional difficulties in solving the problem. 
%For example, the algorithm of splitting orthogonality constraint (SOC)~\cite{lai2014splitting, ozolins_compressed_2013} solves the $l^1$ operator compression~\eqref{intro:OCL1} in an iterative manner, where the cost of \textit{each iteration} is roughly the same as the total cost of our approach. Thirdly, 
Finally, our construction of $\{\psi_i^{\mathrm{loc}}\}_{i=1}^n$ is completely decoupled, while all the basis functions in Eqn.~\eqref{intro:OCL1} are coupled together. This decoupling leads to a simple parallel execution, and thus makes the computation of $\{\psi_i^{\mathrm{loc}}\}_{i=1}^n$ even more efficient.

The rest of the paper is organized as follows. 
In Section 2, we introduce the abstract framework of the sparse operator compression. In Section~\ref{sec:generalizedPoincare}, we prove a projection-type polynomial approximation property for the Sobolev spaces, which can be seen as a generalization of the Poincare inequality for functions with higher regularity. This polynomial approximation property is critical in our analysis of the higher-order case. It plays a role similar to that of the Poincare inequality in the analysis of the second-order elliptic operator. In Section 4, we prove the inverse energy estimate by scaling. In Section 5, we use the second-order elliptic PDE to illustrate the main idea of our analysis. In Section~\ref{sec:k2case}, we first introduce the notion of strong ellipticity, and then prove the exponential decay of the constructed basis function for strongly elliptic operators. In Section~\ref{sec:localization}, we localize the basis functions, and provide the convergence rate for the corresponding MsFEM and the compression rate for the corresponding operator compression. Finally, we present several numerical results to support the theoretical findings in Section~\ref{sec:numericalExps}. Some concluding remarks are made in Section \ref{k1sec:conclusions} and a few technical proofs are deferred to the Appendix.

\section{Operator compression}
\label{sec:functionalanalysis}
In this section, we provide an abstract and general framework to compress a bounded self-adjoint positive semidefinite operator $\CalK: X \to X$, where $X$ can be any separable Hilbert space with inner product $(\cdot, \cdot)$. In the case of operator compression of an elliptic operator $\CalL$, $\CalK$ plays the role of the solution operator $\CalL^{-1}$ and $X = L^2(D)$. In the case of the SPCA, $\CalK$ plays the role of the covariance operator. 
%A probabilistic framework for the Bayesian numerical homogenization has been proposed in~\cite{owhadi2015bayesian}. However,  it is not clear that there exists a Gaussian measure for a given covariance operator. Our following framework is purely based on functional analysis, which applies to any bounded self-adjoint positive semidefinite operators.
In Section~\ref{subsec:CameronMartin}, we introduce the Cameron--Martin space, which plays the role of the solution space of $\CalL$. In Section~\ref{subsec:OCerroranalysis}, we provide our main theorem to estimate the compression error. 
%In Section~\ref{subsec:energyminimization}, we give the abstract form to construct basis functions by minimizing their energies. In Section~\ref{subsec:localminimization}, we give the abstract form to construct localized basis functions. 
We will use this abstract framework to compress elliptic operators in the rest of the paper. 

\subsection{The Cameron--Martin space}
\label{subsec:CameronMartin}
Suppose $\{(\lambda_n, e_n)\}_{n=1}^{\infty}$ are the eigen pairs of the operator $\CalK$ with the eigenvalues $\{\lambda_n\}_{n=1}^{\infty}$ in a descending order. We have $\lambda_n\ge 0$ for all $n$ since $\CalK$ is self-adjoint and positive semidefinite. From the spectral theorem of a self-adjoint operator, we know that $\{e_n)\}_{n=1}^{\infty}$ forms an orthonormal basis of $X$.

\begin{lemma}\label{lem:defnorm}
Let $\CalK(X)$ be the range space of $\CalK$. We have
\begin{enumerate}
\item[1.] $\CalK(X)$ is an inner product space with inner product defined by
\begin{equation}\label{def:Hinner}
	( \CalK \phi_1, \CalK \phi_2 )_H = (\CalK \phi_1, \phi_2) \qquad \forall \phi_1, \phi_2 \in X.
\end{equation}
\item[2.] $\CalK(X)$ is continuously imbedded in $X$. 
\item[3.] $\CalK(X)$ is dense in $X$ if the null space of $\CalK$ only contains the origin, i.e., $\mathrm{null}(\CalK) = \{\zerovct\}$.
\end{enumerate}
\end{lemma}
\begin{proof}
\begin{enumerate}
\item[1.]
Since $\CalK$ is self-adjoint, we have $( \CalK \phi_1, \CalK \phi_2 )_H = ( \CalK \phi_2, \CalK \phi_1 )_H$. The linearity and nonnegativity are obvious. Finally, if $( \CalK \phi, \CalK \phi )_H = 0$ for some $\phi \in X$, then $(\CalK \phi, \phi) = 0$. Suppose that $\phi = \sum_n \alpha_n e_n$ by expanding $\phi$ with eigenvectors of $\CalK$. Then, we have $(\CalK \phi, \phi) = \sum_n \lambda_n \alpha_n^2 = 0$. Therefore, $\alpha_n = 0$ for all $\lambda_n > 0$. Equivalently, we obtain $\phi \in \mathrm{null}(\CalK)$, i.e., $\CalK \phi = 0$.
\item[2.]
Since $\lambda_n^2 \le \lambda_1 \lambda_n$ for all $n\in \N$, we have $\CalK^2 \preceq \lambda_1 \CalK$. Then, we obtain
\begin{equation}\label{eqn:embedding0}
	\sqrt{(\CalK \phi, \CalK \phi)} \le \sqrt{ \lambda_1  (\CalK \phi, \phi)} = \sqrt{ \lambda_1} \sqrt{ ( \CalK \phi, \CalK \phi )_H},
\end{equation}
where we have used the definition of $( \cdot, \cdot )_H$ in Eqn.~\eqref{def:Hinner} in the last step.
\item[3.]
If $\mathrm{null}(\CalK) = \{\zerovct\}$, we have $\text{span}\{e_n, n\ge 1\} \subset \CalK(X)$. Then, $\CalK(X)$ is dense in $X$.
\end{enumerate}
\end{proof}
We define the Cameron--Martin space $H$ as the completion of $\CalK(X)$ with respect to the norm $\sqrt{( \cdot, \cdot )_H}$. Then, $H$ is a separable Hilbert space and  we have the following lemma. 
\begin{lemma}\label{lem:innerproduct}
\begin{enumerate}
\item[1.] $H$ can be continuously embedded into $X$.
\item[2.] $H$ is dense in $X$ if $\mathrm{null}(\CalK) = \{\zerovct\}$.
\item[3.] For all $\psi \in X$ and all $f \in H$, we have
\begin{equation} \label{eqn:Hinner}
	(f, \CalK \psi)_H = (f, \psi).
\end{equation}
\end{enumerate}
\end{lemma}
\begin{proof}
\begin{enumerate}
\item[1.] By the continuous imbedding from $\CalK(X)$ to $X$, we know that a Cauchy sequence in $\CalK(X)$ is also a Cauchy sequence in $X$. Therefore, we have $H \subset X$. By Eqn.~\eqref{eqn:embedding0} and the the continuity of norms, we have $(\psi, \psi) \le \lambda_1 (\psi, \psi)_H$ for any $\psi \in H$.
\item[2.] It is obvious from item 3 in Lemma~\ref{lem:defnorm}.
\item[3.] If $f \in \CalK(X)$, Eqn.~\eqref{eqn:Hinner} is exactly the definition of $(\cdot, \cdot)_H$ in Eqn.~\eqref{def:Hinner}. By the continuity of the inner product, Eqn.~\eqref{eqn:Hinner} is true for any $f \in H$.
\end{enumerate}
\end{proof}

\subsection{Operator compression}
\label{subsec:OCerroranalysis}
Suppose $H$ is an arbitrary separable Hilbert space and $\Phi \subset H$ is $n$-dimensional subspace in $H$ with basis $\{\phi_i\}_{i=1}^n$. In the rest of the paper, $\CalP_{\Phi}^{(H)}$ denotes the orthogonal projection from a Hilbert space $H$ to its subspace $\Phi$. With this notation, we present our theorem for error estimates below.
\begin{theorem}\label{thm:conditioning1}
Suppose there is a $n$-dimensional subspace $\Phi \subset X$ with basis $\{\phi_i\}_{i=1}^n$ such that
\begin{equation}\label{eqn:poincare1}
	\| u - \CalP_{\Phi}^{(X)} u \|_X \le k_n \|u\|_{H} \qquad \forall u \in \CalK(X) \subset H.
\end{equation}
Let $\Psi$ be the $n$-dimensional subspace in $H$ (also in $X$) spanned by $\{\CalK \phi_i\}_{i=1}^n$. Then
\begin{enumerate}
\item[1.] For any $u \in \CalK(X)$ and $u = \CalK f$, we have
\begin{equation}\label{eqn:GFEMH}
	\| u - \CalP_{\Psi}^{(H)} u \|_{H} \le k_n \|f\|_X\,.
\end{equation}
\item[2.] For any $u \in \CalK(X)$ and $u = \CalK f$, we have
\begin{equation}\label{eqn:GFEML2}
	\| u - \CalP_{\Psi}^{(H)} u \|_{X} \le k_n^2 \|f\|_X\,.
\end{equation}
\item[3.] We have
\begin{equation}
	\| \CalK - \CalP_{\Psi}^{(H)} \CalK \| \le k_n^2\,,
\end{equation}
where $\|\cdot\|$ is the induced operator norm on $\mathcal{B}(X,X)$. Moreover, the rank-$n$ operator $\CalP_{\Psi}^{(H)} \CalK : X \to X $ is self-adjoint.
\end{enumerate}
\end{theorem}

In Theorem~\ref{thm:conditioning1}, by using a projection-type approximation property of $\Phi$ in $H$, i.e., Eqn.~\eqref{eqn:poincare1}, we obtain the error estimates of the multiscale finite element method with finite element basis $\{\CalK \phi_i\}_{i=1}^n$ in the energy norm, i.e., Eqn.~\eqref{eqn:GFEMH}. We will take $\Phi$ as the discontinuous piecewise polynomial space later, which is a poor finite element space for elliptic equations with rough coefficients. However, after smoothing $\Phi$ with the solution operator $\CalK$, the smoothed basis functions $\{\CalK \phi_i\}_{i=1}^n$ have the optimal convergence rate. This data-dependent methodology to construct finite element spaces was pioneered by the generalized finite element (GFEM)~\cite{babuvska1983generalized, strouboulis2001generalized}, the multiscale finite element method (MsFEM)~\cite{hou_multiscale_1997, hughes1998variational,Houbook09}, and numerical homogenization \cite{maalqvist2014localization,owhadi2015multi}.

Our error analysis is different from the traditional finite element error analysis in two aspects. First of all, the traditional error analysis relies on an interpolation type approximation property where higher regularity is required. For example, the error analysis for the FEM with standard linear nodal basis functions for the Poisson equation requires the following interpolation type approximation:
\begin{equation}\label{eqn:traditionalFEM}
	|u - \mathcal{I}_h u|_{1,2,D} \le C h |u|_{2,2,D} \quad \forall u \in H_0^2(D),
\end{equation}
where $\mathcal{I}_h u$ is the piecewise linear interpolation of the solution $u$. In Eqn.~\eqref{eqn:traditionalFEM}, one assumes $u \in H^2(D)$, but this is not the case for elliptic operators with rough coefficients. Secondly, in our projection-type approximation property~\eqref{eqn:poincare1} the error is measured by the ``weaker'' $\|\cdot\|_X$ norm, while in the traditional interpolation type approximation property the error is measured by the ``stronger'' $\|\cdot\|_H$ norm. In this sense, our error estimate relies on weaker assumptions. \revise{As far as we know, this kind of error estimate was first introduced in Proposition 3.6 in~\cite{owhadi2015multi}.}

\begin{proof}\textbf{[Proof of Theorem~\ref{thm:conditioning1}]}
\begin{enumerate}
\item[1.] For an arbitrary $v \in \Psi$, due to the definition of $\Psi$, we can write $v = \CalK( \sum_{i=1}^n c_i \phi_i )$, and thus we get $u - v = \CalK( f - \sum_{i=1}^n c_i \phi_i )$.
By Lemma~\ref{lem:innerproduct}, we have
\begin{equation*}
\begin{split}
	& \|u - v\|_{H}^2 = \left(u - v, f - \sum_{i=1}^n c_i \phi_i\right) \\
	&= \left(u - v - \CalP_{\Phi}^{(X)} (u-v), f - \sum_{i=1}^n c_i \phi_i\right) + \left(\CalP_{\Phi}^{(X)} (u-v), f - \sum_{i=1}^n c_i \phi_i\right).
\end{split}
\end{equation*}
By choosing $c_i$ such that $\sum_{i=1}^n c_i \phi_i = \CalP_{\Phi}^{(X)} (f)$, the second term vanishes. Then, we obtain
\begin{equation*}
\begin{split}
	&\|u - v\|_{H}^2 = \left(u - v - \CalP_{\Phi}^{(X)} (u-v), f - \sum_{i=1}^n c_i \phi_i\right) \\
	&\le \|u - v - \CalP_{\Phi}^{(X)} (u-v)\|_X \|f - \CalP_{\Phi}^{(X)} (f)\|_X \le k_n \|u - v\|_{H} \|f\|_X
\end{split}
\end{equation*}
Therefore, we conclude $\|u - v\|_{H} \le k_n \|f\|_X$.
\item[2.] We use the Aubin--Nistche duality argument to get the estimation in item 2. Let $v = \CalK( u - \CalP_{\Psi}^{(H)} u )$. On one hand, we get
\begin{equation*}
(u - \CalP_{\Psi}^{(H)} u, v - \CalP_{\Psi}^{(H)} v)_{H} = (u - \CalP_{\Psi}^{(H)} u, v)_{H} = (u - \CalP_{\Psi}^{(H)} u, u - \CalP_{\Psi}^{(H)} u)_{X} = \|u - \CalP_{\Psi}^{(H)} u\|_{X}^2.
\end{equation*}
On the other hand, we obtain
%\begin{equation*}
%\begin{split}
%(u - \CalP_{\Psi}^{(H)} u, v - \CalP_{\Psi}^{(H)} v)_{H} & \le \|u - \CalP_{\Psi}^{(H)} u\|_{H} \|v - \CalP_{\Psi}^{(H)} v\|_{H} \\
%&\le k_n \|f\|_X \, k_n \|u - \CalP_{\Psi}^{(H)} u\|_X.
%\end{split}
%\end{equation*}
\begin{equation*}
(u - \CalP_{\Psi}^{(H)} u, v - \CalP_{\Psi}^{(H)} v)_{H} \le \|u - \CalP_{\Psi}^{(H)} u\|_{H} \|v - \CalP_{\Psi}^{(H)} v\|_{H} \le k_n \|f\|_X \, k_n \|u - \CalP_{\Psi}^{(H)} u\|_X.
\end{equation*}
We have used the result of item 1 in the last step. Combining these two estimates, the result follows.
\item[3.] From the last item, we obtain that $\| \CalK f - \CalP_{\Psi}^{(H)} \CalK f \|_X \le k_n^2 \|f\|_X$ for any $f \in X$. Therefore, we conclude $\| \CalK - \CalP_{\Psi}^{(H)} \CalK \| \le k_n^2$. Now, we prove that $\CalP_{\Psi}^{(H)} \CalK$ is self-adjoint. For any $x_1, x_2 \in X$, by definition of $H$-norm we have 
$$(x_1, \CalP_{\Psi}^{(H)} \CalK x_2) = (\CalK x_1, \CalP_{\Psi}^{(H)} \CalK x_2)_H.$$
Since $\CalP_{\Psi}^{(H)}$ is self-adjoint in $H$, we have
$$ (\CalK x_1, \CalP_{\Psi}^{(H)} \CalK x_2)_H = (\CalP_{\Psi}^{(H)} \CalK x_1, \CalK x_2)_H = (\CalP_{\Psi}^{(H)} \CalK x_1, x_2),$$
where we have used the definition of $H$-norm again in the last step.
\end{enumerate}
\end{proof}

%In Theorem~\ref{thm:conditioning1}, based on a projection-type approximation property of $\Phi$ in $H$, i.e., Eqn.~\eqref{eqn:poincare1}, we obtain the error estimates of the MsFEM with multiscale finite element basis $\{\CalK \phi_i\}_{i=1}^n$ in the energy norm, i.e., Eqn.~\eqref{eqn:GFEMH}. We will take $\Phi$ as the discontinuous piecewise polynomial space later. It is well known that a piecewise polynomial basis gives a poor approximation to the solution of an elliptic equation with rough coefficients. However, after applying the solution operator $\CalK$ to $\Phi$, the resulting basis functions $\{\CalK \phi_i\}_{i=1}^n$ have the optimal convergence rate. This is similar in spirit to the generalized finite element (GFEM)~\cite{babuvska1983generalized, strouboulis2001generalized} and the multiscale finite element method (MsFEM)~\cite{hou_multiscale_1997, hughes1998variational,Houbook09}. 

Although the basis functions $\{\CalK \phi_i\}_{i=1}^n$ have good approximation accuracy, they are typically not localized. Therefore, we construct another set of basis functions $\{\psi_i\}_{i=1}^{n}$ for $\Psi$ via the following variational approach, which results in basis functions with good localization properties. For any given $i \in \{1, 2, \ldots, n\}$, consider the following quadratic optimization problem
\begin{equation}\label{eqn:psivariational}
\begin{split}
	\psi_i = \argmin_{\psi \in H} \quad & \|\psi\|_H^2 \\
	\text{s.t.} \quad & (\psi, \phi_j) = \delta_{i,j}, \quad j = 1,2, \ldots, n.
\end{split}
\end{equation}
Define $\Theta \in \R^{n \times n}$ by
\begin{equation} \label{def:Theta}
	\Theta_{i,j} := (\CalK \phi_i, \phi_j).
\end{equation}
It is easy to verify that $\{\CalK \phi_i\}_{i=1}^n$ are linearly independent if and only if $\Theta$ is invertible. We will write $\Theta^{-1}$ as its inverse and $\Theta^{-1}_{i,j}$ as the $(i,j)$th entry of $\Theta^{-1}$. It is not difficult to prove the following properties of $\psi_i$, which is defined as the unique minimizer of Eqn.~\eqref{eqn:psivariational}.
\begin{theorem}\label{thm:variational}
If $\mathrm{null}(\CalK) \cap \Phi = \{\zerovct\}$ holds true, then we have
\begin{enumerate}
\item[1.] The optimization problem~\eqref{eqn:psivariational} admits a unique minimizer $\psi_i$, which can be written as
\begin{equation}\label{eqn:psi1}
	\psi_i = \sum_{j=1}^n \Theta_{i,j}^{-1} \CalK \phi_j.
\end{equation}
\item[2.] For $w \in \R^n$, $\sum_{i=1}^n w_i\psi_i$ is the minimizer of $\|\psi\|_{H}$ subject to $(\phi_j, \psi) = w_j$ for $j = 1,2, \ldots, n$. Moreover, for any $\psi$ which satisfies $(\phi_j, \psi) = w_j$ for $j = 1,2, \ldots, n$, we have
\begin{equation}\label{eqn:psiortho}
	\|\psi\|_H^2 = \left\|\sum_{i=1}^n w_i\psi_i \right\|_H^2 + \left\|\psi - \sum_{i=1}^n w_i\psi_i \right\|_H^2.
\end{equation}
\item[3.] $(\psi_i, \psi_j)_H = \Theta^{-1}_{i,j}$.
\end{enumerate}
\end{theorem}

With a good choice of the space $\Phi$ and its basis $\{\phi_i\}_{i=1}^n$, the energy-minimizing basis $\psi_i$, defined in Eqn.~\eqref{eqn:psivariational}, enjoys good localization properties. 
%For example, when $\CalK$ is the solution operator of a second-order elliptic operator, we take $\Phi$ as the space of piecewise constant functions and its basis $\{\phi_i\}_{i=1}^n$ as the indicator function of each patch. 
We will prove that the energy-minimizing basis function $\psi_i$ decays exponentially fast away from its associated patch. The localization property justifies the following local construction of the basis functions:
\begin{equation}\label{eqn:psivariational_local}
\begin{split}
	\psi_{i}^{\mathrm{loc}} = \argmin_{\psi \in H} \quad & \|\psi\|_H^2 \\
	\text{s.t.} \quad & (\psi, \phi_j) = \delta_{i,j}, \quad j = 1,2, \ldots, n,\\
				& \psi(x) \equiv 0, \quad x \in D\backslash S_i,
\end{split}
\end{equation}
where $S_i \subset D$ is a neighborhood of the patch that $\psi_i$ is associated with. Compared with Eqn.~\eqref{eqn:psivariational}, the localized basis $\psi_{i}^{\mathrm{loc}}$ is obtained by solving exactly the same quadratic problem but on a local domain $S_i$. 

To compress elliptic operators with order $2k$, we take $\Phi$ as the space of (discontinuous) piecewise polynomials, with degree no more than $k-1$. We take its basis as $\{\phi_{i,q}\}_{i=1, q=1}^{m, Q}$, where $Q := \binom{k+d-1}{d}$ is the dimension of the $d$-variate polynomial space with degree no more than $k-1$ and $\{\phi_{i,q}\}_{q=1}^{Q}$ is an orthonormal basis of the polynomial space on the patch $\tau_i$. Two main theoretical results in this paper are as follows.
\begin{enumerate}
\item[1.] The basis function $\psi_i$ decays exponentially fast away from its associated patch; see Theorem~\ref{thm:expdecay2simple} and Theorem~\ref{thm:expdecay2}.
\item[2.] The localized basis function $\psi_i^{\mathrm{loc}}$ approximates $\psi_i$ accurately; see Theorem~\ref{thm:localization}. Meanwhile, the compression rate $E_{\mathrm{oc}}(\mathit{\Psi}^{\mathrm{loc}}; \CalL^{-1})$ is the same as $E_{\mathrm{oc}}(\mathit{\Psi}; \CalL^{-1})$; see Theorem~\ref{thm:localizedMsFEM} and Corollary~\ref{col:localizedOC}.
\end{enumerate}

\section{A Projection-type Polynomial Approximation Property}
\label{sec:generalizedPoincare}
The following projection-type polynomial approximation property in the Sobolev space $H^k(D)$ plays an essential role in both obtaining the optimal approximation error and proving the exponential decay of the energy-minimizing basis functions. It can be viewed as a generalized Poincare inequality.
\begin{theorem}\label{thm:generalPoincare}
Suppose $\Omega \subset \R^d$ is affine equivalent to $\hat{\Omega}$, i.e., there exists an invertible affine mapping
\begin{equation}\label{def:affineequiv}
	F : \hat{x} \in \hat{\Omega} \to F(\hat{x}) = B \hat{x} + b \in \Omega
\end{equation}
such that $F(\hat{\Omega}) = \Omega$. Let $h$ be the diameter of $\Omega$ and $\delta h$ be the maximum diameter of a ball inscribed in $\Omega$. Let the mapping $\Pi: H^{k+1}(\Omega) \to \CalP_k(\Omega)$ be the projection onto the polynomial space with degree no greater than $k$ in $L^2(\Omega)$. Then, there exists a constant $C(k, \hat{\Omega})$ such that for any $u \in H^{k+1}(\Omega)$ and any $0 \le p \le k+1$
\begin{equation}\label{eqn:generalPoincare0}
	|u - \Pi u|_{p,2,\Omega} \le C(k, \hat{\Omega}) \delta^{-p} h^{k-p+1} |u|_{k+1,2,\Omega}.
\end{equation}
\end{theorem}

To prove Theorem~\ref{thm:generalPoincare}, we use a basic result about the Sobolev spaces, due to J. Deny and J.L. Lions, which pervades the mathematical analysis of the finite element method: over the quotient space $H^{k+1}(D)/\CalP_k(D)$, the seminorm $|\cdot|_{k+1, D}$ is a norm equivalent to the quotient norm. We will use the following theorem (Theorem 3.1.4 in \cite{ciarlet2002finite}), to prove Theorem~\ref{thm:generalPoincare}.
\begin{theorem}\label{thm:Ciarlet3.1.4}
For some integers $k \ge 0$ and $m \ge 0$, let $H^{k+1}(\hat{\Omega})\equiv W^{k+1,2}(\hat{\Omega})$ and $H^m(\hat{\Omega})\equiv W^{m,2}(\hat{\Omega})$ be Sobolev spaces satisfying the inclusion
\begin{equation*}
	H^{k+1}(\hat{\Omega}) \subset H^m(\hat{\Omega}),
\end{equation*}
and let $\hat{\Pi} : H^{k+1}(\hat{\Omega}) \to H^m(\hat{\Omega})$ be a continuous linear mapping such that 
\begin{equation*}
	\hat{\Pi} \hat{p} = \hat{p}, \qquad \forall \hat{p}\in \CalP_k(\hat{\Omega}).
\end{equation*}
For any open set $\Omega$ which is affine equivalent to the set $\hat{\Omega}$ (see Eqn.~\eqref{def:affineequiv}), let the mapping $\Pi_{\Omega}$ be defined by
\begin{equation*}
	\widehat{\Pi_{\Omega} v} = \hat{\Pi} \hat{v},
\end{equation*}
for all functions $\hat{v} \in H^{k+1}(\hat{\Omega})$ and $v \in H^{k+1}(\Omega)$ in the correspondence $(\hat{v}: \hat{\Omega}\to \R) \to (v = \hat{v}\circ F^{-1}: \Omega \to \R)$. Then, there exists a constant $C(\hat{\Pi}, \hat{\Omega})$ such that, for all affine-equivalent sets $\Omega$, 
\begin{equation}\label{eqn:Ciarlet3.1.4}
	|v - \Pi_{\Omega} v|_{m,2,\Omega} \le C(\hat{\Pi}, \hat{\Omega}) \delta^{-m} h^{k-m+1} |v|_{k+1,2,\Omega}, \qquad \forall v \in H^{k+1}(\Omega),
\end{equation}
where $h = \text{diam}(\Omega)$ and $\delta h$ is the diameter of the biggest ball contained in $\Omega$.
\end{theorem}

By specializing the operator $\hat{\Pi}$ to be the projection of $H^{k+1}(\hat{\Omega})$ to the polynomial space $\CalP_k(\hat{\Omega})$ in $L^2(\hat{\Omega})$, we can prove Theorem~\ref{thm:generalPoincare}.
\begin{proof} [\bf{Proof of Theorem~\ref{thm:generalPoincare}}]
Let $\hat{\Pi}: H^{k+1}(\hat{\Omega}) \to \CalP_k(\hat{\Omega})$ be the orthogonal projection in $L^2(\hat{\Omega})$. Let $F:\hat{\Omega} \to \Omega$ be the invertible linear map and write $F(\hat{x}) = B \hat{x} + b$. Define $\Pi_{\Omega}$ as 
\begin{equation*}
	\widehat{\Pi_{\Omega} v} = \hat{\Pi} \hat{v},
\end{equation*}
for all functions $\hat{v} \in H^{k+1}(\hat{\Omega})$ and $v \in H^{k+1}(\Omega)$ in the correspondence of the linear mapping. In the following, we prove that $\Pi_{\Omega}:H^{k+1}(\Omega) \to H^{k+1}(\Omega)$ is indeed the orthogonal projection from $H^{k+1}(\Omega)$ to $\CalP_k(\Omega)$ in $L^2(\Omega)$.

First of all, we have $\Pi_{\Omega} v = (\hat{\Pi} \hat{v}) \circ F^{-1}$ from definition. Since $\hat{\Pi} \hat{v} \in \CalP_k(\hat{\Omega})$, we have $\Pi_{\Omega} v \in \CalP_k(\Omega)$. Secondly, for any $v \in \CalP_k(\Omega)$, $\hat{v} = v\circ F \in \CalP_k(\hat{\Omega})$, and thus $\hat{\Pi} \hat{v} = \hat{v}$ by the definition of $\hat{\Pi}$. Therefore, we have $\Pi_{\Omega} v = \hat{v} \circ F^{-1} = v$ for any $v \in \CalP_k(\Omega)$. Thirdly, by changing variable with $x = F(\hat{x})$, for any $v \in H^{k+1}(\Omega)$ and any $p(x) \in \CalP_k(\Omega)$, we have
\begin{equation*}
	\int_{\Omega} \left(v(x) - (\Pi_{\Omega} v) (x)\right) p(x) \rd x  = \int_{\hat{\Omega}} \left(\hat{v}(\hat{x}) - (\hat{\Pi} \hat{v}) (\hat{x})\right) \hat{p}(\hat{x}) \rd \hat{x} \det{B} = 0.
\end{equation*}
In the last equality, we have used the fact that $\hat{p} \in \CalP_k(\hat{\Omega})$ if $p \in \CalP_k(\Omega)$ and the fact that $\hat{\Pi}: H^{k+1}(\hat{\Omega}) \to \CalP_k(\hat{\Omega})$ is the orthogonal projection in $L^2(\hat{\Omega})$. Therefore, the kernel space of $\Pi_{\Omega}$ is orthogonal to its range space, i.e., $\CalP_k(\Omega)$. With the three points above, we have proved that $\Pi_{\Omega}$ is the orthogonal projection from $H^{k+1}(\Omega)$ to $\CalP_k(\Omega)$ in $L^2(\Omega)$. 

Finally, applying Theorem~\ref{thm:Ciarlet3.1.4} with $\hat{\Pi}$ and $\Pi_{\Omega}$ above, we prove Theorem~\ref{thm:generalPoincare} with the constant $C(k, \hat{\Omega}) := C(\hat{\Pi}, \hat{\Omega})$ in Eqn.~\eqref{eqn:Ciarlet3.1.4}.
\end{proof}

We also give the following theorem, which is a direct result of the Friedrichs' inequality; see, e.g., \cite{Nikolskii1975}.
\begin{theorem}\label{thm:generalFriedrich}
Let $\Omega_h$ be a smooth, bounded, open subset of $\R^d$ with diameter at most $h$. There exists a positive constant $C_f$ such that
\begin{equation}\label{eqn:generalFriedrich}
	|u|_{p,2,\Omega_h} \le C_f h^{k-p} |u|_{k,2,\Omega_h} \quad \forall u \in H_0^k(\Omega_h).
\end{equation}
Here, $C_f = C_f(d,k)$ depends only on the physical dimension $d$ and the order of the derivative $k$.
\end{theorem}

\section{An inverse energy estimation by scaling}
\label{sec:inverseenergy}
In the sparse operator compression, we will show that for a large set of compact operators, the basis functions $\{\psi_i\}_{i=1}^n$ constructed in~\eqref{eqn:psivariational} have exponentially decaying tails, which makes localization of these basis functions possible. The following lemma plays a key role in proving such exponential decay property.
\begin{lemma}\label{lem:scaling}
Let $\Omega_{h}$ be a smooth, bounded, open subset of $\R^d$ with diameter at most $h$ and $B(0, \delta h/2) \subset \Omega_{h}$ for some $\delta > 0$. For $k \in \N$, consider the operator $\CalL = (-1)^k \sum_{|\sigma| = k} D^{2 \sigma}$ with the homogeneous Dirichlet boundary condition on $\partial \Omega_{h}$, i.e.,
\begin{equation}\label{eqn:highorderL}
\begin{split}
	(-1)^k \sum_{|\sigma| = k} D^{2 \sigma} u_h(x) = f(x) &\qquad x \in \Omega_{h},\\
	u_h \in H_0^k(\Omega_{h}).
\end{split}
\end{equation}
Let $\CalP_{s}$ be the space of polynomials with order not greater than $s$. For $\gamma \ge 0$, there exists $C(k, s, d, \delta) > 0$, such that
\begin{equation}\label{eqn:invPoincare}
	\|\CalL u_h\|_{L^2(\Omega_{h})} \le C(k, s, d, \delta) h^{-k}  | u_h|_{k,2,\Omega_{h}}, \quad \forall u_h \in \CalL^{-1} \CalP_{s-1}.
\end{equation}
\end{lemma}
\begin{proof}
Let $G_h$ be the Green's function of Eqn.~\eqref{eqn:highorderL}. After multiplying $u_h$ on both sides of Eqn.~\eqref{eqn:highorderL} and integration by parts, we have $| u_h|_{k,2,\Omega_{h}} = \int_{\Omega_{h}} u_h(x) f(x)\rd x$. Recall that $\CalL u_h \in \CalP_{s-1}$, and thus Eqn.~\eqref{eqn:invPoincare} is equivalent to
\begin{equation}\label{eqn:invPoincare1}
	\int_{\Omega_{h}} p^2(x)\rd x \le \left(C(k, s, d, \delta)\right)^{2} h^{- 2 k}  \int_{\Omega_{h}} \int_{\Omega_{h}} G_h(x,y)p(x)p(y)\rd x\, \rd y, \quad \forall p \in \CalP_{s-1}.
\end{equation}
Let $\{p_1, p_2, \ldots, p_Q\}$ be all the monomials that span $\CalP_{s-1}$. It is easy to see $Q = \binom{s+d-1}{d}$. For convenience, we assume that $\{p_i\}_{i=1}^Q$ are in non-decreasing order with respect to its degree. Specifically, $p_1 = 1$. Let $u_{h,i}$ be the solution of Eqn.~\eqref{eqn:highorderL} with right-hand side $p_i$, and $S_h, M_h \in \R^{Q\times Q}$ be defined as follows:
\begin{equation}\label{eqn:SM}
	S_h(i,j) = \int_{\Omega_{h}} \int_{\Omega_{h}} G_h p_i p_j = \int_{\Omega_{h}} u_{h,i} p_j, \qquad M_h(i,j) = \int_{\Omega_{h}} p_i p_j.
\end{equation}
Then, Eqn.~\eqref{eqn:invPoincare1} is equivalent to 
\begin{equation}\label{eqn:invPoincare2}
	M_h \preceq \left(C(k, s, d, \delta)\right)^{2} h^{- 2 k} S_h,
\end{equation}
where $A \preceq B$ means that $B-A$ is positive semidefinite. The change of variable $x = h z$ leads to $u_i(x) = h^{2k + o_i} u_{1,i}(z)$ where $u_{1,i}$ is the solution of the following PDE on $\Omega_1 \equiv \{x/h~:~x \in \Omega_h\}$:
\begin{equation}\label{eqn:highorderL1}
\begin{split}
	(-1)^k \sum_{|\sigma| = k} D^{2 \sigma} u_{1,i}(x) = p_i(x), &\qquad x \in \Omega_{1},\\
	u_{1,i} \in H_0^k(\Omega_{1}),
\end{split}
\end{equation}
and $o_i$ is the degree of $p_i$. Therefore, it is easy to check that
\begin{equation}\label{eqn:SM2}
	S_h(i,j) = h^{2k + o_i + o_j + d} S_1(i,j), \qquad M_h(i,j) = h^{o_i + o_j + d} M_1(i,j),
\end{equation}
where $S_1(i,j) = \int_{\Omega_{1}} \int_{\Omega_{1}} G_1 p_i p_j = \int_{\Omega_{1}} u_{1,i} p_j$ and $M_1(i,j) = \int_{\Omega_{1}} p_i p_j$, which are independent of $h$. Notice that both $S_1$ and $M_1$ are symmetric positive definite, and let $\lambda_{\max}(M_1, S_1) > 0$ be the largest generalized eigenvalue of $M_1$ and $S_1$. By choosing
\begin{equation}\label{eqn:Cbyeig}
	C(k, s, d, \Omega_1) = \sqrt{\lambda_{\max}(M_1, S_1)},
\end{equation}
we have
\begin{equation}\label{eqn:Cchoose}
	M_1 \preceq \left( C(k, s, d, \Omega_1) \right)^2 S_1.
\end{equation}
Combining~\eqref{eqn:SM2} and~\eqref{eqn:Cchoose}, Eqn.~\eqref{eqn:invPoincare2} naturally follows. In Appendix~\ref{app:moreonscaling}, we prove that $C(k, s, d, \Omega_1)$ can be bounded by $C(k, s, d, \delta)$, and this proves the lemma.
\end{proof}

For the case $s = k = 1$, we can take 
\begin{equation*}
	C(1, 1, d, \delta) = 2 \sqrt{d (d+2)} \delta^{-1 - d/2}.
\end{equation*}
as proved in Proposition~\eqref{prop:Clowerbound}. In this case, we have the estimate
\begin{equation*}
	|u_h|_{1,2,\Omega_h}^2 \ge \frac{\delta^{d+2} h^2 |\Omega_h|}{ 4 d (d+2) },
\end{equation*}
where $|\Omega_{h}|$ is the volume of $\Omega_h$. The above bound is tight: when $\Omega_h$ is a ball with diameter $h$, the equality holds true. Making use of the mean exit time of a Brownian motion, the author of \cite{owhadi2015multi} obtained a different bound
\begin{equation*}
	|u_h|_{1,2,\Omega_h}^2 \ge \frac{\delta^{d+2} h^{2+d} V_d}{ 2^{5+2d} },
\end{equation*}
where $V_d$ is the volume of a unit $d$-dimensional ball. The two estimates have the same order of $\delta$ and $h$, but our estimates from Lemma~\ref{lem:scaling} is much tighter. Moreover, Lemma~\ref{lem:scaling} give estimates for any order $k$ and any degree $s$, which plays a key role in proving the exponential decay in high-order cases, but the mean exit time of a Brownian motion is difficult to generalize to get these higher-order results. 

\section{Exponential decay of basis functions: the second-order case}\label{sec:k1case}
\revise{The analysis for a general higher-order elliptic PDE is quite technical. In this section, we will prove that the basis function $\psi_i$ for a second-order elliptic PDE has exponential decay away from $\tau_i$. When $c \equiv 0$, this problem has been studied in~\cite{owhadi2015multi}. When $c \neq 0$, it has been recently studied in~\cite{owhadi2016gamblets} independently of our work. The results presented in this second-order case are not new~\cite{owhadi2015multi}. We would like to use the simpler second-order elliptic PDE example to illustrate the main ingredients in the proof of exponential decay for a higher-order elliptic PDE, namely the recursive argument, the projection-type approximation property and the inverse energy estimate. }

Consider the following second-order elliptic equation:
\begin{equation}\label{eqn:elliptic}
\begin{split}
	\CalL u &:= -\nabla\cdot (a(x) \nabla u(x) ) + c(x) u(x) = f(x) \qquad x \in D,\\
	u &\in H_0^1(D),
\end{split}
\end{equation}
where $D$ is an open bounded domain in $\R^d$, the potential $c(x) \ge 0$ and the diffusion coefficient $a(x)$ is a symmetric, uniformly elliptic $d\times d$ matrix with entries in $L^{\infty}(D)$. For simplicity, we consider the homogeneous Dirichlet boundary condition here. We emphasize that all our analysis can be carried over for other types of homogeneous boundary conditions. We assume that there exist $0 < a_{\min} \le a_{\max}$ and $c_{\max}$ such that
\begin{equation}\label{eqn:coersive}
	a_{\min} I_d \ \preceq a(x) \preceq  a_{\max} I_d, \quad 0 \le c(x) \le c_{\max} \qquad  x \in D.
\end{equation}

To simply our notations, for any $\psi \in H$ and any subdomain $S \subset D$, $\|\psi\|_{H(S)}$ denotes $\left(\int_{S} \nabla \psi \cdot a \nabla \psi  + c \psi^2\right)^{1/2}$. For the second-order case, the projection-type approximation property is simply the Poincare inequality. The following lemma provides us the inverse energy estimate. It is a special case of Lemma~\ref{lem:boundLf2}, and can be proved by using Lemma~\ref{lem:scaling}. 
\begin{lemma}\label{lem:boundLf}
For any domain partition with $h \le h_0 \equiv \pi \sqrt{\frac{a_{\max}}{2 c_{\max}}}$, we have
\begin{equation}\label{k1eqn:invenergy}
	\|\CalL v\|_{L^2(\tau_j)} \le \sqrt{a_{\max}} C(d, \delta) h^{-1} \|v\|_{H(\tau_j)} \quad \forall v \in \Psi, \quad \forall j = 1,2,\ldots,m,
\end{equation}
where $C(d, \delta) = \sqrt{8 d(d+2)} \delta^{-1-d/2}$. If $c_{\max} = 0$, i.e., $c(x) \equiv 0$, Eqn.~\eqref{k1eqn:invenergy} holds true for all $h > 0$ and $C(d, \delta) = \sqrt{4 d(d+2)} \delta^{-1-d/2}$.
\end{lemma}

Now, we are ready to prove the exponential decay of the basis function $\psi_i$.

\begin{theorem}\label{thm:expdecay1}
For $h \le h_0 \equiv \pi \sqrt{\frac{a_{\max}}{2 c_{\max}}}$, it holds true that 
\begin{equation}\label{eqn:expdecay1}
	\|\psi_i\|_{H(D \cap (B(x_i,r))^c)}^2 \le \exp\left(1 - \frac{r}{l h}\right) \|\psi_i\|_{H(D)}^2
\end{equation}
with $l = \frac{e-1}{\pi}(1+C(d, \delta))\sqrt{\frac{a_{\max}}{a_{\min}}}$ and $C(d, \delta) = \sqrt{8 d(d+2)} (1/\delta)^{d/2+1}$. If $c_{\max} = 0$, i.e., $c(x) \equiv 0$, Eqn.~\eqref{eqn:expdecay1} holds true for all $h > 0$ with $l = \frac{e-1}{\pi}(1+C(d, \delta))\sqrt{\frac{a_{\max}}{a_{\min}}}$ and $C(d, \delta) = \sqrt{4 d(d+2)} \delta^{-1-d/2}$.
\end{theorem}
\begin{proof}
Let $k\in \N$, $l > 0$ and $i \in \{1, 2, \ldots, m\}$. Let $S_0$ be the union of all the domains $\tau_j$ that are contained in the closure of $B(x_i, k l h) \cap D$, let $S_1$ be the union of all the domains $\tau_j$ that are not contained in the closure of $B(x_i, (k+1) l h) \cap D$ and let $S^*=S_0^c \cap S_1^c \cap D$ (be the union of all the remaining elements $\tau_j$ not contained in $S_0$ or $S_1$), as illustrated in Figure~\ref{fig:illustrate1}. 
\begin{figure}[ht]
\centering
\includegraphics[width = 0.4\textwidth]{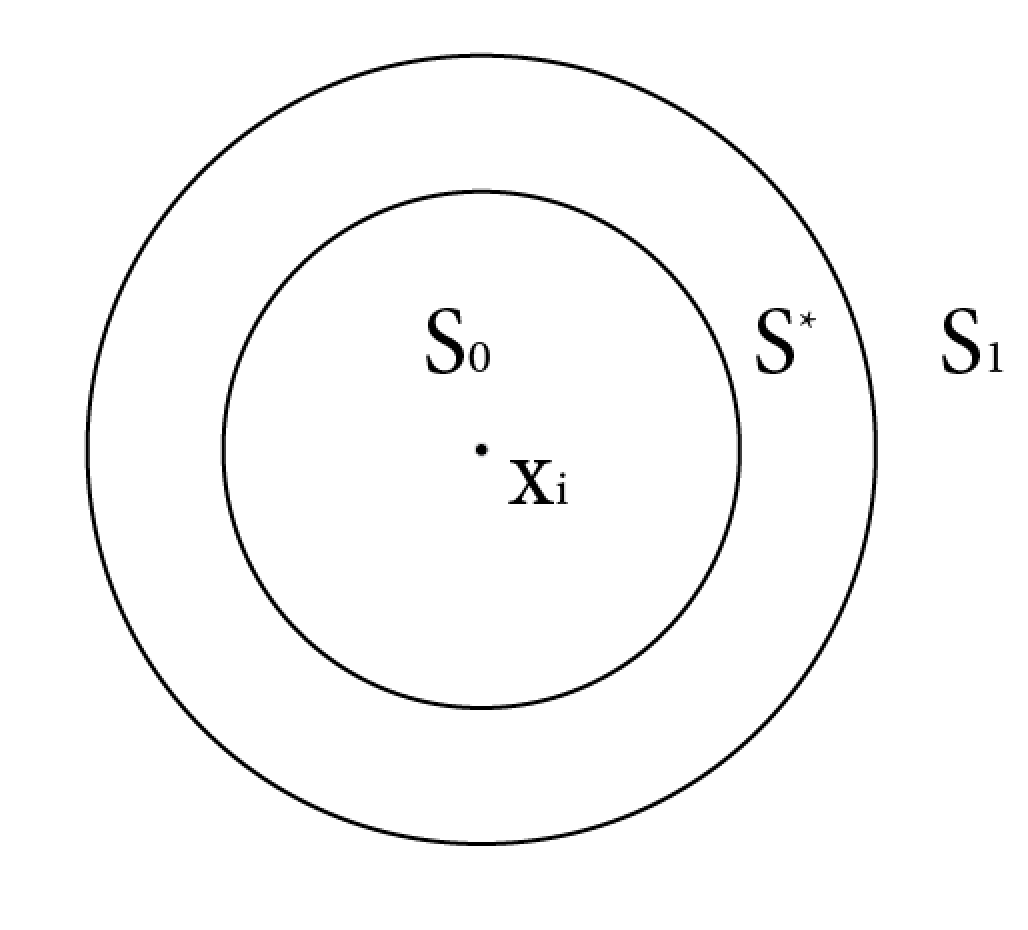}
\caption{Illustration of $S_0$, $S_1$ and $S^*$.}\label{fig:illustrate1}
\end{figure}

Let $b_k := \|\psi_i\|_{H(S_0^c)}^2$, and from definition we have $b_{0} = \|\psi_i\|_{H(D)}^2$, $b_{k+1} = \|\psi_i\|_{H(S_1)}^2$ and $b_{k}-b_{k+1} =  \|\psi_i\|_{H(S^*)}^2$. The strategy is to prove that for any $k \ge 1$, there exists constant $C$ such that $b_{k+1} \le C(b_{k} - b_{k+1})$. Then, we have $b_{k+1} \le \frac{C}{C+1} b_k$ for any $k \ge 1$ and thus we get the exponential decay $b_k \le (\frac{C}{C+1})^{k-1} b_1 \le (\frac{C}{C+1})^{k-1} b_0$. We will choose $l$ such that $C \le \frac{1}{\econst - 1}$ and thus get $b_k \le e^{1-k} b_0$, which gives the result~\eqref{eqn:expdecay1}. We start from $k=1$ because we want to make sure $\tau_i \in S_0$; otherwise, $S_0 = \emptyset$ and $\tau_i \in S^*$.

Now, we prove that for any $k \ge 1$, there exists constant $C$ such that $b_{k+1} \le C(b_{k} - b_{k+1})$, i.e., $ \|\psi_i\|_{H(S_1)}^2 \le C \|\psi_i\|_{H(S^*)}^2$. Let $\eta$ be the function on $D$ defined by $\eta(x) = \dist(x, S_0)/\left(\dist(x, S_0) + \dist(x, S_1)\right)$. Observe that (1) $0 \le \eta \le 1$ (2) $\eta$ is equal to zero on $S_0$ (3) $\eta$ is equal to one on $S_1$ (4) $\|\nabla \eta\|_{L^{\infty}(D)} \le \frac{1}{l h}$. \footnote{$\|\nabla \eta\|_{L^{\infty}(D)} := \esssup\limits_{x\in D} |\nabla \eta (x)|$.} 

By integration by parts, we obtain
\begin{equation}\label{eqn:intbyparts}
	\int_D \eta \nabla \psi_i \cdot a \nabla \psi_i + \int_D \eta c |\psi_i|^2 = \underbrace{\int_D \eta \psi_i (-\nabla\cdot (a \nabla \psi_i ) + c \psi_i)}_{I_2} \underbrace{- \int_D \psi_i \nabla \eta \cdot a \nabla \psi_i}_{I_1}.
\end{equation}
Since $a \succeq 0$ and $c \ge 0$, the left-hand side gives an upper bound for $\|\psi_i\|_{H(S_1)}$. Combining $\nabla \eta \equiv 0$ on $S_0 \cup S_1$ and the Cauchy--Schwarz inequality, we obtain
\begin{equation}\label{eqn:k1_I1}
\begin{split}
	 I_1 &\le \|\nabla \eta\|_{L^{\infty}(D)} \|\psi_i\|_{L^2(S^*)} \left(\int_{S^*} \nabla \psi_i \cdot a \nabla \psi_i \right)^{1/2} \sqrt{a_{\max}} \\
	 & \le \frac{1}{l h} \|\psi_i\|_{L^2(S^*)} \|\psi_i\|_{H(S^*)} \sqrt{a_{\max}}.
\end{split}
\end{equation}
We have used $c \ge 0$ to get $\left(\int_{S^*} \nabla \psi_i \cdot a \nabla \psi_i \right)^{1/2} \le \|\psi_i\|_{H(S^*)}$ in the last inequality. By the construction of $\psi_i$~\eqref{eqn:psivariational}, we have $\int_D \psi_i \phi_j = 0$ for $i \neq j$. Thanks to~\eqref{eqn:psi1}, we have $-\nabla\cdot (a \nabla \psi_i ) + c \psi_i \in \Phi$. Therefore, we have $\int_{S_1} \eta \psi_i (-\nabla\cdot (a \nabla \psi_i ) + c \psi_i) = 0$. Denoting $\eta_j$ as the volume average of $\eta$ over $\tau_j$, we have
\begin{equation}\label{eqn:k1_I2}
\begin{split}
	 I_2 &= - \int_{S^*} \eta \psi_i (-\nabla\cdot (a \nabla \psi_i ) + c \psi_i) = -\sum_{\tau_j \in S^*}\int_{\tau_j} (\eta-\eta_j) \psi_i (-\nabla\cdot (a \nabla \psi_i ) + c \psi_i) \\
	 & \le \frac{1}{l} \sum_{\tau_j \in S^*} \|\psi_i\|_{L^2(\tau_j)} \|\CalL \psi_i\|_{L^2(\tau_j)}.
\end{split}
\end{equation}
Up to now, $I_1$ and $I_2$ are some quantities of $\psi_i$ purely on $S^*$, and we only need to prove that both of them can be bounded by $\|\psi_i\|_{H(S^*)}^2$ (up to a constant). By applying the Poincare inequality, we can easily do this for $I_1$, as we will see soon. However, $I_2$ involves the high-order term $\|\CalL \psi_i\|_{L^2(\tau_j)}$ which in general may not be bounded by the lower-order term $\|\psi_i\|_{H(S^*)}$. Fortunately, this can be proved since $\CalL \psi_i \in \Phi$, the piecewise constant function space. For the current operator $\CalL u= -\nabla\cdot (a(x) \nabla u) + c(x)u$ with rough coefficient $a$ and nonzero potential $c$, Lemma~\ref{lem:boundLf} implies $\|\CalL \psi_i\|_{L^2(\tau_j)} \le \sqrt{a_{\max}}C(d, \delta) h^{-1} \|\psi_i\|_{H(\tau_j)}$ when $h \le h_0\equiv \pi \sqrt{\frac{a_{\max}}{2 c_{\max}}}$. Then, we obtain
\begin{equation}\label{eqn:I2}
	 I_2 \le \frac{\sqrt{a_{\max}}C(d, \delta)}{l h} \|\psi_i\|_{L^2(S^*)} \|\psi_i\|_{H(S^*)}\quad \forall h \le h_0.
\end{equation}

By the construction of $\psi_i$~\eqref{eqn:psivariational}, we have $\int_{\tau_j} \psi_i = 0$ for all $\tau_j \in S^*$. By the Poincare inequality, we have $\|\psi_i\|_{L^2(\tau_j)}\le \|\nabla \psi_i\|_{L^2(\tau_j)} h/\pi$, and then we obtain
\begin{equation}\label{eqn:I1}
	\|\psi_i\|_{H(S_1)}^2 \le I_1 + I_2 \le \frac{1 + C(d, \delta)}{\pi l} \sqrt{\frac{a_{\max}}{a_{\min}}} \|\psi_i\|_{H(S^*)}^2.
\end{equation}

By taking $l \ge \frac{\econst-1}{\pi}(1+C(d, \delta))\sqrt{\frac{a_{\max}}{a_{\min}}}$, we have the constant $\frac{1 + C(d, \delta)}{\pi l} \sqrt{\frac{a_{\max}}{a_{\min}}} \le\frac{1}{\econst - 1}$. With the iterative argument given before, we have proved the exponential decay.
\end{proof}

\begin{remark}\label{rem:otherBCk1}
We point out that boundary conditions may be important in several applications. For example, the Robin boundary condition is useful in the application of the SPCA. The periodic boundary condition is useful in compressing a Hamiltonian with a periodic boundary condition in quantum physics.

The above proof can be applied to the operator $\CalL$ in~\eqref{eqn:elliptic} with other boundary conditions as long as the corresponding problem $\CalL u = f$ has a unique solution $u \in H^k(D)$ for every $f\in L^2(D)$. For other homogeneous boundary condition, the Cameron--Martin space is not $H_0^1(D)$. Instead, we should use the solution space associated with the corresponding boundary condition. The proof of Theorem~\ref{thm:expdecay1} can be easily carried over to other homogeneous boundary conditions, and the only difference is that a different boundary condition leads to slightly different integration by parts in \eqref{eqn:intbyparts}. For the homogeneous Neumann boundary condition or the periodic boundary condition, the proof is exactly the same because the integration by parts~\eqref{eqn:intbyparts} can be carried out in exactly the same way. For the problems with the Robin boundary condition, i.e.,
\begin{equation}\label{eqn:ellipticRobin}
\begin{split}
	\CalL u := -\nabla\cdot (a(x) \nabla u(x) ) + c(x) u(x) &= f(x) \qquad x \in D,\\
	\frac{\partial u}{\partial n} + \alpha(x) u(x) & = 0 \qquad x \in \partial D,
\end{split}
\end{equation}
where $\alpha(x) \ge 0$, the Cameron--Martin space is the subspace of $H^1(D)$ in which all elements satisfy the Robin boundary condition and the associated energy norm is defined as
\begin{equation}\label{eqn:energyRobin}
	\|u\|_H^2 = \int_D \nabla u \cdot a\nabla u + \int_D c u^2 + \int_{\partial D} \alpha u^2.
\end{equation}
In this case, for a subdomain $S \subset D$, the local energy norm on $S$ should be modified as follows:
\begin{equation}\label{eqn:energyRobinlocal}
	\|u\|_{H(S)}^2 = \int_S \nabla u \cdot a\nabla u + \int_S c u^2 + \int_{\partial D \cap \partial S} \alpha u^2.
\end{equation}
Similarly, we can define the Cameron--Martin space and the associated energy norm for the homogeneous mixed boundary conditions.
\end{remark}

\section{Exponential decay of basis functions: the higher-order case}\label{sec:k2case}
In this section, we will study the case when $\CalK:L^2(D) \to L^2(D)$ is the solution operator of the following higher-order elliptic equation:
\begin{equation}\label{eqn:highorder}
\begin{split}
	& \CalL u := \sum_{0 \le |\sigma|, |\gamma| \le k} (-1)^{|\sigma|} D^{\sigma} (a_{\sigma \gamma}(x) D^{\gamma} u) = f,\\
	& f \in L^2(D), \qquad u \in H_0^k(D).
\end{split}
\end{equation}
Here, we only consider the case when $\CalL$ (thus $\CalK$) is self-adjoint, i.e., 
\begin{equation}\label{eqn:selfadjoint}
	\int_D (\CalL u) v = \int_D u (\CalL v) \qquad \forall u,v\in H_0^k(D).
\end{equation} 
The corresponding symmetric bilinear form on $H_0^k(D)$ is denoted as
\begin{equation}\label{eqn:bilinearB}
	B(u, v) = \sum_{0 \le |\sigma|, |\gamma| \le k} \int_D a_{\sigma \gamma}(x) D^{\sigma} u D^{\gamma} v.
\end{equation}
%where
%\begin{equation}\label{eqn:I1I2}
%\begin{split}
%	I_1 &= \sum_{ |\sigma| = |\gamma| = k} \int_D a_{\sigma \gamma}(x) D^{\sigma} u D^{\gamma} v,\\
%	I_2 &= \sum_{ |\sigma| + |\gamma| < 2 k} \int_D a_{\sigma \gamma}(x) D^{\sigma} u D^{\gamma} v.
%\end{split}
%\end{equation}
We assume that $B$ is an inner product on $H_0^k(D)$ and the induced norm $\left(B(u,u)\right)^{1/2}$ is equivalent to the $H_0^k(D)$ norm, i.e., there exists $0 < a_{\min} \le a_{\max}$ such that
\begin{equation}\label{eqn:normequiv}
	a_{\min} |u|_{k,2,D}^2 \le B(u,u) \le a_{\max} |u|_{k,2,D}^2 \qquad \forall u \in H_0^k(D).
\end{equation}
Thanks to the Riesz representation lemma,  Eqn.~\eqref{eqn:highorder} has a unique weak solution in $H_0^k(D)$ for  $f \in L^2(D)$. 
%We further assume that $a_{\gamma \sigma} \in L^{\infty}(D)$ for all $\gamma$ and  $\sigma$ and that there exists a constant $\theta > 0$ such that
%\begin{equation}\label{eqn:strongelliptic}
%	I_1 \ge \theta \sum_{|\sigma| = k} \int_D |D^{\sigma} u|^2 \equiv \theta |u|_{k,2,D}^2.
%\end{equation}
%\textcolor{red}{Notice that the condition~\eqref{eqn:strongelliptic} is different from the standard uniformly elliptic condition which requires $\sum_{ |\sigma| = |\gamma| = k} \vct{\xi}^{\sigma} a_{\sigma \gamma}(x) \vct{\xi}^{\gamma} \ge \theta |\vct{\xi}|^{2 k}$ for every $x \in D$ and $\vct{\xi} \in \R^d$. Our results are expected to be true under the standard uniformly elliptic condition, but the proof needs more details.}

\subsection{Construction of basis functions and the approximation rate}\label{subsec:k2construction}
Suppose $D$ is divided into elements $\{\tau_i\}_{1\le i \le m}$, where each element $\tau_i$ is a triangle or a quadrilateral in 2D, or a tetrahedron or hexahedron in 3D. Denote the maximum element diameter by $h$. We also assume that the subdivision is regular~\cite{ciarlet2002finite}. This means that if $h_i$ denotes the diameter of $\tau_i$ and $\rho_i$ denotes the maximum diameter of a ball inscribed in $\tau_i$, there is a constant $\delta>0$ such that
\begin{equation*}
	\frac{\rho_i}{h_i} \ge \delta \qquad \forall i = 1, 2, \ldots, m.
\end{equation*}
Applying Theorem~\ref{thm:generalPoincare} to $\Omega = \tau_j$, for any $u \in H^k(D)$ and any $0 \le p \le k$, we have
\begin{equation*}\
	|u - \Pi_i u|_{p,2,\tau_i} \le C(k-1,\hat{\tau}_i) \delta^{-p} h^{k-p} |u|_{k,2,\tau_i},
\end{equation*}
where $\Pi_i:H^k(\tau_i) \to \CalP_{k-1}(\tau_i)$ is the orthogonal projection to the polynomial space $\CalP_{k-1}(\tau_i)$ in $L^2(\tau_i)$, and $\hat{\tau}_i$ is some reference domain that is affine equivalent to $\tau_i$. Notice that the constant $C(k-1, \hat{\tau}_i) \delta^{-p}$ can be bounded from above by a constant $C_p$ for all the elements $\{\tau_i\}_{1\le i \le m}$, because all elements in $\{\tau_i\}_{1\le i \le m}$ are affine equivalent to an equilateral triangle or square in 2D, or a equilateral 3-simplex or cubic in 3D. Therefore, for any $u \in H^k(D)$, any $1 \le i \le m$ and any $0 \le p \le k$, we have
\begin{equation}\label{eqn:generalPoincare}
	|u - \Pi_i u|_{p,2,\tau_i} \le C_p h^{k-p} |u|_{k,2,\tau_i}.
\end{equation}
Specifically for $p=0$, $\tilde{u} \in L^2(D)$ with $\tilde{u}|_{\tau_i} = \Pi_i u$, we conclude that
\begin{equation}\label{eqn:Poincare2L2}
	\|u - \tilde{u}\|_{L^2(D)} \le C_p h^{k} |u|_{k,2,D}.
\end{equation}

Let $X = L^2(D)$ and $H = H_0^k(D)$. We use the standard inner product for $L^2(D)$ and use the inner product $\langle u, v \rangle = B(u,v)$ for $H$. Further, we denote $\CalK : L^2(D) \to L^2(D)$ as the operator mapping $f$ to the solution $u$ in Eqn.~\eqref{eqn:highorder}. Let $\{\phi_{i,q}\}_{q=1}^{Q}$ be an orthogonal basis of $\CalP_{k-1}(\tau_i)$ with respect to the inner product in $L^2(\tau_i)$, where $Q = \binom{k+d-1}{d}$ is the number of $d$-variate monomials with degree at most $k-1$. We take
\begin{equation}\label{eqn:PhiPsik2}
	\Phi = \text{span}\{\phi_{i,q}: 1\le q\le Q, 1\le i\le m\}, \quad \Psi = \CalK \Phi.
\end{equation}
Without loss of generality, we normalize these basis functions such that 
\begin{equation}\label{eqn:phinormalize}
	\int_{\tau_i} \phi_{i,q} \phi_{i,q'} = |\tau_i| \delta_{q,q'}.
\end{equation}
A set of basis functions of $\Psi$ is defined by Eqn.~\eqref{eqn:psivariational} accordingly, i.e.,
\begin{equation}\label{eqn:psivariationalk2}
\begin{split}
	\psi_{i,q} = &\argmin_{\psi \in H_0^k(D)} \quad \|\psi\|_H^2 \\
			& \text{s.t.} \quad \int_D \psi_{i,q} \phi_{j,q'} = \delta_{iq,jq'} \quad \forall 1 \le q' \le Q, \quad 1\le j \le m.
\end{split}	 
\end{equation}

Combining Eqn.~\eqref{eqn:normequiv} and \eqref{eqn:Poincare2L2}, we have
\begin{equation}\label{eqn:Poincare2L2Hnorm}
	\|u - \CalP_{\Phi}^{(X)} u\|_{L^2(D)} \le \frac{C_p h^k}{\sqrt{a_{\min}}} \|u\|_{H}, \quad \forall u \in H.
\end{equation}
Applying Theorem~\ref{thm:conditioning1} with $X$ and $H$ defined above, we have
\begin{enumerate}
\item[1.] For any $u \in H$ and $\CalL u = f$, we have
\begin{equation}\label{eqn:Herror2}
	\| u - \CalP_{\Psi}^{(H)} u \|_{H} \le \frac{C_p h^k}{\sqrt{a_{\min}}} \|f\|_{L^2(D)}\,.
\end{equation}
Here, $C_p$ plays the role of the Poincare constant $1/\pi$.
\item[2.] For any $u \in H$ and $\CalL u = f$, we have
\begin{equation}\label{eqn:Xerror2}
	\| u - \CalP_{\Psi}^{(H)} u \|_{L^2(D)} \le \frac{C_p^2 h^{2 k}}{a_{\min}} \|f\|_{L^2(D)}\,.
\end{equation}
\item[3.] We have
\begin{equation}\label{eqn:Kerror2}
	\| \CalK - \CalP_{\Psi}^{(H)} \CalK \| \le \frac{C_p^2 h^{2 k}}{a_{\min}}\,.
\end{equation}
\end{enumerate}

Notice that the eigenvalues of the operator $\CalL$ (with the homogeneous Dirichlet boundary conditions) in~\eqref{eqn:highorder} grow like $\lambda_n(\CalL) \sim n^{2 k/d}$ (see, e.g., \cite{MELENK2000272, DAHLKE200629}), and thus, the eigenvalues of $\CalK$ decay like $\lambda_n(\CalK) \sim n^{-2 k/d}$. Meanwhile, the rank of the operator $\CalP_{\Psi}^{(H)} \CalK$, denoted as $n$, roughly scales like $Q/h^d$ where $1/h^d$ is roughly the number of patches. Plugging $n = Q/h^d$ into Eqn.~\eqref{eqn:Kerror2}, we have
\begin{equation}
	\| \CalK - \CalP_{\Psi}^{(H)} \CalK \| \le \frac{C_p^2 Q^{2 k/d}}{a_{\min}} n^{-2k/d} \underset{\sim}{<}  \lambda_n(\CalK)\,.
\end{equation}
Therefore, our construction of the $m$-dimensional subspace $\Psi$ approximates $\CalK$ at the optimal rate. In Subsection~\ref{subsec:strongellipticity}, we introduce the concept of {\it strong ellipticity} that enables us to prove exponential decay results. In Subsection~\ref{subsec:k2decay}, we will prove that the basis functions $\psi_{i,q}$ defined in Eqn.~\eqref{eqn:psivariationalk2} have exponential decay away from $\tau_i$. 

\subsection{The strong ellipticity condition}\label{subsec:strongellipticity}
In our proof, we need the following {\it strong ellipticity condition} of the operator $\CalL$ to obtain the exponential decay. 
\begin{definition}\label{def:strongellipticity}
An operator in the divergence form $\CalL u := \sum\limits_{0 \le |\sigma|, |\gamma| \le k} (-1)^{|\sigma|} D^{\sigma} (a_{\sigma \gamma}(x) D^{\gamma} u)$ is strongly elliptic if there exists $\theta_{k,\min} > 0$ such that
\begin{equation}\label{eqn:strongelliptic0}
	\sum_{|\sigma|= |\gamma| = k} a_{\sigma \gamma}(x) \vct{\zeta}_{\sigma} \vct{\zeta}_{\gamma} \ge \theta_{k,\min} \sum_{|\sigma| = k} \vct{\zeta}_{\sigma}^2 \quad \forall x \in D, \quad \forall \vct{\zeta} \in \R^{\binom{k+d-1}{k}},
\end{equation}
where $\vct{\zeta}_{\sigma}$ and $\vct{\zeta}_{\gamma}$ are the $\sigma$'th and $\gamma$'th entry of $\vct{\zeta}$, respectively. One can check that $\binom{k+d-1}{k}$ is exactly the number of all possible $k$th derivatives, i.e., $\#\{D^{\sigma} u: |\sigma| = k\}$.

For a $2k$th-order partial differential operator $\CalL u = (-1)^k \sum\limits_{|\alpha|\le 2k} a_{\alpha} D^{\alpha} u$, $\CalL$ is strongly elliptic if there exists a strongly elliptic operator in the divergence form $\tilde{\CalL}$ such that $\CalL u = \tilde{\CalL} u$ for all $u \in C^{2k}(D)$.
\end{definition}

\begin{remark}\label{rem:strongellipticex}
For a $2k$th-order partial differential operator $\CalL u = (-1)^k \sum\limits_{|\alpha|\le 2k} a_{\alpha} D^{\alpha} u$, its divergence form may not be unique. It is possible that it has two divergence forms, and one does not satisfy the {\it strong ellipticity condition}~\eqref{def:strongellipticity} while the other does. For example, the biharmonic operator $\CalL = \Delta^2$ in two space dimensions have the following two different divergence forms:
\begin{equation}\label{ex:bihamornic1}
	\CalL u = \sum\limits_{|\sigma| = |\gamma| = 2} D^{\sigma} (a_{\sigma \gamma} D^{\gamma} u) = \sum\limits_{|\sigma| = |\gamma| = 2} D^{\sigma} (\tilde{a}_{\sigma \gamma}(x) D^{\gamma} u),
\end{equation}
where
\begin{equation}\label{ex:bihamornic2}
	(a_{\sigma \gamma}) = \begin{bmatrix} 1 & 1 & 0 \\ 1 & 1 & 0 \\ 0 & 0 & 0\end{bmatrix}, \quad (\tilde{a}_{\sigma \gamma}) = \begin{bmatrix} 1 & 0 & 0 \\ 0 & 1 & 0 \\ 0 & 0 & 2 \end{bmatrix},
\end{equation}
when $\{D^{\sigma} u: |\sigma| = 2\}$ is ordered as $(\partial_{x_1}^2, \partial_{x_2}^2, \partial_{x_1} \partial_{x_2})$. Obviously, the first one does not satisfy the {\it strong ellipticity condition}~\eqref{def:strongellipticity} while the second one does. These two divergence forms correspond to two bilinear forms on $H_0^2(D)$:
 \begin{equation}\label{ex:bihamornic3}
 	B(u, v) = \int_D \Delta u \Delta v, \quad \tilde{B}(u,v) = \int_D D^2 u : D^2 v,
 \end{equation}
 where $D^2 u : D^2 v = \sum_{i,j} \frac{\partial^2 u}{\partial x_i \partial x_j} \frac{\partial^2 v}{\partial x_i \partial x_j}$.
\end{remark}

The {\it strong ellipticity condition} guarantees that for any local subdomain $S \subset D$, the seminorm $|\cdot|_{k,2,S}$ can be controlled by the {\it local energy norm} $\|\cdot\|_{H(S)}$.
\begin{lemma}\label{lem:coercive}
Suppose $\CalL u = \sum\limits_{0 \le |\sigma|, |\gamma| \le k} (-1)^{|\sigma|} D^{\sigma} (a_{\sigma \gamma}(x) D^{\gamma} u)$ is self-adjoint. Assume that $a_{\sigma \gamma}(x) \in L^{\infty}(D)$ for all $0 \le |\sigma|, |\gamma| \le k$ and that for any $x \in D$
\begin{itemize}
\item $\CalL$ is nonnegative, i.e.,
	\begin{equation}\label{eqn:Lnonnegativelem}
		\sum_{ 0 \le |\sigma|, |\gamma| \le k} a_{\sigma \gamma}(x) \vct{\zeta}_{\sigma} \vct{\zeta}_{\gamma} \ge 0 \qquad \forall \vct{\zeta} \in \R^{\binom{k+d}{k}},
	\end{equation}
\item $\CalL$ is bounded, i.e., there exist $\theta_{0,\max}\ge 0$ and $\theta_{k,\max} > 0$ such that
	\begin{equation}\label{eqn:Lboundedlem}
		\sum_{ 0 \le |\sigma|, |\gamma| \le k} a_{\sigma \gamma}(x) \vct{\zeta}_{\sigma} \vct{\zeta}_{\gamma} \le \theta_{k,\max} \sum_{|\sigma| = k} \vct{\zeta}_{\sigma}^2 + \theta_{0,\max} \sum_{|\sigma| < k} \vct{\zeta}_{\sigma}^2 \qquad \forall x \in D, \quad \forall \vct{\zeta} \in \R^{\binom{k+d}{k}},
	\end{equation}
\item and $\CalL$ is strongly elliptic, i.e., there exists $\theta_{k,\min} > 0$ such that
	\begin{equation}\label{eqn:strongellipticlem}
		\sum_{ |\sigma| = |\gamma| = k} a_{\sigma \gamma}(x) \vct{\zeta}_{\sigma} \vct{\zeta}_{\gamma} \ge \theta_{k,\min} \sum_{|\sigma| = k} \vct{\zeta}_{\sigma}^2 \qquad \forall x \in D, \quad \forall \vct{\zeta} \in \R^{\binom{k+d-1}{k}}.
	\end{equation}
\end{itemize}

For any subdomain $S \subset D$ and any $\psi \in H^{k}(D)$, define 
\begin{equation}\label{eqn:localHnorm}
	\|\psi\|_{H(S)}^2 = \sum_{0 \le |\sigma|, |\gamma| \le k} \int_S a_{\sigma \gamma}(x) D^{\sigma} \psi D^{\gamma} \psi.
\end{equation}
Then, the following two claims hold true.
\begin{itemize}
	\item If $\CalL$ contains only highest order terms, i.e., $\CalL u = \sum\limits_{|\sigma| = |\gamma| = k} (-1)^{|\sigma|} D^{\sigma} (a_{\sigma \gamma}(x) D^{\gamma} u)$, then we have
		\begin{equation}\label{eqn:kleH1}
			|\psi|_{k,2,S} \le \theta_{k,\min}^{-1/2} \|\psi\|_{H(S)} \qquad \forall \psi \in H^{k}(D).
		\end{equation}
	\item If $\CalL$ contains lower-order terms, for any regular domain partition $D = \cup_{i=1}^m \tau_i$ with diameter $h>0$ satisfying $\frac{h^2(1-h^{2k})}{1-h^2} \le \frac{\theta_{k,\min}^2}{16 \theta_{0,\max}\theta_{k,\max} C_p^2}$, and any subdomain $S = \cup_{_{j\in \Lambda}} \tau_j$, we have
		\begin{equation}\label{eqn:kleH2}
			|\psi_{i,q}|_{k,2,S} \le \left(2/\theta_{k,\min}\right)^{1/2} \|\psi_{i,q}\|_{H(S)} \qquad \forall \tau_i \not\in \mathcal{S}, \quad \forall 1\le q\le Q.
		\end{equation}
		Here, $\Lambda$ is any subset of $\{1, 2, \ldots, m\}$, and $\psi_{i,q}$ is defined by Eqn.~\eqref{eqn:psivariationalk2}.
\end{itemize}
\end{lemma}
\begin{proof}
The first point can be obtained directly from the definition of strong ellipticity. In the following, we provide the proof of the second point. For $S$ stated in the second point and any $\psi \in H^{k}(D)$, we have
\begin{equation}\label{eqn:kleHproof1}
\begin{split}
	\|\psi\|_{H(S)}^2 =& \underbrace{\sum_{|\sigma|= |\gamma|= k} \int_S a_{\sigma \gamma} D^{\sigma} \psi D^{\gamma} \psi}_{J_1} + \underbrace{\sum_{|\sigma|, |\gamma|< k} \int_S a_{\sigma \gamma} D^{\sigma} \psi D^{\gamma} \psi}_{J_2} \\
	&+ \underbrace{\sum_{|\sigma|=k, |\gamma|< k} \int_S (a_{\sigma \gamma}+a_{\gamma \sigma}) D^{\sigma} \psi D^{\gamma} \psi}_{J_3}.
\end{split}
\end{equation}
From the strong ellipticity~\eqref{eqn:strongellipticlem}, we have 
\begin{equation}\label{eqn:J1}
	J_1 \ge \theta_{k,\min} |\psi|_{k,2,S}^2.
\end{equation}
From the nonnegativity~\eqref{eqn:Lnonnegativelem}, we have
\begin{equation}\label{eqn:J2}
	J_2 \ge 0.
\end{equation}
Combining the nonnegativity~\eqref{eqn:Lnonnegativelem} and the boundedness~\eqref{eqn:Lboundedlem}, we can prove that
\begin{equation*}
	\left| \sum_{|\sigma|=k, |\gamma|< k} (a_{\sigma \gamma}+a_{\gamma \sigma}) D^{\sigma} \psi D^{\gamma} \psi \right| \le 2 \left(\theta_{0,\max}\theta_{k,\max} \sum_{|\sigma|=k} |D^{\sigma} \psi|^2 \sum_{|\sigma|<k} |D^{\sigma} \psi|^2 \right)^{1/2}.
\end{equation*}
Therefore, using the Cauchy--Schwarz inequality, we obtain
\begin{equation}\label{eqn:J3}
	|J_3| \le 2 \theta_{0,\max}^{1/2}\theta_{k,\max}^{1/2} |\psi|_{k,2,S} \|\psi\|_{k-1,2,S}.
\end{equation}
Thanks to the polynomial approximation property, for any $\tau_i \not\in \mathcal{S}$ and $1\le q\le Q$, we have
\begin{equation}\label{eqn:km1tok}
	\|\psi_{i,q}\|_{k-1,2,S}^2 \le C_p^2 \frac{h^2(1-h^{2k})}{1-h^2} |\psi_{i,q}|_{k,2,S}^2.
\end{equation}
Combining Eqn.~\eqref{eqn:J3} and \eqref{eqn:km1tok}, for $\frac{h^2(1-h^{2k})}{1-h^2} \le \frac{\theta_{k,\min}^2}{16 \theta_{0,\max}\theta_{k,\max} C_p^2}$, we have
\begin{equation}\label{eqn:J3_2}
	|J_3| \le \frac{\theta_{k,\min}}{2} |\psi|_{k,2,S}^2.
\end{equation}
Combining Eqn.~\eqref{eqn:kleHproof1}, \eqref{eqn:J1}, \eqref{eqn:J2} and \eqref{eqn:J3_2}, we prove the second point.
\end{proof}
\begin{remark}\label{rem:kleH2}
When $\CalL$ contains lower-order terms but there is no crossing term between $D^{\sigma} u$ ($|\sigma| = k$) and $D^{\sigma} u$ ($|\sigma| < k$), i.e., $J_3 = 0$, we can directly get the same bound in Eqn.~\eqref{eqn:kleH1} for all $h >0$.
\end{remark}

The strong ellipticity condition above is different from the standard uniformly elliptic condition (see Definition 9.2 in \cite{renardy2006introduction}), i.e., a linear partial differential operator $\CalL u= (-1)^k \sum\limits_{|\alpha|\le 2k} a_{\alpha} D^{\alpha} u$ is uniformly elliptic if there exists a constant $\theta_{k,\min}>0$ such that
\begin{equation}\label{eqn:uniformelliptic}
	\sum_{|\alpha| = 2k} a_{\alpha}(x) \vct{\xi}^{\alpha} \ge \theta_{k,\min} |\vct{\xi}|^{2 k}, \quad \forall x \in D, \vct{\xi} \in \R^{d}.
\end{equation}

On one hand, it is obvious that a strongly elliptic operator with smooth coefficients is uniformly elliptic, by taking $\vct{\zeta}_{\sigma} := \vct{\xi}^{\sigma}$ in Eqn.~\eqref{eqn:strongelliptic0}. On the other hand, the relation between the uniform ellipticity and the strong ellipticity turns out to be closely related to the relation between nonnegative polynomials and sum-of-square (SOS) polynomials. In fact, the strongly ellipticity condition~\eqref{eqn:strongelliptic0} is equivalent to that there exists $\theta_{k,\min} > 0$ such that 
\begin{equation*}
	\sum_{|\sigma|= |\gamma| = k} a_{\sigma \gamma}(x) \vct{\xi}^{\sigma} \vct{\xi}^{\gamma} - \theta_{k,\min} \sum_{|\sigma| = k} |\vct{\xi}|^{2 k} = \text{Sum-Of-Squares (SOS) polynomials}.
\end{equation*}
Using the famous Hilbert's theorem (1888) on nonnegative polynomials and SOS polynomials, we have the following theorem. Readers can find the proof and more discussions in \cite{Pengchuan-thesis-2017}.
\begin{theorem}\label{thm:hilbert}
Let $a_{\alpha} \in C^{|\alpha|-k}(\overline{D})$ for $k < |\alpha| \le 2k$, $a_{\alpha} \in C(\overline{D})$ for $|\alpha| \le k$, and $\CalL u = (-1)^k \sum\limits_{|\alpha|\le 2k} a_{\alpha} D^{\alpha} u$ for all $u \in C^{2 k} (D)$. Then, in the following two cases, if $\CalL$ is uniformly elliptic it is also strongly elliptic.
\begin{itemize}
\item $d=1$ or $2$ : one- or two-dimensional physical domain,
\item $k=1$ : second-order partial differential operators.  
\end{itemize}
For the case $(d, k)=(3, 2)$, i.e., fourth-order partial differential operators in 3-dimensional physical domain, all uniformly elliptic operators with constant coefficients are also strongly elliptic. 
\end{theorem}
For the case $(d, k)= (3, 2)$, we are not able to prove that strong ellipticity is equivalent to uniform ellipticity for elliptic operators with smooth and multiscale coefficients, but we suspect that it is true. For all other cases, there are uniformly but not strongly elliptic operators. Fortunately, for small physical dimensions $d$ and differential orders $k$, strongly elliptic operators approximate uniformly elliptic operators well and counter examples are difficult to construct.

\subsection{Exponential decay of basis functions I}\label{subsec:k2decaysimple}
In this subsection, we prove the exponential decay of basis functions constructed in Eqn.~\eqref{eqn:psivariationalk2} for higher-order elliptic operators that contain only the highest order terms. We will leave the proof for the general operators to the next subsection. The proof follows exactly the same structure as that in the second-order elliptic case.  

\begin{theorem}\label{thm:expdecay2simple}
Let $\CalL u = (-1)^k \sum\limits_{|\sigma|=|\gamma|=k} D^{\sigma}(a_{\sigma \gamma} D^{\gamma} u)$ and $a_{\sigma \gamma}(x) \in L^{\infty}(D)$ for all $|\sigma|= |\gamma| = k$. 
Assume that for any $x \in D$
\begin{itemize}
\item $\CalL$ is bounded, i.e., there exist nonnegative $\theta_{k,\max}$ such that
	\begin{equation}\label{eqn:Lboundedsimple}
		\sum_{ |\sigma|= |\gamma|= k} a_{\sigma \gamma}(x) \vct{\zeta}_{\sigma} \vct{\zeta}_{\gamma} \le \theta_{k,\max} \sum_{|\sigma| = k} \vct{\zeta}_{\sigma}^2 \qquad \forall \vct{\zeta} \in \R^{\binom{k+d-1}{k}},
	\end{equation}
\item and $\CalL$ is strongly elliptic, i.e., there exists $\theta_{k,\min} > 0$ such that
	\begin{equation}\label{eqn:strongellipticsimple}
		\sum_{ |\sigma| = |\gamma| = k} a_{\sigma \gamma}(x) \vct{\zeta}_{\sigma} \vct{\zeta}_{\gamma} \ge \theta_{k,\min} \sum_{|\sigma| = k} \vct{\zeta}_{\sigma}^2 \qquad \forall \vct{\zeta} \in \R^{\binom{k+d-1}{k}}.
	\end{equation}
\end{itemize}
Then, for any $1\le i \le m$ and $1\le q\le Q$, it holds true that
\begin{equation}\label{eqn:expdecay2simple}
	\|\psi_{i,q}\|_{H(D \cap (B(x_i,r))^c)}^2 \le \exp\left(1 - \frac{r}{l h} \right) \|\psi_{i,q}\|_{H(D)}^2
\end{equation}
with $\sqrt{l^2-1} \ge (e-1) C_{\eta} C_p (C_1 + C(k,d,\delta)) \sqrt{\frac{\theta_{k,\max}}{\theta_{k,\min}}}$. Here, $C_1$ and $C_{\eta}$ only depends on $k$ and $d$, $C_p$ is the constant in Eqn.~\eqref{eqn:generalPoincare} and $C(k, d, \delta) := C(k, k, d, \delta)$ from Lemma~3.1.
\end{theorem}

\begin{proof}
\revise{The proof follows the same structure as that of Theorem~\ref{thm:expdecay1}  and \cite{owhadi2015multi} (Thm. 3.9).} Let $k\in \N$, $l > 0$ and $i \in \{1, 2, \ldots, m\}$. Let $S_0$ be the union of all the domains $\tau_j$ that are contained in the closure of $B(x_i, k l h) \cap D$, let $S_1$ be the union of all the domains $\tau_j$ that are not contained in the closure of $B(x_i, (k+1) l h) \cap D$ and let $S^*=S_0^c \cap S_1^c \cap D$ (be the union of all the remaining elements $\tau_j$ not contained in $S_0$ or $S_1$). In the following, we will prove that for any $k \ge 1$, there exists constant $C$ such that $\|\psi_{i,q}\|^2_{H(S_1)} \le C \|\psi_{i,q}\|^2_{H(S^*)}$. Then, the same recursive argument in the proof of Theorem~\ref{thm:expdecay1} can be used to prove the exponential decay. 

Let $\eta(x)$ be a smooth function which satisfies (1) $0 \le \eta \le 1$, (2) $\eta|_{B(x_i, k l h)} = 0$, (3) $\eta|_{B^c(x_i, (k+1) l h)} = 1$ and (4) $\|D^{\sigma} \eta\|_{L^{\infty}(D)} \le \frac{C_{\eta}}{(l h)^{|\sigma|}}$ for all $\sigma$.

By integration by parts, we have
\begin{equation*}
	\int_D \eta \psi_{i,q} \CalL \psi_{i,q} = \sum_{|\sigma|= |\gamma| = k} \int_D a_{\sigma \gamma}(x) D^{\sigma}(\eta \psi_{i,q}) D^{\gamma} \psi_{i,q}.
\end{equation*}
Making use of the binomial theorem $D^{\sigma}(\eta \psi_{i,q}) = \eta D^{\sigma} \phi_{i,q} + \sum_{\stackrel{\sigma_1 + \sigma_2 = \sigma}{|\sigma_1| \ge 1}} \binom{\sigma}{\sigma_1} D^{\sigma_1} \eta D^{\sigma_2} \psi_{i,q}$, we obtain
\begin{equation}\label{eqn:intbyparts2simple}
\begin{split}
	& \sum_{|\sigma|= |\gamma| = k} \int_D \eta a_{\sigma \gamma}(x) D^{\sigma}(\psi_{i,q}) D^{\gamma} \psi_{i,q} = \underbrace{\int_D \eta \psi_{i,q} \CalL \psi_{i,q}}_{I_2} \\
	& \underbrace{- \sum_{|\sigma|= |\gamma| = k}\sum_{\stackrel{\sigma_1 + \sigma_2 = \sigma}{|\sigma_1| \ge 1}} \binom{\sigma}{\sigma_1} \int_D a_{\sigma \gamma}(x) D^{\sigma_1} \eta D^{\sigma_2} \psi_{i,q} D^{\gamma} \psi_{i,q}}_{I_1}.
\end{split}
\end{equation}
Since $\sum\limits_{|\sigma|= |\gamma| = k} a_{\sigma \gamma}(x) D^{\sigma}\psi_{i,q} D^{\gamma} \psi_{i,q} \ge 0$ for every $x\in D$, the left-hand side gives an upper bound for $\|\psi_{i,q}\|_{H(S^1)}^2$. 
Since $D^{\sigma_1} \eta = 0$ ($|\sigma_1| \ge 1$) on both $S_0$ and $S_1$, we obtain
\begin{eqnarray}
	  I_1 &=& - \sum_{|\sigma|= |\gamma| = k}\sum_{\stackrel{\sigma_1 + \sigma_2 = \sigma}{|\sigma_1| \ge 1}} \binom{\sigma}{\sigma_1} \int_{S^*} a_{\sigma \gamma}(x) D^{\sigma_1} \eta D^{\sigma_2} \psi_{i,q} D^{\gamma} \psi_{i,q}  \label{eqn:I1_1_1simple} \\
	 &\le& \left( \sum_{|\sigma| = k} \int_{S^*} \left| \sum_{\stackrel{\sigma_1 + \sigma_2 = \sigma}{|\sigma_1| \ge 1}} \binom{\sigma}{\sigma_1} D^{\sigma_1} \eta D^{\sigma_2} \psi_{i,q} \right|^2 \right)^{1/2} \|\psi_{i,q}\|_{H(S^*)} \sqrt{\theta_{k,\max}} \label{eqn:I1_1_2simple}\\
	 &\le& C_1 C_{\eta} \left( \sum_{s'=1}^k (l h)^{-2 s'} |\psi_{i,q}|_{k-s', 2, S^*}^2 \right)^{1/2} \|\psi_{i,q}\|_{H(S^*)} \sqrt{\theta_{k,\max}} \; . \label{eqn:I1_1_3simple}
\end{eqnarray}
Here, $C_1$ is a constant only dependent on $k$ and $d$. We have used the Cauchy--Schwarz inequality and the bound~\eqref{eqn:Lboundedsimple} in Eqn.~\eqref{eqn:I1_1_2simple}. We will defer the proof of the last step in Eqn.~\eqref{eqn:I1_1_3simple} to the Appendix. Since $\psi_{i,q} \perp \CalP_{k-1}$ locally in $L^2$, we obtain from Theorem~\ref{thm:generalPoincare} that
\begin{equation*}
	|\psi_{i,q}|_{k-s', 2, S^*} \le C_p h^{s'} |\psi_{i,q}|_{k, 2, S^*}.
\end{equation*}
Therefore, we get
\begin{eqnarray}
	  I_1 &\le& C_1 C_{\eta} \sqrt{\theta_{k,\max}} C_p \left(\sum_{s'=1}^k l^{- 2s'} |\psi_{i,q}|_{k, 2, S^*}^2\right)^{1/2} \|\psi_{i,q}\|_{H(S^*)}  \label{eqn:I1_2_1simple}\\
	 &\le& \frac{C_1 C_{\eta} \sqrt{\theta_{k,\max}} C_p}{\sqrt{l^2-1}} |\psi_{i,q}|_{k,2,S^*} \|\psi_{i,q}\|_{H(S^*)} \label{eqn:I1_2_3simple}. 
\end{eqnarray}
In the last inequality, we have used $\sum_{s'=1}^k l^{- 2s'} = \frac{1-l^{-2k}}{l^2 -1} \le \frac{1}{l^2 -1}$.

By the construction of $\psi_{i,q}$ given in~\eqref{eqn:psivariationalk2}, we have $\int_D \psi_{i,q} \phi_{j,q'} = 0$ for $i \neq j$. Thanks to~\eqref{eqn:psi1}, we have $\CalL \psi_{i,q} \in \Phi$. Therefore, we get $\int_{S_1} \eta \psi_{i,q} \CalL \psi_{i,q} = 0$. Denoting $\eta_j$ as the volume average of $\eta$ over $\tau_j$, we obtain
\begin{equation}\label{eqn:I2_1simple}
\begin{split}
	 I_2 &= \int_{S^*} \eta \psi_{i,q} \CalL \psi_{i,q} = \sum_{\tau_j \in S^*}\int_{\tau_j} (\eta-\eta_j) \psi_{i,q} \CalL \psi_{i,q} \le \frac{C_{\eta}}{l} \sum_{\tau_j \in S^*} \|\psi_{i,q}\|_{L^2(\tau_j)} \|\CalL \psi_{i,q}\|_{L^2(\tau_j)}.
\end{split}
\end{equation}
By using Lemma~\ref{lem:boundLf2}, which is stated in the beginning of Section 6.5, we have $\|\CalL \psi_{i,q}\|_{L^2(\tau_j)} \le \sqrt{\theta_{k,\max}}C(k, d, \delta) h^{-k} \|\psi_{i,q}\|_{H(\tau_j)}$ for any $h > 0$ because $\CalL$ contains only the highest order derivatives. Then, we obtain
\begin{equation}\label{eqn:I2_2simple}
\begin{split}
	 I_2 &\le \frac{\sqrt{\theta_{k,\max}} C_{\eta} C(k, d, \delta)}{l h^k} \|\psi_{i,q}\|_{L^2(S^*)} \|\psi_{i,q}\|_{H(S^*)} \\
	       &\le \frac{\sqrt{\theta_{k,\max}} C_{\eta} C(k, d, \delta) C_p}{l} |\psi_{i,q}|_{k,2,S^*} \|\psi_{i,q}\|_{H(S^*)},
\end{split}
\end{equation}
where we have used Eqn.~\eqref{eqn:generalPoincare} in the last step. 

Combining Eqn.~\eqref{eqn:I1_2_3simple} and \eqref{eqn:I2_2simple}, we obtain
\begin{equation*}
	I_1 + I_2 \le \sqrt{\frac{\theta_{k,\max}}{l^2-1}} C_{\eta} C_p (C_1 + C(k,d,\delta)) |\psi_{i,q}|_{k,2,S^*} \|\psi_{i,q}\|_{H(S^*)}\,.
\end{equation*}
By the strong ellipticity~\eqref{eqn:strongellipticsimple} and Eqn.~\eqref{eqn:kleH1}, we have $|\psi_{i,q}|_{k,2,S^*} \le \theta_{k,\min}^{-1/2} \|\psi_{i,q}\|_{H(S^*)}$. Therefore, we have
\begin{equation}\label{eqn:recursivek2simple}
	\|\psi_{i,q}\|_{H(S^1)}^2 \le \sqrt{\frac{\theta_{k,\max}}{(l^2-1) \theta_{k,\min}}} C_{\eta} C_p (C_1 + C(k,d,\delta)) \|\psi_{i,q}\|_{H(S^*)}^2.
\end{equation}
By taking $\sqrt{l^2-1} \ge (e-1) C_{\eta} C_p (C_1 + C(k,d,\delta)) \sqrt{\frac{\theta_{k,\max}}{\theta_{k,\min}}}$, the exponential decay naturally follows.
\end{proof}

\subsection{Exponential decay of basis functions II}\label{subsec:k2decay}
The following theorem gives the exponential decay property of $\psi_{i,q}$ for an operator $\CalL$ with lower-order terms. Similar to the proof of Theorem~\ref{thm:expdecay2}, we need the polynomial approximation property~\eqref{eqn:generalPoincare} and the Friedrichs' inequality~\eqref{eqn:generalFriedrich} to bound the lower-order terms, and we get an extra factor of 2 in our error bound. 
\begin{theorem}\label{thm:expdecay2}
Suppose $\CalL u = \sum\limits_{0 \le |\sigma|, |\gamma| \le k} (-1)^{|\sigma|} D^{\sigma} (a_{\sigma \gamma}(x) D^{\gamma} u)$ is self-adjoint. Assume that $a_{\sigma \gamma}(x) \in L^{\infty}(D)$ for all $0 \le |\sigma|, |\gamma| \le k$ and that for any $x \in D$
\begin{itemize}
\item $\CalL$ is nonnegative, i.e.,
	\begin{equation}\label{eqn:Lnonnegative}
		\sum_{ 0 \le |\sigma|, |\gamma| \le k} a_{\sigma \gamma}(x) \vct{\zeta}_{\sigma} \vct{\zeta}_{\gamma} \ge 0, \qquad \forall x \in D , \quad \forall \vct{\zeta} \in \R^{\binom{k+d}{k}},
	\end{equation}
\item $\CalL$ is bounded, i.e., there exist $\theta_{0,\max}\ge 0$ and $\theta_{k,\max} > 0$ such that
	\begin{equation}\label{eqn:Lbounded}
		\sum_{ 0 \le |\sigma|, |\gamma| \le k} a_{\sigma \gamma}(x) \vct{\zeta}_{\sigma} \vct{\zeta}_{\gamma} \le \theta_{k,\max} \sum_{|\sigma| = k} \vct{\zeta}_{\sigma}^2 + \theta_{0,\max} \sum_{|\sigma| < k} \vct{\zeta}_{\sigma}^2, \qquad \forall x \in D , \quad  \forall \vct{\zeta} \in \R^{\binom{k+d}{k}},
	\end{equation}
\item and $\CalL$ is strongly elliptic, i.e., there exists $\theta_{k,\min} > 0$ such that
	\begin{equation}\label{eqn:strongelliptic}
		\sum_{ |\sigma| = |\gamma| = k} a_{\sigma \gamma}(x) \vct{\zeta}_{\sigma} \vct{\zeta}_{\gamma} \ge \theta_{k,\min} \sum_{|\sigma| = k} \vct{\zeta}_{\sigma}^2, \qquad \forall \vct{\zeta} \in \R^{\binom{k+d-1}{k}}.
	\end{equation}
\end{itemize}
Then, there exists $h_0 > 0$ such that for any $h \le h_0$, $1\le i \le m$ and $1\le q\le Q$, it holds true that
\begin{equation}\label{eqn:expdecay2}
	\|\psi_{i,q}\|_{H(D \cap (B(x_i,r))^c)}^2 \le \exp\left(1 - \frac{r}{l h} \right) \|\psi_{i,q}\|_{H(D)}^2
\end{equation}
with $\sqrt{l^2-1} \ge 2 (e-1) C_{\eta} C_p (C_1 + C(k,d,\delta)) \sqrt{\frac{\theta_{k,\max}}{\theta_{k,\min}}}$. Here, $C_1$ and $C_{\eta}$ depend on $k$ and $d$ only, $C_p$ is the constant given in Eqn.~\eqref{eqn:generalPoincare}, $C(k, d, \delta) := C(k, k, d, \delta)$ is given in Lemma~4.1 and $\theta_{k,\max} := \max(\theta_{0,\max}, \theta_{k,\max})$. The constant $h_0$ can be taken as
\begin{equation*}
	h_0 = \sup \left\{h > 0 : \frac{h^2-h^{2k}}{1-h^2} \le \frac{1}{C_p^2}, \frac{h^2(1-h^{2k})}{1-h^2} \le \min\left(\frac{\theta_{k,\max}}{2 \theta_{0,\max} C_f^2}, \frac{\theta_{k,\min}^2}{16 \theta_{0,\max}\theta_{k,\max} C_p^2}\right)\right\},
\end{equation*}
where $C_f$ is the constant in the Friedrichs' inequality~\eqref{eqn:generalFriedrich}.
\end{theorem}

\begin{proof}
The proof follows the same structure as the proof of Theorem~\ref{thm:expdecay2simple}. All we need to do is to use the polynomial approximation property~\eqref{eqn:generalPoincare} and the Friedrichs' inequality~\eqref{eqn:generalFriedrich} to bound the lower-order terms when they appear. First, the $I_1$ in Eqn.~\eqref{eqn:intbyparts2simple} contains all the lower-order terms and its estimation should be modified as follows:
\begin{eqnarray}
	  I_1 &=& - \sum_{0 \le |\sigma|, |\gamma| \le k}\sum_{\stackrel{\sigma_1 + \sigma_2 = \sigma}{|\sigma_1| \ge 1}} \binom{\sigma}{\sigma_1} \int_{S^*} a_{\sigma \gamma}(x) D^{\sigma_1} \eta D^{\sigma_2} \psi_{i,q} D^{\gamma} \psi_{i,q}  \label{eqn:I1_1_1} \\
	 &\le& \left( \sum_{|\sigma| \le k} \int_{S^*} \left| \sum_{\stackrel{\sigma_1 + \sigma_2 = \sigma}{|\sigma_1| \ge 1}} \binom{\sigma}{\sigma_1} D^{\sigma_1} \eta D^{\sigma_2} \psi_{i,q} \right|^2 \right)^{1/2} \|\psi_{i,q}\|_{H(S^*)} \sqrt{\theta_{k,\max}} \label{eqn:I1_1_2}\\
	 &\le& C_1 C_{\eta} \left( \sum_{s=1}^k \sum_{s'=1}^s (l h)^{-2 s'} |\psi_{i,q}|_{s-s', 2, S^*}^2 \right)^{1/2} \|\psi_{i,q}\|_{H(S^*)} \sqrt{\theta_{k,\max}} \label{eqn:I1_1_3}.
\end{eqnarray}
Here, $\theta_{k,\max} := \max(\theta_{0,\max}, \theta_{k,\max})$. We have used the Cauchy--Schwarz inequality and the bound~\eqref{eqn:Lbounded} in Eqn.~\eqref{eqn:I1_1_2}. We will defer the proof of the last step in Eqn.~\eqref{eqn:I1_1_3} to the Appendix. Since $\psi_{i,q} \perp \CalP_{k-1}$ locally in $L^2$, we obtain from Theorem~\ref{thm:generalPoincare} that
\begin{equation*}
	|\psi_{i,q}|_{s-s', 2, S^*} \le C_p h^{s'} |\psi_{i,q}|_{s, 2, S^*} \quad \forall \, 0 \le s' \le s \le k.
\end{equation*}
Therefore, we have
\begin{eqnarray}
	  I_1 &\le& C_1 C_{\eta} \sqrt{\theta_{k,\max}} C_p \left(\sum_{s=1}^k \sum_{s'=1}^s l^{- 2s'} |\psi_{i,q}|_{s, 2, S^*}^2\right)^{1/2} \|\psi_{i,q}\|_{H(S^*)}  \label{eqn:I1_2_1}\\
	 &\le& \frac{C_1 C_{\eta} \sqrt{\theta_{k,\max}} C_p}{\sqrt{l^2-1}} \left(\sum_{s=1}^k |\psi_{i,q}|_{s, 2, S^*}^2\right)^{1/2} \|\psi_{i,q}\|_{H(S^*)} \label{eqn:I1_2_2}\\
	 &\le& \frac{C_1 C_{\eta} \sqrt{2 \theta_{k,\max}} C_p}{\sqrt{l^2-1}} |\psi_{i,q}|_{k,2,S^*} \|\psi_{i,q}\|_{H(S^*)} \label{eqn:I1_2_3}. 
\end{eqnarray}
If we compare the above estimate with Eqn.~\eqref{eqn:I1_2_3simple}, we conclude that Eqn.~\eqref{eqn:I1_2_2} contains all the lower-order terms. We will use the polynomial approximation property~\eqref{eqn:generalPoincare} and take $\frac{h^2-h^{2k}}{1-h^2} \le 1/C_p^2$ to guarantee that Eqn.~\eqref{eqn:I1_2_3} is valid. When $\CalL$ contains lower-order terms, by Lemma~\ref{lem:boundLf2}, we have $\|\CalL \psi_{i,q}\|_{L^2(\tau_j)} \le \sqrt{2 \theta_{k,\max}}C(k, d, \delta) h^{-k} \|\psi_{i,q}\|_{H(\tau_j)}$ for any $h >0$ satisfying $\frac{h^2(1-h^{2k})}{1-h^2} \le \frac{\theta_{k,\max}}{2 \theta_{0,\max} C_f^2}$. Therefore, using Eqn.~\eqref{eqn:I2_2simple} we get
\begin{equation}\label{eqn:I2_2}
	 I_2 \le \frac{\sqrt{2 \theta_{k,\max}} C_{\eta} C(k, d, \delta) C_p}{l} |\psi_{i,q}|_{k,2,S^*} \|\psi_{i,q}\|_{H(S^*)},
\end{equation}
when $h$ satisfies $\frac{h^2(1-h^{2k})}{1-h^2} \le \frac{\theta_{k,\max}}{2 \theta_{0,\max} C_f^2}$. Finally, we need to use Eqn.~\eqref{eqn:kleH2} instead of Eqn.~\eqref{eqn:kleH1} to bound $|\psi_{i,q}|_{k,2,S^*}$. We get
\begin{equation}\label{eqn:recursivek2}
	\|\psi_{i,q}\|_{H(S^1)}^2 \le 2 \sqrt{\frac{\theta_{k,\max}}{(l^2-1) \theta_{k,\min}}} C_{\eta} C_p (C_1 + C(k,d,\delta)) \|\psi_{i,q}\|_{H(S^*)}^2,
\end{equation}
where we have imposed another condition on $h$, i.e., $\frac{h^2(1-h^{2k})}{1-h^2} \le \frac{\theta_{k,\min}^2}{16 \theta_{0,\max}\theta_{k,\max} C_p^2}$.
By taking $\sqrt{l^2-1} \ge 2 (e-1) C_{\eta} C_p (C_1 + C(k,d,\delta)) \sqrt{\frac{\theta_{k,\max}}{\theta_{k,\min}}}$, we prove the exponential decay.
\end{proof}

\begin{remark}\label{rem:expdecay2}
As we have pointed out in Remark~\ref{rem:kleH2}, when $\CalL$ contains lower-order terms but there is no crossing term between $D^{\sigma} u$ ($|\sigma| = k$) and $D^{\sigma} u$ ($|\sigma| < k$), Eqn.~\eqref{eqn:kleH1} can be used to bound $|\psi_{i,q}|_{k,2,S^*}$. In this case, the constraint on $l$ is
\begin{equation*}
	\sqrt{l^2-1} \ge \sqrt{2} (e-1) C_{\eta} C_p (C_1 + C(k,d,\delta)) \sqrt{\frac{\theta_{k,\max}}{\theta_{k,\min}}}
\end{equation*}
and the $h_0$ can be taken as
\begin{equation*}
	h_0 = \sup \left\{h > 0 : \frac{h^2-h^{2k}}{1-h^2} \le \frac{1}{C_p^2}, \frac{h^2(1-h^{2k})}{1-h^2} \le \frac{\theta_{k,\max}}{2 \theta_{0,\max} C_f^2}\right\}.
\end{equation*}
\end{remark}

\subsection{Lemmas}\label{sec:k2lemmas}
In this subsection, we will prove the following lemma, which is used in the proof of Theorem~\ref{thm:expdecay2simple} and Theorem~\ref{thm:expdecay2}. 
\begin{lemma}\label{lem:boundLf2}
$\CalL$ is defined in Eqn.~\eqref{eqn:highorder} and the space $\Psi$ is defined as above. Assume that for any $x \in D$
\begin{equation}\label{eqn:bounded}
	\sum_{ 0 \le |\sigma|, |\gamma| \le k} a_{\sigma \gamma}(x) \vct{\zeta}_{\sigma} \vct{\zeta}_{\gamma} \le \theta_{k,\max} \sum_{|\sigma| = k} \vct{\zeta}_{\sigma}^2 + \theta_{0,\max} \sum_{|\sigma| < k} \vct{\zeta}_{\sigma}^2, \qquad \forall \vct{\zeta} \in \R^{\binom{k+d}{k}}.
\end{equation}
Let $C_f$ be the constant in the Friedrichs' inequality~\eqref{eqn:generalFriedrich}. Then, for any domain partition with $\frac{h^2(1-h^{2k})}{1-h^2} \le \frac{\theta_{k,\max}}{2 \theta_{0,\max} C_f^2}$, we have
\begin{equation}\label{eqn:invenergy2}
	\|\CalL v\|_{L^2(\tau_j)} \le \sqrt{2 \theta_{k,\max}} C(k, d, \delta) h^{-k} \|v\|_{H(\tau_j)} \quad \forall v \in \Psi, \, \forall j = 1, 2, \ldots, m,
\end{equation}
where $C(k, d, \delta) = C(k,k,d,\delta)$ from Lemma~4.1. 

If the operator $\CalL$ contains only the highest order terms, i.e., $\CalL u = (-1)^k \sum\limits_{|\sigma|=|\gamma|=k} D^{\sigma}(a_{\sigma \gamma} D^{\gamma} u)$, we have $\|\CalL v\|_{L^2(\tau_j)} \le \sqrt{\theta_{k,\max}} C(k,d,\delta) h^{-k} \|v\|_{H(\tau_j)}$ for all $h > 0$.
\end{lemma}

We will use Lemma~4.1 to prove this result, but we need to deal with the variable coefficients $a_{\sigma \gamma}$ and the lower-order terms $a_{\sigma \gamma}$ with $|\sigma|+|\gamma| < 2 k$ before we can apply Lemma~4.1. Our strategy is to transfer the variable coefficients to constant ones by the variational formulation (see Lemma~\ref{lem:compare12}), and to use the polynomial approximation property to deal with the lower-order terms; see Lemma~\ref{lem:compare32}. For this purpose, we first introduce the following two lemmas.
\begin{lemma}\label{lem:compare12}
Let $\Omega$ be a smooth, bounded, open subset of $\R^d$. $\CalL u = \sum\limits_{0 \le |\sigma|, |\gamma| \le k} (-1)^{|\sigma|} D^{\sigma} (a_{\sigma \gamma}(x) D^{\gamma} u)$ and $\CalM u = \sum\limits_{0 \le |\sigma|, |\gamma| \le k} (-1)^{|\sigma|} D^{\sigma} (b_{\sigma \gamma}(x) D^{\gamma} u)$ are two symmetric operators on $H_0^k(\Omega)$. Moreover, we assume that the bilinear forms induced by both $\CalL$ and $\CalM$ are equivalent to the standard norm on $H_0^k(\Omega)$. Let $G_{\CalL}$ and $G_{\CalM}$ be the Green's functions of $\CalL$ and $\CalM$ respectively. If for any $x \in D$ we have
\begin{equation}\label{eqn:majorizek2}
	 \sum_{ 0 \le |\sigma|, |\gamma| \le k} a_{\sigma \gamma}(x) \vct{\zeta}_{\sigma} \vct{\zeta}_{\gamma} \le \sum_{ 0 \le |\sigma|, |\gamma| \le k} b_{\sigma \gamma}(x) \vct{\zeta}_{\sigma} \vct{\zeta}_{\gamma} \qquad \forall \vct{\zeta} \in \R^{\binom{k+d}{k}}.
\end{equation}
 then for all $f \in L^2(\Omega)$, 
\begin{equation}\label{eqn:compare12}
	\int_{\Omega} \int_{\Omega} G_{\CalM}(x,y) f(x) f(y) \rd x\, \rd y \le \int_{\Omega} \int_{\Omega} G_{\CalL}(x,y) f(x) f(y) \rd x\, \rd y.
\end{equation}
\end{lemma}
\begin{proof}
Let $f\in L^2(\Omega)$. Let $\psi_{\CalL}$ and $\psi_{\CalM}$ be the weak solutions of $\CalL \psi_{\CalL} = f$ and $\CalM \psi_{\CalM} = f$ with the homogeneous Dirichlet boundary conditions on $\partial \Omega$. Observe that $\psi_{\CalL}$ and $\psi_{\CalM}$ are the unique minimizers of $I_{\CalL}(u, f)$ and $I_{\CalM}(u, f)$ with 
\begin{equation}\label{eqn:variationalvark2}
\begin{split}
	I_{\CalL}(u, f) &= \frac{1}{2} \sum_{0 \le |\sigma|, |\gamma| \le k} \int_D a_{\sigma \gamma}(x) D^{\sigma} u D^{\gamma} u  - \int_{\Omega} u f, \quad u \in H_0^k(\Omega),\\
	I_{\CalM}(u, f) &= \frac{1}{2} \sum_{0 \le |\sigma|, |\gamma| \le k} \int_D b_{\sigma \gamma}(x) D^{\sigma} u D^{\gamma} u  - \int_{\Omega} u f, \quad u \in H_0^k(\Omega).
\end{split}	
\end{equation}
At the minima $\psi_{\CalL}$ and $\psi_{\CalM}$, we have
\begin{equation} \label{eqn:minimalvark2}
\begin{split}
	I_{\CalL}(\psi_{\CalL}, f) &= - \frac{1}{2} \int_{\Omega} \psi_{\CalL} f = -\frac{1}{2}\sum_{0 \le |\sigma|, |\gamma| \le k} \int_D a_{\sigma \gamma}(x) D^{\sigma} \psi_{\CalL} D^{\gamma} \psi_{\CalL},\\
	I_{\CalM}(\psi_{\CalM}, f) &= - \frac{1}{2} \int_{\Omega} \psi_{\CalM} f = -\frac{1}{2}\sum_{0 \le |\sigma|, |\gamma| \le k} \int_D a_{\sigma \gamma}(x) D^{\sigma} \psi_{\CalM} D^{\gamma} \psi_{\CalM}.
\end{split}	
\end{equation}
Observe that
\begin{equation}\label{eqn:variationalcomparek2}
	I_{\CalL}(\psi_{\CalL},f) \le I_{\CalL}(\psi_{\CalM}, f) \le I_{\CalM}(\psi_{\CalM}, f), 
\end{equation}
where the first inequality is true because $\psi_{\CalL}$ is the minimizer of $I_{\CalL}$, and the second inequality is true because $I_{\CalL}(u,f)\le I_{\CalM}(u,f)$ for any $u \in H_0^k(\Omega)$. Combining Eqn.~\eqref{eqn:minimalvark2} and \eqref{eqn:variationalcomparek2}, we obtain $\int_{\Omega} \psi_{\CalM} f \le \int_{\Omega} \psi_{\CalL} f$. This proves the lemma.
\end{proof}

\begin{lemma}\label{lem:compare32}
Let $\Omega_h$ be a smooth, convex, bounded, open subset of $\R^d$ with diameter at most $h$. Let $G_{h}$ be the Green's function of $\CalL u = (-1)^k \sum_{|\sigma| = k} D^{2 \sigma} u + c \sum_{|\sigma| < k} (-1)^{\sigma} D^{2 \sigma} u$ with the homogeneous Dirichlet boundary condition on $\partial \Omega_h$ and $G_{h,0}$ be the Green's function of $\CalL_0 u = (-1)^k \sum_{|\sigma| = k} D^{2 \sigma} u$ with the homogeneous Dirichlet boundary condition on $\partial \Omega_h$. Here, $c > 0$ is a positive constant. Then, for any $f\in L^2(\Omega_h)$
\begin{equation}\label{eqn:compare322}
	\lim_{h \to 0} \frac{\int_{\Omega_h} \int_{\Omega_h} G_{h}(x,y) f(x) f(y) \rd x\, \rd y}{\int_{\Omega_h} \int_{\Omega_h} G_{h,0}(x,y) f(x) f(y)\rd x\, \rd y} = 1.
\end{equation}
Moreover, $\frac{\int_{\Omega_h} \int_{\Omega_h} G_{h}(x,y) f(x) f(y) \rd x\, \rd y}{\int_{\Omega_h} \int_{\Omega_h} G_{h,0}(x,y) f(x) f(y)\rd x\, \rd y} \ge 1/2$ for all $h>0$ such that $\frac{h^2(1-h^{2k})}{1-h^2} \le \frac{1}{2 c C_f^2}$.
\end{lemma}
\begin{proof}
Let $\psi_{h}$ be the solution of $\CalL \psi_{h} = f$ with the homogeneous Dirichlet boundary conditions on $\partial \Omega_h$ and $\psi_{h,0}$ be the solution of $\CalL_0 \psi_{h,0} = f$ with the homogeneous Dirichlet boundary conditions on $\partial \Omega_h$. Let
\begin{equation}\label{eqn:variational3k2}
\begin{split}
	I_{\CalL}(u, f) &= \frac{1}{2} |u|_{k,2,\Omega_h}^2 +  \frac{c}{2}\|u\|_{k-1,2,\Omega_h}^2 - \int_{\Omega_h} u f, \\
	I_{\CalL_0}(u, f) &=  \frac{1}{2} |u|_{k,2,\Omega_h}^2  - \int_{\Omega_h} u f.
\end{split}
\end{equation}
At the minima $\psi_{h}$ and $\psi_{h,0}$, we have
\begin{equation} \label{eqn:minimal3k2}
\begin{split}
	I_{\CalL}(\psi_{h}, f) &= - \frac{1}{2} \int_{\Omega_h} \psi_{h} f = -\frac{1}{2}\left( |\psi_{h}|_{k,2,\Omega_h}^2 + c \|\psi_{h}\|_{k-1,2,\Omega_h}^2 \right),\\
	I_{\CalL_0}(\psi_{\CalL_0}, f) &= - \frac{1}{2} \int_{\Omega_h} \psi_{h,0} f = -\frac{1}{2} |\psi_{h,0}|_{k,2,\Omega_h}^2.
\end{split}	
\end{equation}
Note that Eqn.~\eqref{eqn:minimal3k2} implies that $I_{\CalL_0}(\psi_{h,0}, f) < 0$. By the definition of Green's function, we further have
\begin{equation}\label{eqn:minimal32k2}
\begin{split}
	\int_{\Omega_h} \int_{\Omega_h} G_{h}(x,y) f(x) f(y) \rd x\, \rd y &= \int_{\Omega_h} \psi_h f =-2 I_{\CalL}(\psi_h, f) = |\psi_{h}|_{k,2,\Omega_h}^2 + c \|\psi_{h}\|_{k-1,2,\Omega_h}^2,\\
	\int_{\Omega_h} \int_{\Omega_h} G_{h,0}(x,y) f(x) f(y) \rd x\, \rd y &= \int_{\Omega_h} \psi_{h,0} f =-2 I_{\CalL_0}(\psi_{h,0}, f)= |\psi_{h,0}|_{k,2,\Omega_h}^2.
\end{split}
\end{equation}
Since $I_{\CalL_0}(u, f) \le I_{\CalL}(u, f)$ for any $u \in H_0^k(\Omega)$, we have $\frac{\int_{\Omega_h} \int_{\Omega_h} G_{h}(x,y)f(x) f(y) \rd x\, \rd y}{\int_{\Omega_h} \int_{\Omega_h} G_{h,0}(x,y) f(x) f(y) \rd x\, \rd y} \le 1$ for any $h > 0$. Applying the Friedrich's inequality~\eqref{eqn:generalFriedrich} to $\|\psi_{h,0}\|_{k-1,2,\Omega_h}^2$, we get
\begin{equation*}
\begin{split}
	-2 I_{\CalL}(\psi_{h,0}, f) &\ge -2 I_{\CalL_0}(\psi_{h,0}, f) - \frac{c C_f^2 h^2(1-h^{2k})}{1-h^2} |\psi_{h,0}|_{k,2,\Omega_h}^2 \\
	& = - 2 \left(1 - \frac{c C_f^2 h^2(1-h^{2k})}{1-h^2} \right)) I_{\CalL_0}(\psi_{h,0}, f).
\end{split}
\end{equation*}
Here, we have used Eqn.~\eqref{eqn:minimal32k2} in the last equality. Therefore, we have
\begin{equation*}
	\frac{\int_{\Omega_h} \int_{\Omega_h} G_{h}(x,y) f(x) f(y)  \rd x\, \rd y}{\int_{\Omega_h} \int_{\Omega_h} G_{h,0}(x,y) f(x) f(y) \rd x\, \rd y} = \frac{-2 I_{\CalL}(\psi_h, f)}{-2 I_{\CalL_0}(\psi_{h,0}, f)} \ge \frac{-2 I_{\CalL}(\psi_{h,0}, f)}{-2 I_{\CalL_0}(\psi_{h,0}, f)} \ge 1 - \frac{c C_f^2 h^2(1-h^{2k})}{1-h^2},
\end{equation*}
where we have used $I_{\CalL}(\psi_{h}, f) \le I_{\CalL}(\psi_{h,0}, f)$ in the first inequality. By using the above upper bound, we prove the lemma.
\end{proof}

Now, we are ready to prove Lemma~\ref{lem:boundLf2}.
\begin{proof}[\bf{Proof of Lemma~\ref{lem:boundLf2}}]
Let $v = \sum_{i=1}^m \sum_{q=1}^{Q} c_{i,q} \psi_{i,q}$. Thanks to Eqn.~\eqref{eqn:psi1}, we have
$$\CalL v = \sum_{i,q} \sum_{j,q'} c_{i,q} \Theta_{iq,jq'}^{-1} \phi_{j,q'}.$$
Let $g_{j} = \sum_{q'=1}^{Q} \sum_{i,q} c_{i,q} \Theta_{iq,jq'}^{-1} \phi_{j,q'}$. Due to the construction of $\phi_{j,q'}$, we have
\begin{equation}\label{eqn:Lvlocal2}
	\|\CalL v\|_{L^2(\tau_j)}^2 = \| g_j \|_{L^2(\tau_j)}^2
\end{equation}
Furthermore, $v$ can be decomposed over $\tau_j$ as $v = v_1 + v_2$, where $v_1$ solves $\CalL v_1 = g_j(x)$ in $\tau_j$ with $v_1 \in H_0^k(\tau_j)$, and $v_2$ solves $\CalL v_2 = 0$ with $v_2 - v \in H_0^k(\tau_j)$. It is easy to check that $\|v\|_{H(\tau_j)}^2 = \|v_1\|_{H(\tau_j)}^2 + \|v_2\|_{H(\tau_j)}^2$. We denote $G_j$ as the Green's function of the operator $\CalL$ with the homogeneous Dirichlet boundary condition on $\tau_j$, then
\begin{equation*}
	\|v_1\|_{H(\tau_j)}^2 = \int_{\tau_j} v_1(x) g_j \rd x = \int_{\tau_j} \int_{\tau_j} G_j(x,y) g_j(x) g_j(y)\rd x\, \rd y.
\end{equation*}
Thanks to Lemma~\ref{lem:compare12}, we have
\begin{equation}\label{eqn:normv1boundk2}
	\|v_1\|_{H(\tau_j)}^2 \ge \frac{1}{\theta_{k,\max}} \int_{\tau_j} \int_{\tau_j} G_j^*(x,y)g_j(x) g_j(y) \rd x\, \rd y,
\end{equation}
where $G_j^*$ is the Green's function of the operator $(-1)^k \sum\limits_{|\sigma| = k} D^{2 \sigma} u + \frac{\theta_{k,\max}}{\theta_{0,\max}} \sum\limits_{|\sigma| < k} (-1)^{\sigma} D^{2 \sigma} u$ with the homogeneous Dirichlet boundary condition on $\partial \tau_j$. Thanks to Lemma~\ref{lem:compare32}, for all $h>0$ such that $\frac{h^2(1-h^{2k})}{1-h^2} \le \frac{\theta_{k,\max}}{2 \theta_{0,\max} C_f^2}$ we have
\begin{equation}\label{eqn:nonzerock2}
	\int_{\tau_j} \int_{\tau_j} G_j^*(x,y) g_j(x) g_j(y) \rd x\, \rd y \ge \frac{1}{2}\int_{\tau_j} \int_{\tau_j} G_{j,0}^*(x,y) g_j(x) g_j(y) \rd x\, \rd y,
\end{equation}
where $G_{j,0}^*$ is the Green's function of the operator $(-1)^k \sum_{|\sigma| = k} D^{2 \sigma} u$ with the homogeneous Dirichlet boundary condition on $\partial \tau_j$. Denote $v_{1,0}$ as the solution of $(-1)^k \sum_{|\sigma| = k} D^{2 \sigma} v_{1,0} = g_j$ on $\tau_j$ with the homogeneous Dirichlet boundary condition, i.e., $v_{1,0}(x) = \int_{\tau_j} G_{j,0}^*(x,y) g_j(y) \rd y$. Since $g_j \in \CalP_{k-1}$ in $\tau_j$ in this case, Lemma~4.1 shows that
\begin{equation}\label{eqn:citelemmak2}
	\|g_j\|_{L^2(\tau_j)}^2 \le \left(C(k,k,d,\delta)\right)^2 h^{-2} \int_{\tau_j} \int_{\tau_j} G_{j,0}^*(x,y) g_j(x) g_j(y) \rd x\, \rd y.
\end{equation}
Combining Eqn.~\eqref{eqn:normv1boundk2}, \eqref{eqn:nonzerock2} and \eqref{eqn:citelemmak2}, we have
\begin{equation*}
	\|g_j\|_{L^2(\tau_j)}^2 \le 2 \left(C(k,k,d,\delta)\right)^2 h^{-2k} \theta_{k,\max} \|v_1\|_{H(\tau_j)}^2 \le 2 \left(C(k,k,d,\delta)\right)^2 h^{-2k} \theta_{k,\max} \|v\|_{H(\tau_j)}^2.
\end{equation*}
Therefore, we have proved Lemma~\ref{lem:boundLf2}. We point out that when the operator $\CalL$ contains only the highest order terms, i.e., $\CalL u = (-1)^k \sum\limits_{|\sigma|=|\gamma|=k} D^{\sigma}(a_{\sigma \gamma} D^{\gamma} u)$, we don't need to pay a factor of $2$ in Eqn.~\eqref{eqn:nonzerock2}, and thus, $\|g_j\|_{L^2(\tau_j)}^2 \le  \left(C(k,k,d,\delta)\right)^2 h^{-2k} \theta_{k,\max} \|v\|_{H(\tau_j)}^2$ for all $h > 0$ in this special case.
\end{proof}

\referee{
Let $\CalL_0^{-1} f \in H_{0}^k(\tau_i)$ be the unique weak solution of the following elliptic equation with the homogeneous Dirichlet boundary condition 
\begin{equation}\label{eqn:localelliptic}
	\CalL u = f(x) \qquad x \in \tau_i, \quad u \in H_{0}^k(\tau_i).
\end{equation}
We define $M_0, A_0 \in \R^{Q \times Q}$ as follows:
\begin{equation}\label{def:M0A0}
	M_0(q, q') = \int_{\tau_i} \phi_{i,q}\phi_{i,q'}, \qquad A_0(q, q') = \int_{\tau_i} \phi_{i,q} \CalL_0^{-1} (a\phi_{i,q'}).
\end{equation}
Let $\lambda_{\max}(M_0, A_0)$ be the largest generalized eigenvalue of the eigenvalue problem $M_0 \alpha = \lambda A_0 \alpha$, which can be written as
\begin{equation}\label{eqn:varM0A0}
	\lambda_{\max}(M_0, A_0) = \sup_{v\in \R^Q} \frac{v^T M_0 v}{v^T A_0 v} = \sup_{\phi \in \CalP_k(\tau_i)} \frac{\|\phi\|_{L^2(\tau_i)}^2}{ \|\CalL_0^{-1}\phi\|_{H(\tau_i)}^2 } .
\end{equation}
The proof of Lemma~\ref{lem:boundLf2} also implies that
 \begin{equation}\label{eqn:boundM0A0}
	\sqrt{\lambda_{\max}(M_0, A_0)} \le \sqrt{2 \theta_{k,\max}} C(k,d,\delta) h^{-k}.
\end{equation}
If the operator $\CalL$ contains only the highest order terms, we have
 \begin{equation}\label{eqn:bound2M0A0}
	\sqrt{\lambda_{\max}(M_0, A_0)} \le \sqrt{\theta_{k,\max}} C(k,d,\delta) h^{-k}.
\end{equation}
}

\section{Localization of the basis functions}\label{sec:localization}
Theorem~\ref{thm:expdecay1} or Theorem~\ref{thm:expdecay2} allows us to localize the construction of basis functions $\psi_{i,q}$ as follows. For $r > 0$, let $S_r$ be the union of the subdomains $\tau_j$ that intersect with $B(x_i, r)$ (recall that $B(x_i, \delta h_i/2) \subset \tau_i$) and let $\psi_{i,q}^{\mathrm{loc}}$ be the minimizer of the following quadratic problem:
\begin{equation}\label{eqn:localVar}
\begin{split}
	\psi_{i,q}^{\mathrm{loc}} = \argmin_{\psi \in H_0^k(S_r)} \quad & \|\psi\|_H^2 \\
	\text{s.t.} \quad & \int \phi_{j,q'} \psi = \delta_{iq,jq'} \quad \forall 1 \le j \le m,\quad \forall 1 \le q' \le Q.
\end{split}
\end{equation}
We will naturally identify $\psi_{i,q}^{\mathrm{loc}}$ with its extension to $H_0^k(D)$ by setting $\psi_{i,q}^{\mathrm{loc}} = 0$ outside of $S_r$. 

If the elliptic operator $\CalL$ is given with some other homogeneous boundary condition, the localized problem~\eqref{eqn:localVar} should be slightly modified as follows such that the basis function $\psi_{i,q}$ honors the given boundary condition on $\partial D$:
\begin{equation}\label{eqn:localVarOtherBC}
\begin{split}
	\psi_{i,q}^{\mathrm{loc}} = \argmin_{\psi \in H} \quad & \|\psi\|_H^2 \\
	\text{s.t.} \quad & \int \phi_{j,q'} \psi = \delta_{iq,jq'} \quad \forall 1 \le j \le m, \quad \forall 1 \le q' \le Q, \\
				& \psi(x) \equiv 0 \quad \forall x \in D\backslash S_r.
\end{split}
\end{equation}
When $\partial S_r \cap \partial D = \emptyset$,  Eqn.~\eqref{eqn:localVarOtherBC} is equivalent to Eqn.~\eqref{eqn:localVar}. However, when $\partial S_r \cap \partial D \neq \emptyset$, Eqn.~\eqref{eqn:localVarOtherBC} only enforces the zero Dirichlet boundary condition on $\partial S_r \backslash \partial D$, but honors the original boundary condition on $\partial D$. 

From now on, to simplify the expression of constants, we will assume without loss of generality that the domain is rescaled so that $\text{diam}(D) \le 1$.

\referee{
\begin{lemma}\label{lem:phinorm}
For any domain partition with $\frac{h^2(1-h^{2k})}{1-h^2} \le \frac{\theta_{k,\max}}{2 \theta_{0,\max} C_f^2}$, it holds true that 
\begin{equation}\label{eqn:phinorm}
	\|\psi_{i,q}^{\mathrm{loc}}\|_H \le C(k,d,\delta) \left(\frac{2^{d+1} \theta_{k,\max}}{V_d \delta^{d}} \right)^{1/2} h^{-d/2 - k}.
\end{equation}
If the operator $\CalL$ contains only the highest order terms, it holds true that $\|\psi_{i,q}^{\mathrm{loc}}\|_H \le C(k,d,\delta) \left(\frac{2^{d} \theta_{k,\max}}{V_d \delta^{d}} \right)^{1/2} h^{-d/2 - k}$ for any $h>0$.
\end{lemma}
\begin{proof}
Consider
\begin{equation*}
	\zeta_{i,q} = \sum_{q=1}^Q A_0^{-1}(q,q') \CalL_0^{-1} \phi_{i,q'},
\end{equation*}
where $ A_0^{-1}$ is the inverse of $A_0$ (defined in Eqn.~\eqref{eqn:boundM0A0}) and $\CalL_0^{-1} \phi_{i,q'}$ is the weak solution of the local problem~\eqref{eqn:localelliptic} with right-hand side $\phi_{i,q'}$. From the definition of $A_0$, we know that $\int_{\tau_i} \phi_{i,q}\zeta_{i,q'} = \delta_{q,q'}$. Notice that $\zeta_{i,q} \in H_{0}^{k} \subset H_{0}^k(S_r)$. Therefore, $\zeta_{i,q}$ satisfies all constraints of $\psi_{i,q}^{\mathrm{loc}}$ (see Eqn.~\eqref{eqn:localVar}), and thus, 
\begin{equation}\label{eqn:psizeta}
	\|\psi_{i,q}^{\mathrm{loc}}\|_H \le \|\zeta_{i,q}\|_H.
\end{equation}
Making use of $(\CalL_0^{-1} \phi_{i,q}, \CalL_0^{-1} \phi_{i,q'})_H = \int_{\tau_i} \phi_{i,q} \CalL_0^{-1} \phi_{i,q'} = A_0(q,q')$, we obtain
\begin{equation} \label{def:zetanorm}
	\|\zeta_{i,q}\|_H^2 = A_0^{-1}(q,q) \le \lambda_{\max}(A_0^{-1}) = \frac{\lambda_{\max}(M_0, A_0)}{|\tau_i|}.
\end{equation}
We have used $M_0(q,q') = |\tau_i| \delta_{i,j}$ (due to the normalization~\eqref{eqn:phinormalize}) in the last inequality. Combining Eqn.~\eqref{eqn:bound2M0A0} (or \eqref{eqn:boundM0A0}), \eqref{eqn:psizeta} and \eqref{def:zetanorm} and $|\tau_i| \ge V_d (\delta h/2)^d$, we complete the proof of Eqn.~\eqref{eqn:phinorm}.
\end{proof}
}

\begin{theorem}\label{thm:localization}
Under the same assumptions as those in Theorem~\ref{thm:expdecay2}, there exists $h_0 > 0$ such that for any $h \le h_0$, $1\le i \le m$ and $1\le q\le Q$, it holds true that
\begin{equation}\label{eqn:localization}
	\|\psi_{i,q} - \psi_{i,q}^{\mathrm{loc}}\|_{H(D)} \le C_3 h^{-d/2 - k} \exp\left(-\frac{r - 2h}{2 l h} \right),
\end{equation}
where 
$$C_3 = C(k,d,\delta) \left(\frac{e 2^{d+1} \theta_{k,\max}}{V_d \delta^{d}} \right)^{1/2} \left( \left( 2 C_1 C_{\eta} C_p \sqrt{\frac{ k \theta_{k,\max}}{\theta_{k,\min}}} + 1 \right)^2 + 2 \sqrt{\frac{ \theta_{k,\max}}{\theta_{k,\min}}} C(k, d, \delta) C_p \right)^{1/2}.$$
Here, all the parameters are the same as those in Theorem~\ref{thm:expdecay2}. 

When the operator $\CalL$ contains only the highest order terms, i.e., $\CalL u = (-1)^k \sum\limits_{|\sigma|=|\gamma|=k} D^{\sigma}(a_{\sigma \gamma} D^{\gamma} u)$, Eqn.~\eqref{eqn:localization} holds true for all $h > 0$. In this case, the constant $C_3$ can be taken as 
$$C_3 = C(k,d,\delta) \left(\frac{e 2^{d} \theta_{k,\max}}{V_d \delta^{d}} \right)^{1/2} \left( \left( C_1 C_{\eta} C_p \sqrt{\frac{k \theta_{k,\max}}{\theta_{k,\min}}} + 1 \right)^2 + \sqrt{\frac{\theta_{k,\max}}{\theta_{k,\min}}} C(k, d, \delta) C_p \right)^{1/2}.$$
\end{theorem}
\begin{proof}
Let $S_0$ be the union of the subdomains $\tau_j$ that are not contained in $S_r$ and let $S_1$ be the union of the subdomains $\tau_j$ that are at distance at least $h$ from $S_0$. (We will assume that $S_0 \neq \emptyset$ and $S_1 \neq \emptyset$. If $S_0 \neq \emptyset$, the proof is trivial. We can choose $r \ge 2h$ such that $S_1 \neq \emptyset$.) Let $S^*$ be the union of the subdomains $\tau_j$ that are not contained in either $S_0$ or $S_1$, as illustrated in Figure~\ref{fig:illustrate2}. Note that in this case, we have $S_1$ in the inner region and $S_0$ in the outer region. This is the opposite of the scenario that we consider in Figure \ref{fig:illustrate1}.
\begin{figure}[ht]
\centering
\includegraphics[width = 0.4\textwidth]{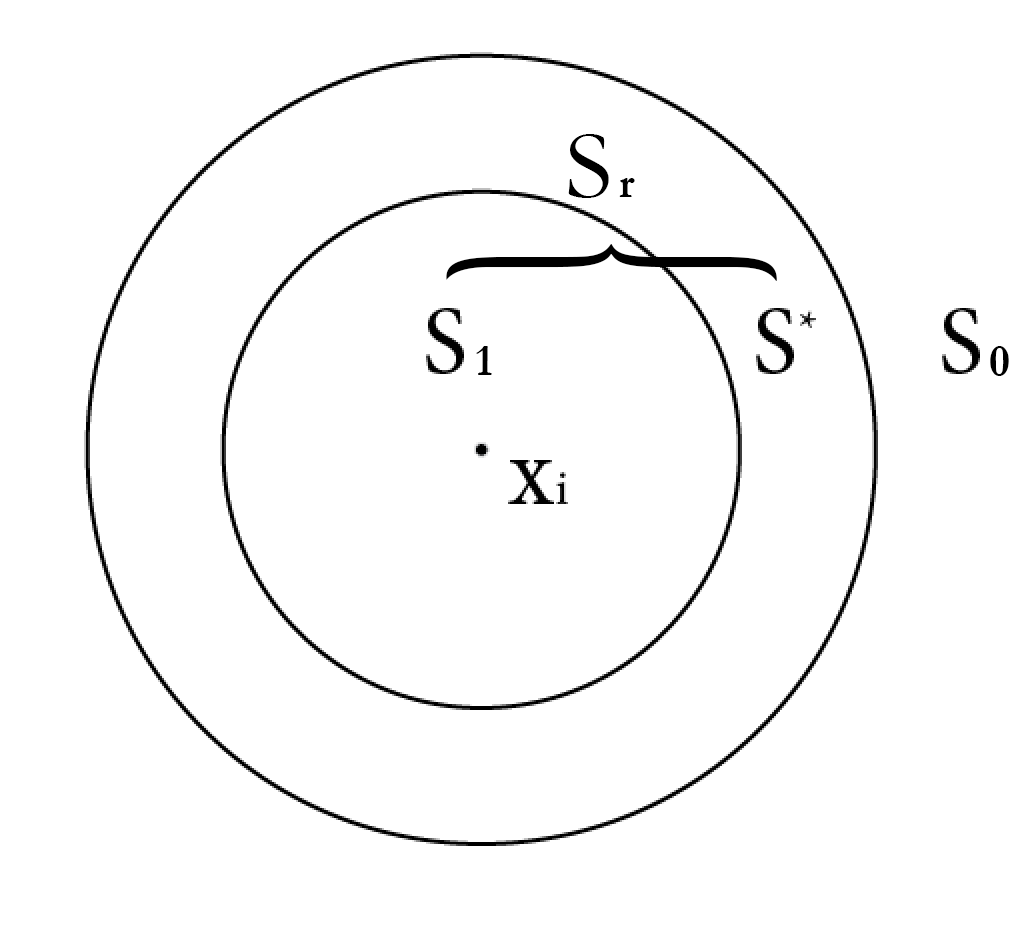}
\caption{Illustration of $S_r$, $S_0$, $S_1$ and $S^*$.}\label{fig:illustrate2}
\end{figure}

\noindent
Let $\eta$ be a smooth cut-off function such that $0 \le \eta \le 1$, $\eta|_{S_1} \equiv 1$, $\eta|_{S_0} \equiv 0$ and $\|D^{\sigma} \eta\|_{L^{\infty}(D)} \le \frac{C_{\eta}}{h^{|\sigma|}}$ for all $\sigma$. 
Since $\psi_{i,q}^{\mathrm{loc}}$ satisfies the same constraints as those in the definition of $\psi_{i,q}$, thanks to Eqn.~\eqref{eqn:psiortho} we have
\begin{equation}\label{eqn:Hdif1}
	\|\psi_{i,q} - \psi_{i,q}^{\mathrm{loc}}\|_{H(D)}^2 = \|\psi_{i,q}^{\mathrm{loc}}\|_{H(D)}^2 - \|\psi_{i,q}\|_{H(D)}^2.
\end{equation}
Define $\psi_{j,q}^{i,r}$ as the (unique) minimizer of the following quadratic optimization:
\begin{equation}\label{eqn:defphikt}
\begin{split}
	\psi_{j,q}^{i,r} := \argmin_{\psi \in H_0^k(S_r)} \quad & \|\psi\|_{H(S_r)}^2 \\
	\text{s.t.} \quad & \int_{S_r} \psi \phi_{j',q'} = \delta_{jq, j'q'} \quad \forall 1\le j' \le m, \quad \forall 1 \le q' \le Q.
\end{split}
\end{equation}
Note that $\psi_{i,q}^{\mathrm{loc}} = \psi_{i,q}^{i,r}$. Let $w_{jq'} = \int_D \eta \psi_{i,q} \phi_{j,q'}$ and $\psi_{w}^{iq,r} = \sum_{j=1}^m \sum_{q'=1}^Q w_{jq'} \psi_{j,q'}^{i,r}$. Thanks to the orthogonality between $\psi_{i,q}$ and $\phi_{j,q'}$, i.e., the constraints in Eqn.~\eqref{eqn:psivariationalk2}, we have
\begin{equation*}
	\psi_{w}^{iq,r} = \psi_{i,q}^{\mathrm{loc}} + \sum_{\tau_j \subset S^*} \sum_{q'=1}^Q w_{jq'} \psi_{j,q'}^{i,r}.
\end{equation*}
Using (3) of Theorem~\ref{thm:variational}, we have $(\psi_{i,q}^{\mathrm{loc}}, \psi_{j,q'}^{i,r})_H = \Theta_{iq, jq'}^{i,-1}$, where $\Theta^{i}$ is defined by Eqn.~\eqref{def:Theta} with $\CalK: L^2(S_r) \to L^2(S_r)$ being the inverse of $\CalL$ with the homogeneous Dirichlet boundary condition on $\partial S_r$. Therefore, we have
\begin{equation}\label{eqn:Hpsiw}
	\|\psi_{w}^{iq,r}\|_H^2 = \|\psi_{i,q}^{\mathrm{loc}}\|_H^2 + \|\sum_{\tau_j \subset S^*} \sum_{q'=1}^Q w_{jq'} \psi_{j,q'}^{i,r}\|_H^2 + 2 \sum_{\tau_j \subset S^*} \sum_{q'=1}^Q w_{jq'} \Theta_{iq, jq'}^{i,-1}.
\end{equation}
By (2) of Theorem~\ref{thm:variational}, we know that $\psi_{w}^{iq,r}$ is the minimizer of the following quadratic problem:
\begin{equation}\label{eqn:defphiw}
\begin{split}
	\psi_{w}^{iq,r} = \argmin_{\psi \in H_0^k(S_r)} \quad & \|\psi\|_{H(S_r)}^2 \\
	\text{s.t.} \quad & \int_{S_r} \psi \phi_{j,q'} = \int_{D} \eta \psi_{i,q} \phi_{j,q'} \quad \forall 1 \le j \le m, \quad \forall 1 \le q' \le Q.
\end{split}
\end{equation}
Noting that $\eta \psi_{i,q}$ satisfies the same constraint, we have $\|\psi_{w}^{iq,r}\|_H^2 \le \|\eta \psi_{i,q}\|_H^2$. By using this estimate with \eqref{eqn:Hdif1} and \eqref{eqn:Hpsiw}, we obtain
\begin{equation}\label{eqn:Hdif2}
	\|\psi_{i,q} - \psi_{i,q}^{\mathrm{loc}}\|_{H(D)}^2 \le \underbrace{\|\eta \psi_{i,q}\|_H^2 - \|\psi_{i,q}\|_H^2}_{I_1} + \underbrace{2 \left|\sum_{\tau_j \subset S^*} \sum_{q'=1}^Q w_{jq'} \Theta_{iq, jq'}^{i,-1}\right|}_{I_2}.
\end{equation}
It turns out that $I_1$ and $I_2$ play almost the same role as $I_1$ and $I_2$ did in the proof of Theorem~\ref{thm:expdecay2} and can be estimated in a similar way. We will estimate these two terms as follows.

%% I_1
Let's first deal with $I_1$. Since $\eta|_{S_1} \equiv 1$ and $\eta|_{S_0} \equiv 0$, we have $I_1 = \|\eta \psi_{i,q}\|_{H(S^*)}^2 - \|\psi_{i,q}\|_{H(S^*\cup S_0)}^2 \le \|\eta \psi_{i,q}\|_{H(S^*)}^2$. In Appendix~\ref{subapp:l1_localization}, we give a bound for $\|\eta \psi_{i,q}\|_{H(S^*)}$ using a similar technique that we used to obtain Eqn.~\eqref{eqn:I1_2_3} from Eqn.~\eqref{eqn:I1_1_1} in the proof of Theorem~\ref{thm:expdecay2}. With this bound, we obtain
\begin{equation}\label{eqn:estimateI1_1}
	I_1 \le \left(\frac{C_3}{2} |\psi_{i,q}|_{k,2,S^*} + \sqrt{\frac{C_3^2}{4} |\psi_{i,q}|_{k,2,S^*}^2 + C_3 |\psi_{i,q}|_{k,2,S^*} \|\psi_{i,q}\|_{H(S^*)} + \|\psi_{i,q}\|_{H(S^*)}^2} \right)^2,
\end{equation}
where $C_3 = C_1 C_{\eta} C_p \sqrt{2 k \theta_{k,\max}}$. With the strong ellipticity~\eqref{eqn:strongelliptic} and the bound~\eqref{eqn:kleH2}, we conclude
\begin{equation}\label{eqn:estimateI1_2}
	I_1 \le \left( 2 C_1 C_{\eta} C_p \sqrt{\frac{k \theta_{k,\max}}{\theta_{k,\min}}} + 1 \right)^2 \|\psi_{i,q}\|_{H(S^*)}^2.
\end{equation}
Applying the exponential decay of Theorem~\ref{thm:expdecay2} to $\|\psi_{i,q}\|_{H(S^*)}$, we get
\begin{equation}\label{eqn:estimateI1_3}
	I_1 \le \left( 2 C_1 C_{\eta} C_p \sqrt{\frac{k \theta_{k,\max}}{\theta_{k,\min}}} + 1 \right)^2 e^{1-\frac{r-2h}{lh}} \|\psi_{i,q}\|_{H(D)}^2.
\end{equation}

%% I_2
We now estimate $I_2$. Combining (3) of Theorem~\ref{thm:variational} with the definition of $H$-norm~\eqref{def:Hinner}, we have
\begin{equation*}
	\Theta_{iq,jq'}^{i,-1} = (\psi_{i,q}^{\mathrm{loc}}, \psi_{j,q'}^{i,r})_{H(S_r)} = (\CalL \psi_{i,q}^{\mathrm{loc}}, \psi_{j,q'}^{i,r})_{L^2(S_r)}.
\end{equation*}
Thanks to $\CalL \psi_{i,q}^{\mathrm{loc}}\mid_{\tau_j} \in \text{span}\{\phi_{j,q'}\}_{q=1}^Q$ and the orthogonality between $\Phi$ and $\psi_{j,q'}^{i,r}$, we have
\begin{equation*}
	\CalL \psi_{i,q}^{\mathrm{loc}}\mid_{\tau_j} = \sum_{q'=1}^Q \Theta_{iq,jq'}^{i,-1} \phi_{j,q'}.
\end{equation*}
Since $\{\phi_{j,q'}\}_{q=1}^Q$ is orthogonal and normalized such that $\int \phi_{j,q} \phi_{j,q'} = |\tau_j| \delta_{q,q'}$, we get
\begin{equation}\label{eqn:I21}
	\|\CalL \psi_{i,q}^{\mathrm{loc}}\|_{L^2(\tau_j)} = |\tau_j|^{1/2} \left( \sum_{q'=1}^Q (\Theta_{iq,jq'}^{i,-1})^2\right)^{1/2}.
\end{equation}
Moreover, we obtain $w_{jq'} = \int_D \eta \psi_{i,q} \phi_{j,q'}$ by definition, and thus we get
\begin{equation}\label{eqn:I22}
	|\tau_j|^{-1/2} \left(\sum_{q'=1}^Q |w_{jq'}|^2\right)^{1/2} \le \|\eta \psi_{i,q}\|_{L^2(\tau_j)} \le \|\psi_{i,q}\|_{L^2(\tau_j)}.
\end{equation}
Here, we have made use of $0 \le \eta \le 1$ in the last step. Combining \eqref{eqn:I21} and \eqref{eqn:I22}, we get
\begin{equation*}
\begin{split}
I_2 &= 2 |\sum_{\tau_j \subset S^*} \sum_{q'=1}^Q w_{jq'} \Theta_{iq, jq'}^{i,-1}| \\
& \le 2 \sum_{\tau_j \subset S^*} \left( \sum_{q'=1}^Q (\Theta_{iq,jq'}^{i,-1})^2\right)^{1/2} \left(\sum_{q'=1}^Q |w_{jq'}|^2\right)^{1/2} \\
& \le 2 \sum_{\tau_j \subset S^*} \|\CalL \psi_{i,q}^{\mathrm{loc}}\|_{L^2(\tau_j)} \|\psi_{i,q}\|_{L^2(\tau_j)}.
\end{split}
\end{equation*}
Now, we arrive at exactly the same situation as $I_2$ (see~\eqref{eqn:I2_1simple}) in the proof of Theorem~\ref{thm:expdecay2simple}. With the same derivation from Eqn.~\eqref{eqn:I2_1simple} to Eqn.~\eqref{eqn:I2_2simple}, i.e., applying Lemma~\ref{lem:boundLf2} to $\|\CalL \psi_{i,q}^{\mathrm{loc}}\|_{L^2(\tau_j)}$ and Theorem~\ref{thm:generalPoincare} to $\|\psi_{i,q}\|_{L^2(\tau_j)}$, we obtain
\begin{equation}\label{eqn:estimateI2_1}
\begin{split}
	I_2 &\le 2 \sqrt{2 \theta_{k,\max}} C(k, d, \delta) C_p |\psi_{i,q}|_{k,2,S^*} \|\psi_{i,q}^{\mathrm{loc}}\|_{H(S^*)} \\
	&\le 4 \sqrt{\frac{\theta_{k,\max}}{\theta_{k,\min}}} C(k, d, \delta) C_p \|\psi_{i,q}\|_{H(S^*)} \|\psi_{i,q}^{\mathrm{loc}}\|_{H(S^*)},
\end{split}
\end{equation}
where we have used $\theta_{k,\max} := \max(\theta_{0,\max}, \theta_{k,\max})$, the strong ellipticity~\eqref{eqn:strongelliptic} and the bound~\eqref{eqn:kleH2} in the last step. Applying the exponential decay of Theorem~\ref{thm:expdecay2} to both $\|\psi_{i,q}\|_{H(S^*)}$ and $\|\psi_{i,q}^{\mathrm{loc}}\|_{H(S^*)}$, we obtain
\begin{equation}\label{eqn:estimateI2_2}
	I_2 \le 2 \sqrt{\frac{\theta_{k,\max}}{\theta_{k,\min}}} C(k, d, \delta) C_p e^{1-\frac{r-2h}{lh}} \|\psi_{i,q}\|_{H(D)} \|\psi_{i,q}^{\mathrm{loc}}\|_{H(D)}.
\end{equation}

Combining Eqn.~\eqref{eqn:Hdif2}, \eqref{eqn:estimateI1_3} and \eqref{eqn:estimateI2_2}, and using Eqn.~\eqref{eqn:phinorm} to bound $\|\psi_{i,q}^{\mathrm{loc}}\|_{H(D)}$ and $\|\psi_{i,q}\|_{H(D)}$ (recall $\|\psi_{i,q}\|_{H(D)} \le \|\psi_{i,q}^{\mathrm{loc}}\|_{H(D)}$), we complete the proof of Eqn.~\eqref{eqn:localization}. 

When the operator $\CalL$ contains only the highest order terms, i.e., $\CalL u = (-1)^k \sum\limits_{|\sigma|=|\gamma|=k} D^{\sigma}(a_{\sigma \gamma} D^{\gamma} u)$, Eqn.~\eqref{eqn:estimateI1_3} and \eqref{eqn:estimateI2_2} hold true for all $h > 0$. In this case, we can get rid of the factor ``2'' in both Eqn.~\eqref{eqn:estimateI1_3} and \eqref{eqn:estimateI2_2}. Therefore, we obtain the estimate on $C_3$ stated in the theorem.
\end{proof}

\begin{theorem}\label{thm:localizedMsFEM}
Let $u\in H_0^k(D)$ be the weak solution of $\CalL u = f$ and $\psi_{i,q}^{\mathrm{loc}}$ be the localized basis functions defined in Eqn.~\eqref{eqn:localVar}. Then, for $ r \ge (d+4k) l h \log(1/h) + 2 (1 + l \log C_4) h$, we have
\begin{equation}\label{eqn:localizedMsFEM}
	\inf_{v \in \Psi^{\mathrm{loc}}} \|u - v\|_{H(D)} \le \frac{2 C_p}{\sqrt{a_{\min}}} h^k \|f\|_{L^2(D)},
\end{equation}
where $C_4 = \frac{C_3 C_e}{C_p} (Q a_{\min})^{1/2}$, and $C_3$ is defined in Theorem~\ref{thm:localization}, $a_{\min}$ comes from the norm-equivalence~\eqref{eqn:normequiv}, and $C_e$ is the constant such that $\|u\|_{L^2(D)} \le C_e \|f\|_{L^2(D)}$ holds true.
\end{theorem}
\begin{proof}
Let $v_1 := \sum_{i=1}^m \sum_{q=1}^Q c_{iq} \psi_{i,q}$ and $v_2 := \sum_{i=1}^m \sum_{q=1}^Q c_{iq} \psi_{i,q}^{\mathrm{loc}}$ with $c_{iq} = \int_D u \phi_{i,q}$. Estimation~\eqref{eqn:Herror2} gives that 
\begin{equation}\label{eqn:localized_part1}
	\|u - v_1\|_H \le \frac{C_p h^k}{\sqrt{a_{\min}}} \|f\|_{L^2(D)}.
\end{equation}
Using the Cauchy inequality, we have
\begin{equation*}
	\|v_1 - v_2\|_H \le \max_{i,q} \|\psi_{i,q} - \psi_{i,q}^{\mathrm{loc}}\|_H \sum_{i=1}^m \sum_{q=1}^Q |c_{iq}| \le \max_{i,q} \|\psi_{i,q} - \psi_{i,q}^{\mathrm{loc}}\|_H  \sum_{i=1}^m Q^{1/2} (\sum_{q=1}^Q |c_{iq}|^2)^{1/2}.
\end{equation*}
Thanks to the orthogonality of $\{\phi_{i,q}\}_{q=1}^Q$~\eqref{eqn:phinormalize}, we have $|\tau_i|^{-1/2} (\sum_{q=1}^Q |c_{iq}|^2)^{1/2} \le \|u\|_{L^2(\tau_i)}$. Then, we obtain
\begin{equation*}
	\|v_1 - v_2\|_H \le \max_{i,q} \|\psi_{i,q} - \psi_{i,q}^{\mathrm{loc}}\|_H Q^{1/2} \sum_{i=1}^m|\tau_i|^{1/2} \|u\|_{L^2(\tau_i)} \le \max_{i,q} \|\psi_{i,q} - \psi_{i,q}^{\mathrm{loc}}\|_H (Q |D|)^{1/2} \|u\|_{L^2(D)}.
\end{equation*}
Using the energy estimation $\|u\|_{L^2(D)} \le C_e \|f\|_{L^2(D)}$ and Theorem~\ref{thm:localization}, we obtain
\begin{equation}\label{eqn:localized_part2}
	\|v_1 - v_2\|_H \le C_3 C_e Q^{1/2} h^{-\frac{d}{2}-k} \exp(- \frac{r-2h}{2 l h}) \|f\|_{L^2(D)}.
\end{equation}
Combining Eqn.~\eqref{eqn:localized_part1} and \eqref{eqn:localized_part2} together, we conclude the proof.
\end{proof}

By applying the Aubin--Nistche duality argument, we can get the following corollary.
\begin{corollary}\label{col:localizedOC}
Let $\psi_{i,q}^{\mathrm{loc}}$ be the localized basis functions defined in Eqn.~\eqref{eqn:localVar}. Then, for $ r \ge (d+4k) l h \log(1/h) + 2 (1 + l \log C_4) h$, we have
\begin{equation}\label{eqn:localizedOC}
	\|\CalK - \CalP_{\Psi^{\mathrm{loc}}}^{(H)} \CalK\| \le \frac{4 C_p^2}{a_{\min}} h^{2k},
\end{equation}
where all the constants are the same as those defined in Theorem~\ref{thm:localizedMsFEM}.
\end{corollary}
Corollary~\ref{col:localizedOC} shows that we can compress the symmetric positive semidefinite operator $\CalK$ with the optimal rate $h^{2 k}$ and with the nearly optimal localized basis (with support size of order $h \log(1/h)$).

\begin{remark}
All the results and proofs presented above can be carried over to other homogeneous boundary conditions. Given a specific homogeneous boundary condition, one only needs to modify the proof of Lemma~\ref{lem:phinorm}. Specifically, when the patch $\tau_i$ intersects with the boundary of $D$, the constructed function $\zeta_{i,q}$ should honor the same boundary condition on $\partial D$. The scaling argument in the proof of Lemma~\ref{lem:phinorm} still works for other homogeneous boundary conditions.
\end{remark}

\section{Numerical Examples}
\label{sec:numericalExps}
In this section, we present several numerical results to support the theoretical findings and to show how the sparse operator compression is utilized in higher-order elliptic operators.  In Section~\ref{subsec:1dexponential}, we apply our method to compress the Mat\'{e}rn covariance function~\eqref{intro:matern} with $\nu = 1/2$. We show that our method is able to achieve the optimal compression error with nearly optimally localized basis functions, which means that we are able to get optimality on both ends of the accuracy--sparsity trade-off in the sparse PCA. In Section~\ref{subsec:1dbiharmonic}, we apply our method to a 1D fourth-order elliptic equation with the homogeneous Dirichlet boundary condition and show that our basis functions, when used as multiscale finite element basis, can achieve the optimal $h^2$ convergence rate in the energy norm. In Section~\ref{subsec:2dbiharmonic}, we apply our method to a 2D fourth-order elliptic equation and show that the energy-minimizing basis functions decays exponentially fast away from its associated patch.

\subsection{The compression of a Mat\'{e}rn covariance kernel}
\label{subsec:1dexponential}
In spatial statistics, geostatistics, machine learning and image analysis, the Mat\'{e}rn covariance~\cite{matern2013spatial} is used to model random fields with smooth samples; see, e.g.,~\cite{stein2012interpolation, guttorp2006studies, gneiting2012matern}. The Mat\'{e}rn covariance between two points $x, y \in D \subset \R^d$ is given by 
\vspace{-0.05in}
\begin{equation}\label{intro:matern}
	K_{\nu}(x, y) = \sigma^2 \frac{2^{1-\nu}}{\Gamma(\nu)} \left( \sqrt{2\nu} \frac{|x - y|}{\rho} \right)^{\nu} K_{\nu}\left(\sqrt{2\nu} \frac{|x - y|}{\rho}\right),
\end{equation}
\vspace{-0.05in}
where $\Gamma$ is the gamma function, $K_{\nu}$ is the modified Bessel function of the second kind, and $\rho$ and $\nu$ are nonnegative parameters of the covariance. Its Fourier transform is given by 
\begin{equation}\label{eqn:maternFourier}
	\hat{k} (\omega) = c_{\nu, \lambda} \sigma^2 \left( \frac{2 \nu}{\lambda^2} + |\omega|^2 \right)^{-(\nu + d/2)}, \quad c_{\nu, \lambda} := \frac{2^d \pi^{d/2} \Gamma(\nu + d/2) (2 \nu)^{\nu}}{\Gamma(\nu) \lambda^{2\nu}},
\end{equation}
where $\hat{f}(\omega)$ is the Fourier transform of $f$. For both sampling from the random fields and performing basic computations like marginalization and conditioning, we need to compress the Mat\'{e}rn covariance operator $\CalK: L^2(D) \to L^2(D)$, which is defined through the Hilbert--Schmidt operator with kernel $K_{\nu}(x, y)$, by a rank-$n$ covariance operator:
\begin{equation}\label{intro:PCA}
	E_{\mathrm{oc}}(\mathit{\Psi}; \CalK) := \min_{K_n \in \R^{n\times n},~ K_n \succeq 0} \|\CalK - \mathit{\Psi} K_n \mathit{\Psi}^T\|_2,
\end{equation}
where $\mathit{\Psi} =[\psi_1, \ldots, \psi_n]$ spans the range space of the approximate operator $\mathit{\Psi} K_n \mathit{\Psi}^T$. Recent study ~\cite{lindgren2011explicit, bolin2011spatial} shows that the Mat\'{e}rn covariance and the elliptic operators are closely connected.
With proper homogeneous boundary conditions, the Mat\'{e}rn covariance operator with $\nu + d/2$ being an integer is the solution operator of an elliptic operator of order $2\nu + d$. For example, 
the Mat\'{e}rn covariance operator with $\nu = 1/2$ is the solution operator of a second-order elliptic operator $(2 l \sigma^2)^{-1} \left( 1 - \rho^2 \frac{\rd^2}{\rd x^2} \right)$ when the physical dimension $d = 1$, and is the solution operator of a fourth-order elliptic operator $(8 \pi \rho^3 \sigma^2)^{-1} \left( 1 - 2 \rho^2 \Delta + \rho^4 \Delta^2\right)$ when $d = 3$. 
%the Mat\'{e}rn covariance operator with $\nu = 1$ is the solution operator of the fourth-order elliptic operator $(4 \pi \rho^2 \sigma^2)^{-1} \left( 1 - 2 \rho^2 \Delta + \rho^4 \Delta^2\right)$ when $d = 2$. 

Based on Eqn.~\eqref{def:Theta} and \eqref{eqn:psi1}, we can also compute the exponentially decaying basis functions from the covariance operator $\CalK$. In this example, we apply our method to compress the following exponential kernel
\begin{equation}\label{eqn:expkernel}
	K(x,y) = \exp(-|x-y|) \qquad x, y \in [0,1],
\end{equation}
which is exactly the Mat\'{e}rn covariance~\eqref{intro:matern} with $\nu = 1/2$, $\sigma = 1$ and $\rho = 1$. This problem has been studied by different groups; see, e.g.,~\cite{gittelson2012representation, d2013coarse, hou_LocalModes_2014, bachmayr2016representations}. \referee{We remark that since the Mat\'{e}rn covariance function corresponds to the solution operator of an elliptic PDE with constant coefficient, one can compress the Mat\'{e}rn covariance kernel by using a piecewise linear polynomial or wavelets with optimal locality and accuracy. It is not necessary to use the exponential decaying basis to perform the operator compression. We use this example to illustrate that our method can be also applied to compress a general kernel function.}
%As far as we know, our result is the first result that achieves optimal compression error and nearly optimal localization with theoretical guarantee. 

We partition the interval $[0,1]$ uniformly into $m=2^6$ patches and follow our strategy to construct basis functions. By the Fourier transform, we know that it is associated with the second-order elliptic operator $\frac{1}{2}\left( 1 - \frac{\rd^2}{\rd x^2} \right)$. Therefore, we take $\Phi$ as piecewise constant functions and then compute $\mathit{\Psi}$ by Eqn.~\eqref{def:Theta} and \eqref{eqn:psi1}. In Figure~\ref{fig:exponential10}, we plot $\phi_{32}$ and $\psi_{32}$, which is associated with the patch $[1/2-h, 1/2]$. We can see that the basis function $\psi_{32}$ clearly has an exponential decay.
\begin{figure}[ht]
\centering
\includegraphics[width = 0.7\textwidth]{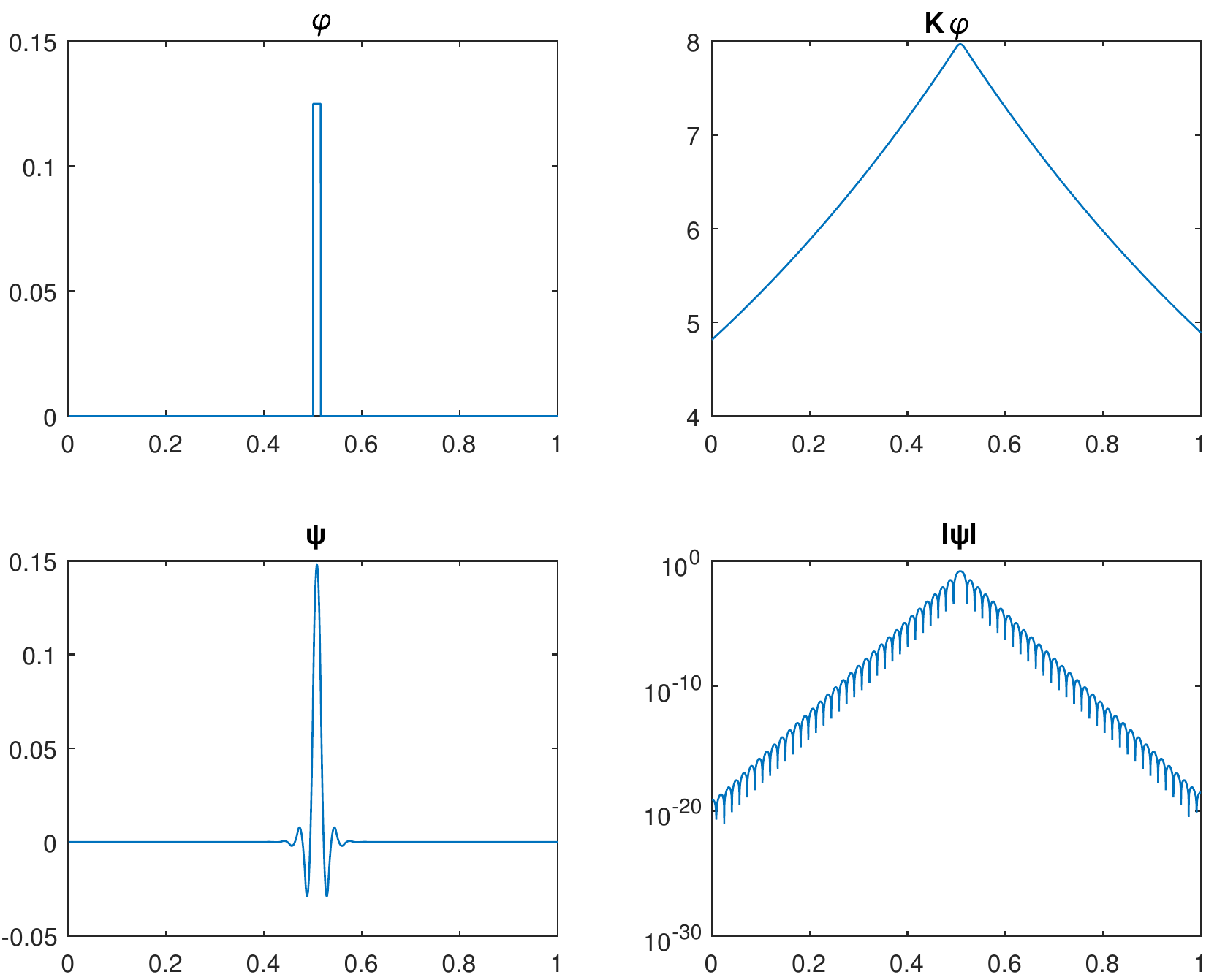}
\caption{Basis function associated with patch $[1/2-h, 1/2]$.}\label{fig:exponential10}
\end{figure}
We take $m = 2^{i}$ for $0\le i \le 7$ and compute the compression error $E(\mathit{\Psi}; \CalK)$. The result is shown in Figure~\ref{fig:experror_decay}. We can see that the exponentially decaying basis functions $\mathit{\Psi}$ have nearly the same compression rate as that of the eigendecomposition.
\begin{figure}[ht]
\centering
\includegraphics[width = 0.36\textwidth]{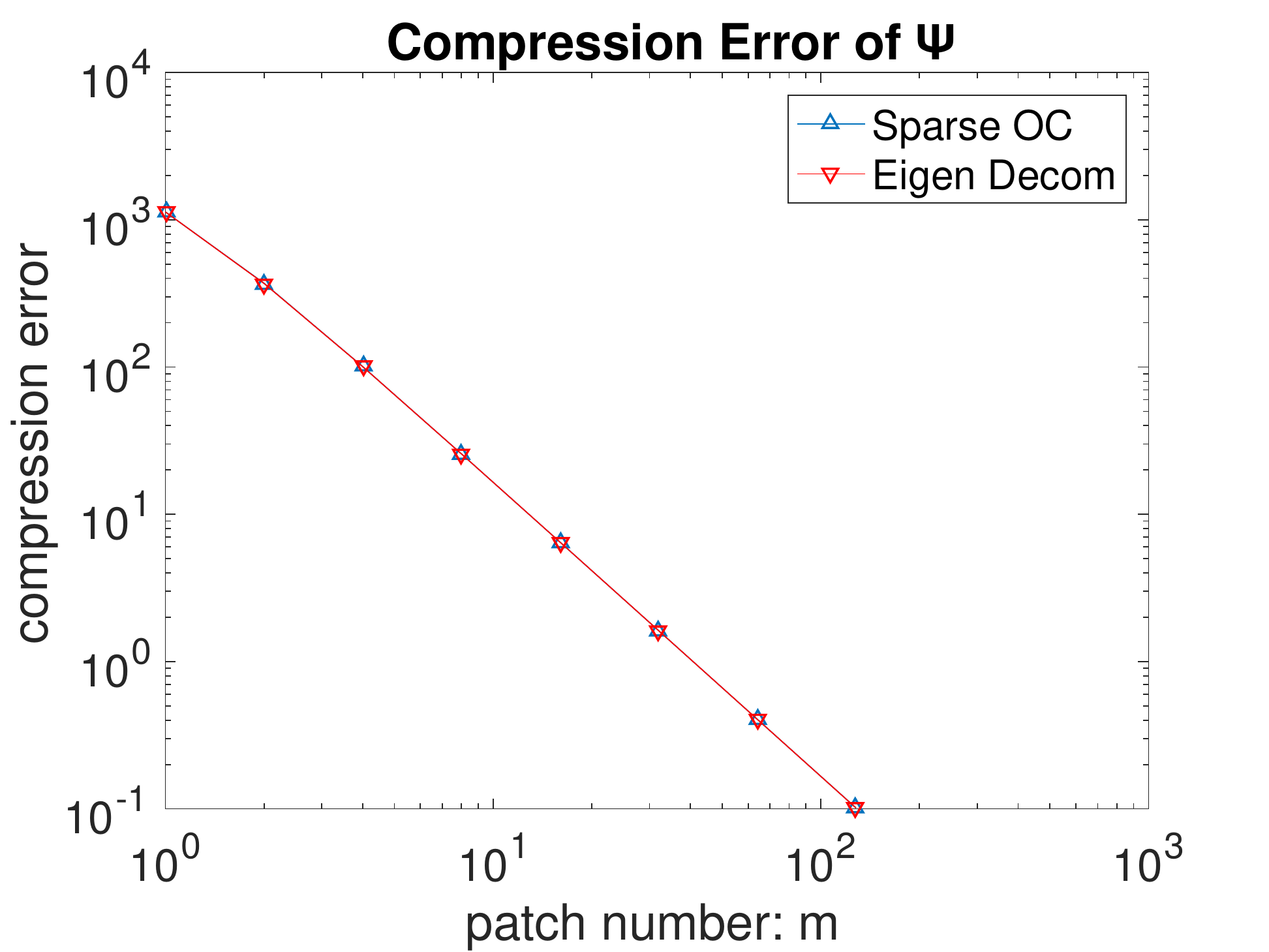}
\caption{Operator compression error $E(\mathit{\Psi}; \CalK)$~\eqref{intro:PCA} for the exponential kernel~\eqref{eqn:expkernel} with exponentially decaying basis functions $\mathit{\Psi}$. They have nearly same compression error as that given by the global eigenfunctions of $\CalK$.}\label{fig:experror_decay}
\end{figure}

One can easily verify that the exponential kernel~\eqref{eqn:expkernel} is the Green's function of the following second-order elliptic equation
\begin{equation}\label{eqn:expelliptic}
	 - \frac{1}{2} u''(x) + \frac{1}{2} u  = f(x), \qquad 0 < x < 1,
	u(0) - u'(0)=0, \quad  u(1) + u'(1) = 0,
\end{equation}
with boundary condition $u(0) - u'(0)=0, \quad  u(1) + u'(1) = 0$.
The associated energy norm is 
\begin{equation}\label{eqn:expenergynorm}
	\|u\|_{H(D)}^2 = \frac{1}{2}\left(u(0)^2 + u(1)^2 + \int_0^1 (u')^2 + \int_0^1 u^2 \right).
\end{equation}
Solving the localized variational problem~\eqref{eqn:localVarOtherBC}, we can get localized basis functions $\mathit{\Psi}^{\mathrm{loc}}$. With different sizes of the support $S_r$, we compute the compression error $E(\mathit{\Psi}^{\mathrm{loc}}; \CalK)$ for $m = 2^{i}$ ($0\le i \le 7$). The results are summarized in Figure~\ref{fig:experror_local}. In the left subfigure of Figure~\ref{fig:experror_local}, we take the support with size $C h$, for $C = 3, 5, 7, 9$ and $11$. In the right subfigure of Figure~\ref{fig:experror_local}, we take the support with size $C h \log_2(1/h)$, for $C = 2, 2.1$ and $2.4$. For a support of size $C h \log_2(1/h)$, it contains $\ceil{C \log_2(1/h)}$ patches, where $\ceil{C \log_2(1/h)}$ is the smallest integer of $C \log_2(1/h)$. We can see that the oversampling strategy with $r= c h$ does not give the optimal convergence rate , while the oversampling strategy with $r= c h \log_2(1/h)$ gives the optimal second-order convergence rate as guaranteed by Corollary~\ref{col:localizedOC}. For $m=2^7$ and $r = 2.4 h \log_2(1/h)$, the constructed localized basis functions achieves the same operator compression error as that using 128 eignefunctions. 
%We show several basis functions $\psi_i^{\mathrm{loc}}$ in Figure~\ref{fig:expbasis_local}. One can see that the basis functions on the boundary honor the Robin boundary conditions. 
\begin{figure}[ht]
\centering
\includegraphics[width = 0.45\textwidth]{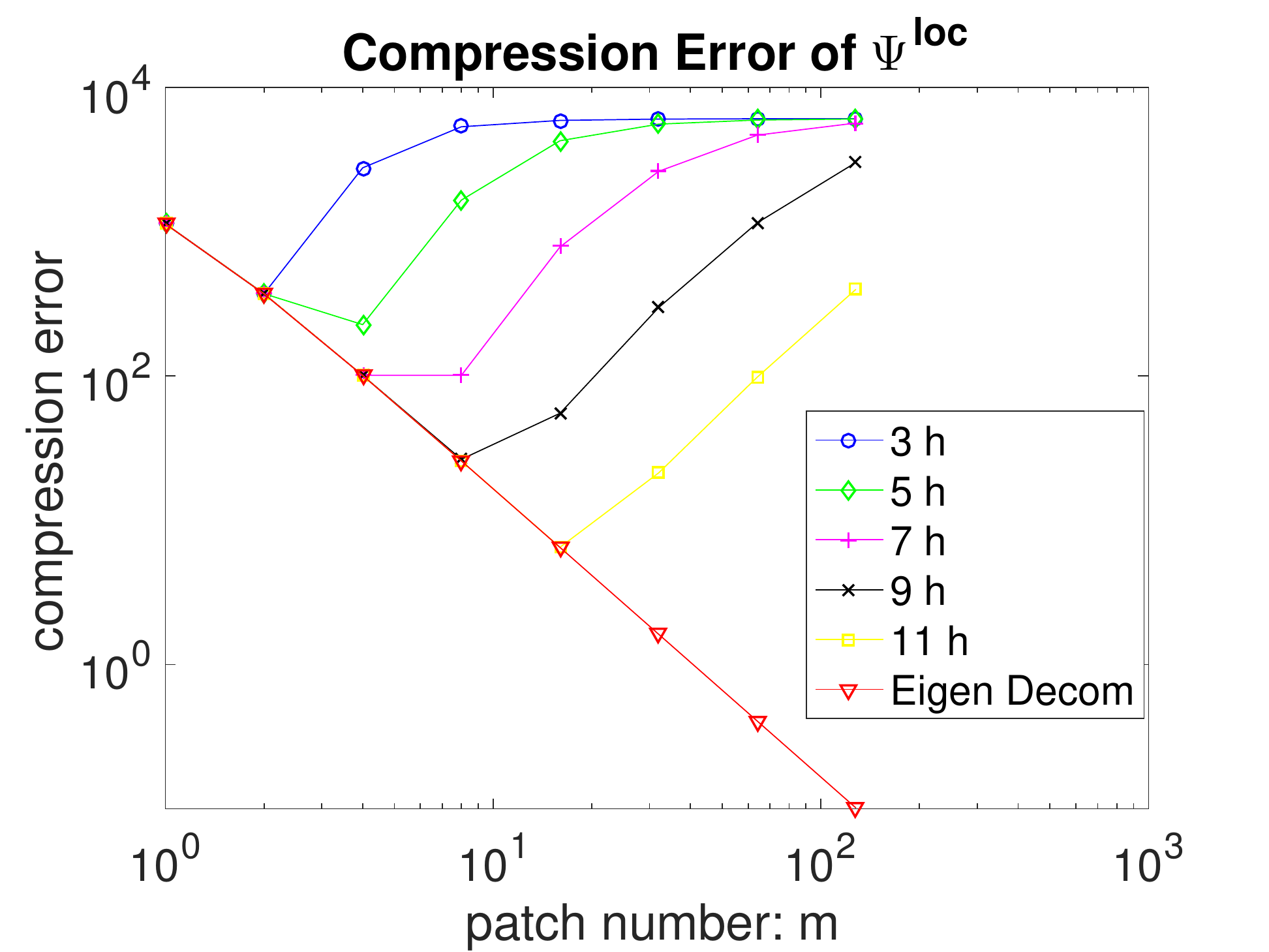}
\includegraphics[width = 0.45\textwidth]{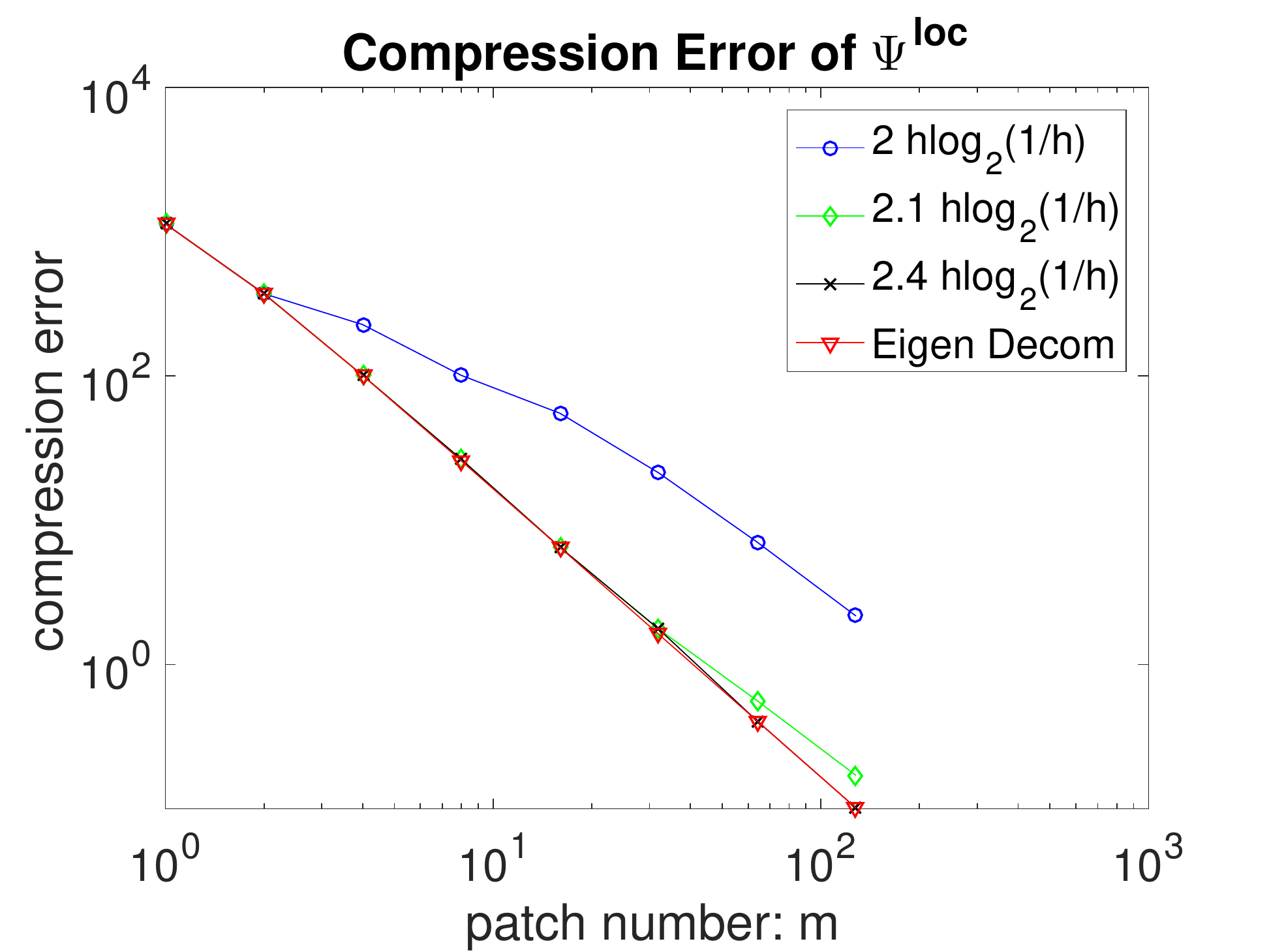}
\caption{Operator compression error $E(\mathit{\Psi}^{\mathrm{loc}}; \CalK)$~\eqref{intro:PCA} with basis functions $\mathit{\Psi}^{\mathrm{loc}}$. The oversampling strategy with $r = c h$ (left) does not work well, while the oversampling strategy with $r = c h \log_2(1/h)$ (right) gives the optimal second-order convergence rate.}\label{fig:experror_local}
\end{figure}
%\begin{figure}[ht]
%\centering
%\includegraphics[width = 0.4\textwidth]{exponential1d/expbasis-eps-converted-to.pdf}
%\includegraphics[width = 0.4\textwidth]{exponential1d/expbasis_log-eps-converted-to.pdf}
%\caption{A few basis functions for the case $m=2^7$ and $r = 2.4 h \log_2(1/h)$.}\label{fig:expbasis_local}
%\end{figure}

\subsection{The 1D fourth-order elliptic operator}
\label{subsec:1dbiharmonic}
Consider the solution operator of the Euler-Bernoulli equation
\begin{equation}\label{eqn:1dbiharmonic}
\begin{split}
	\frac{\rd^2}{\rd x^2} \left( a(x) \frac{\rd^2 u}{\rd x^2} \right) = f(x),& \qquad 0 < x < 1, \\
	u(0) = u'(0)=0, \quad & u(1) = u'(1) = 0,
\end{split}
\end{equation}
which describes the deflection $u$ of a clamped beam subject to a transverse force $f\in L^2([0,1])$. The flexural rigidity $a(x)$ of the beam is modeled by
\begin{equation}\label{eqn:flexural}
	a(x) := 1 + \frac{1}{2} \sin\left(\sum_{k=1}^K k^{-\alpha} (\zeta_{1k} \sin(kx) + \zeta_{2k} \cos(kx))\right),
\end{equation}
where $\{\zeta_{1k}\}_{k=1}^K$ and $\{\zeta_{2k}\}_{k=1}^K$ are two independent random vectors with independent entries uniformly distributed in $[-1/2, 1/2]$. This oscillatory coefficient is also used in~\cite{hou_multiscale_1997, ming2006numerical, owhadi_polyharmonic_2014}, and has no scale separation. We choose $\alpha = 0$ and $K = 40$ in the numerical experiment. A sample coefficient is shown in Figure~\ref{fig:exp1_acoef}.
\begin{figure}[ht]
\centering
\includegraphics[width = 0.38\textwidth]{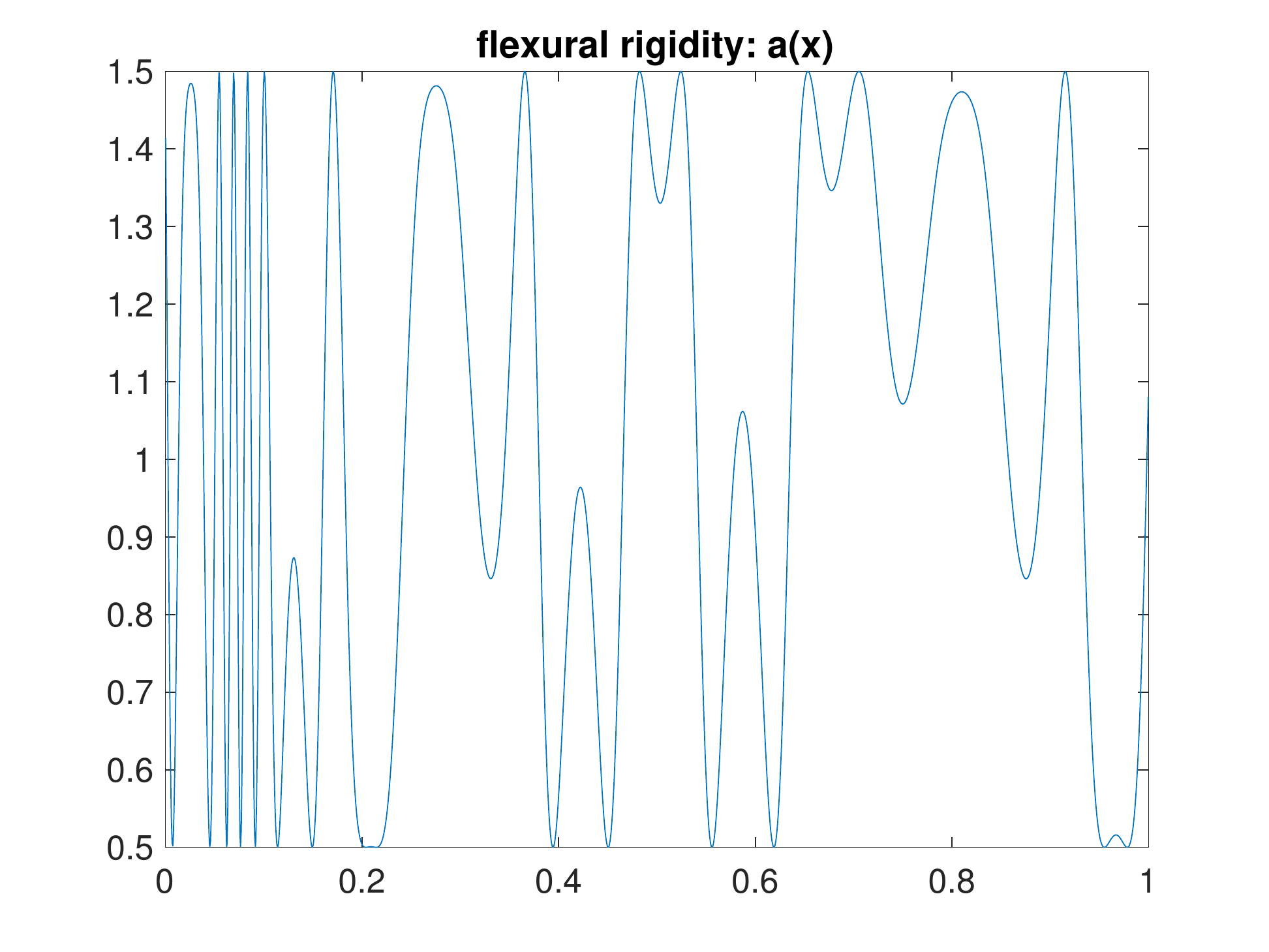}
\caption{Highly oscillatory flexural rigidity without scale separation.}\label{fig:exp1_acoef}
\end{figure}

We partition the physical space $[0,1]$ uniformly into $m = 2^6$ patches, where the $i$th patch $I_i = [(i-1) h, i h]$ with $h = 1/m$. In this fourth-order case, our theory requires the piecewise polynomial space $\Phi$ be the space of (discontinuous) piecewise linear functions, which has dimension $n = 2 m$. We have two $\phi$'s, denoted as $\phi_{i,1}$ and $\phi_{i,2}$, associated with the patch $I_i$. Solving the quadratic optimization problem~\eqref{eqn:psivariationalk2}, we obtain the exponentially decaying basis functions. We also have two $\psi$'s, denoted as $\psi_{i,1}$ and $\psi_{i,2}$, associated with the patch $I_i$. We plot $\phi_{i,1}$ and $\phi_{i,2}$ associated with the patch $I_{32} = [1/2-h, 1/2]$ in Figure~\ref{fig:1dbiharmonicpsi} A. In Figure~\ref{fig:1dbiharmonicpsi}(B-C), we plot the basis functions $\psi_{32,1}$ and $\psi_{32,2}$, which clearly show exponential decay. 

To demonstrate the necessity for $\Psi$ to contain all piecewise linear functions, in the third column of Figure~\ref{fig:1dbiharmonicpsi}, we also plot the basis functions associated the patch $I_{32}$ when $\Phi$ is the space of piecewise constant functions. In this case, we have only one $\phi$, denoted as $\phi_{i}$, associated with the patch $I_i$. In the third column of Figure~\ref{fig:1dbiharmonicpsi}(A) and (B), we plot $\phi_{32}$ and $\psi_{32}$. Solving the quadratic optimization problem~\eqref{eqn:psivariationalk2}, we obtain only one basis function $\psi$, denoted as $\psi_{i}$, associated with the patch $I_i$. In Figure~\ref{fig:1dbiharmonicpsi}(C), we plot the basis function $\psi_{32}$ in the third column. Note that $\psi_{32}$ also shows an exponential decay, but its decay rate is much smaller than that of $\psi_{32,1}$ and $\psi_{32,2}$. 
\begin{figure}[t!]
\centering
    \subfloat[$\phi_{32,1}, \phi_{32,2}$ for piecewise linear $\Phi$ and $\phi_{32}$ for piecewise constant $\Phi$]{
       \centering
        \includegraphics[width = 0.3\textwidth, height = 0.2\textwidth]{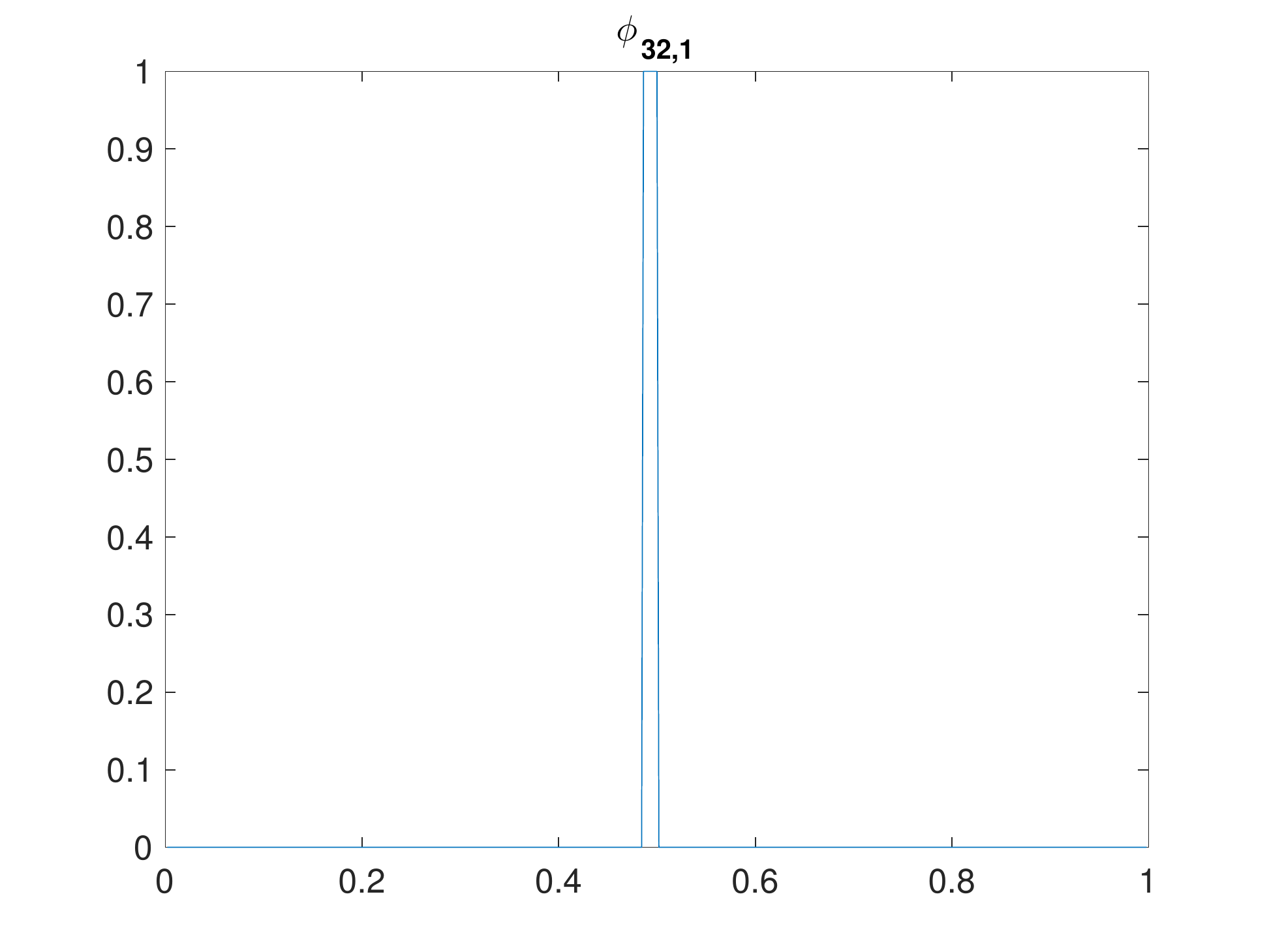}
        \includegraphics[width = 0.3\textwidth, height = 0.2\textwidth]{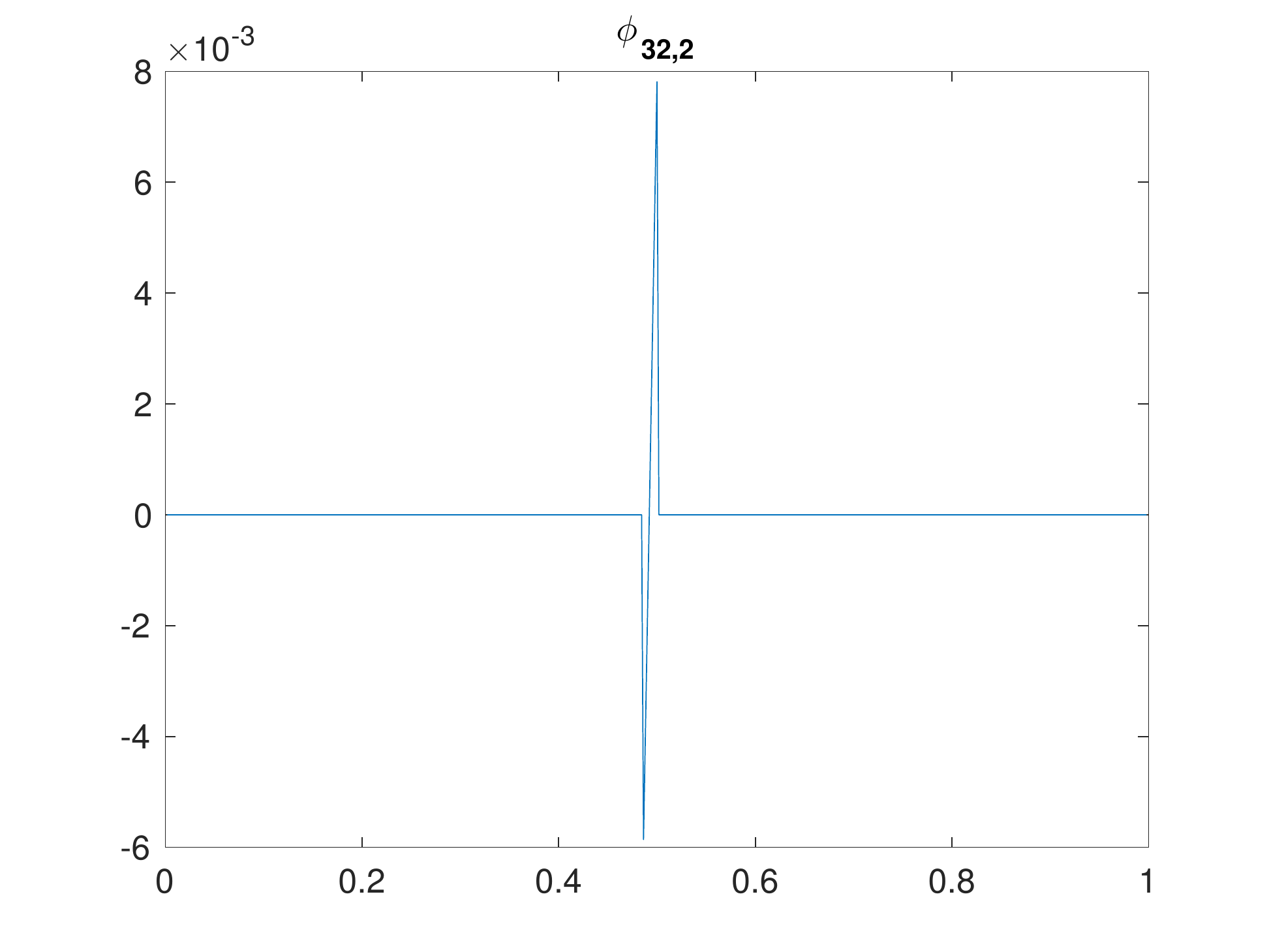}
        \includegraphics[width = 0.3\textwidth, height = 0.2\textwidth]{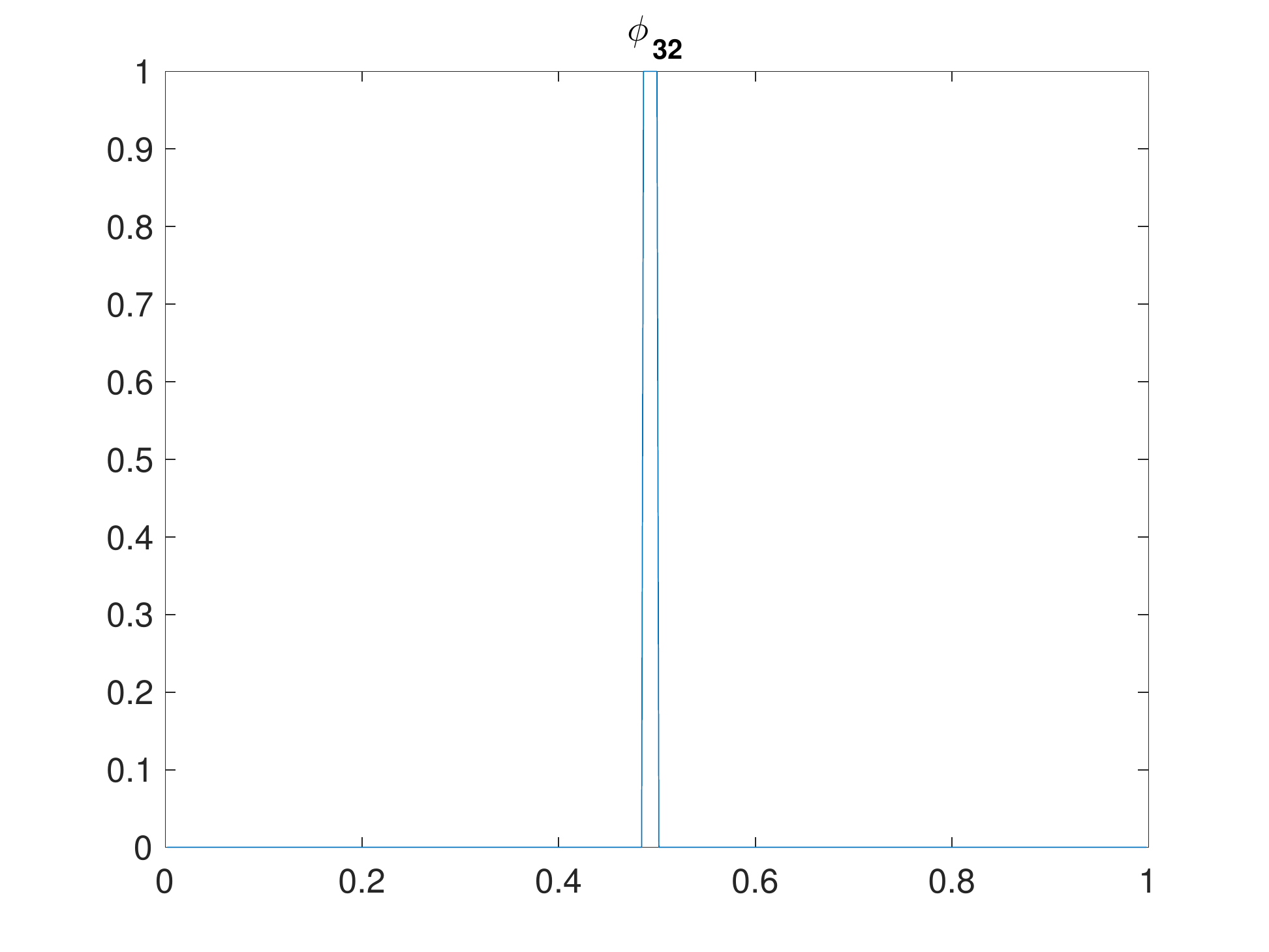}}
    
%    \subfloat[$\CalK \phi_{32,1}, \CalK \phi_{32,2}$ for piecewise %linear $\Phi$ and $\CalK \phi_{32}$ for piecewise constant $\Phi$]{
%        \centering
%%%        \includegraphics[width = 0.3\textwidth, height = 0.2\textwidth]{biharmonic1d/kphi1-eps-converted-to.pdf}
%        \includegraphics[width = 0.3\textwidth, height = 0.2\textwidth]{biharmonic1d/kphi2-eps-converted-to.pdf}
%        \includegraphics[width = 0.3\textwidth, height = 0.2\textwidth]{biharmonic1d/kphi32-eps-converted-to.pdf}}
    
    \subfloat[$\psi_{32,1}, \psi_{32,2}$ for piecewise linear $\Phi$ and $\psi_{32}$ for piecewise constant $\Phi$]{
        \centering
        \includegraphics[width = 0.3\textwidth, height = 0.2\textwidth]{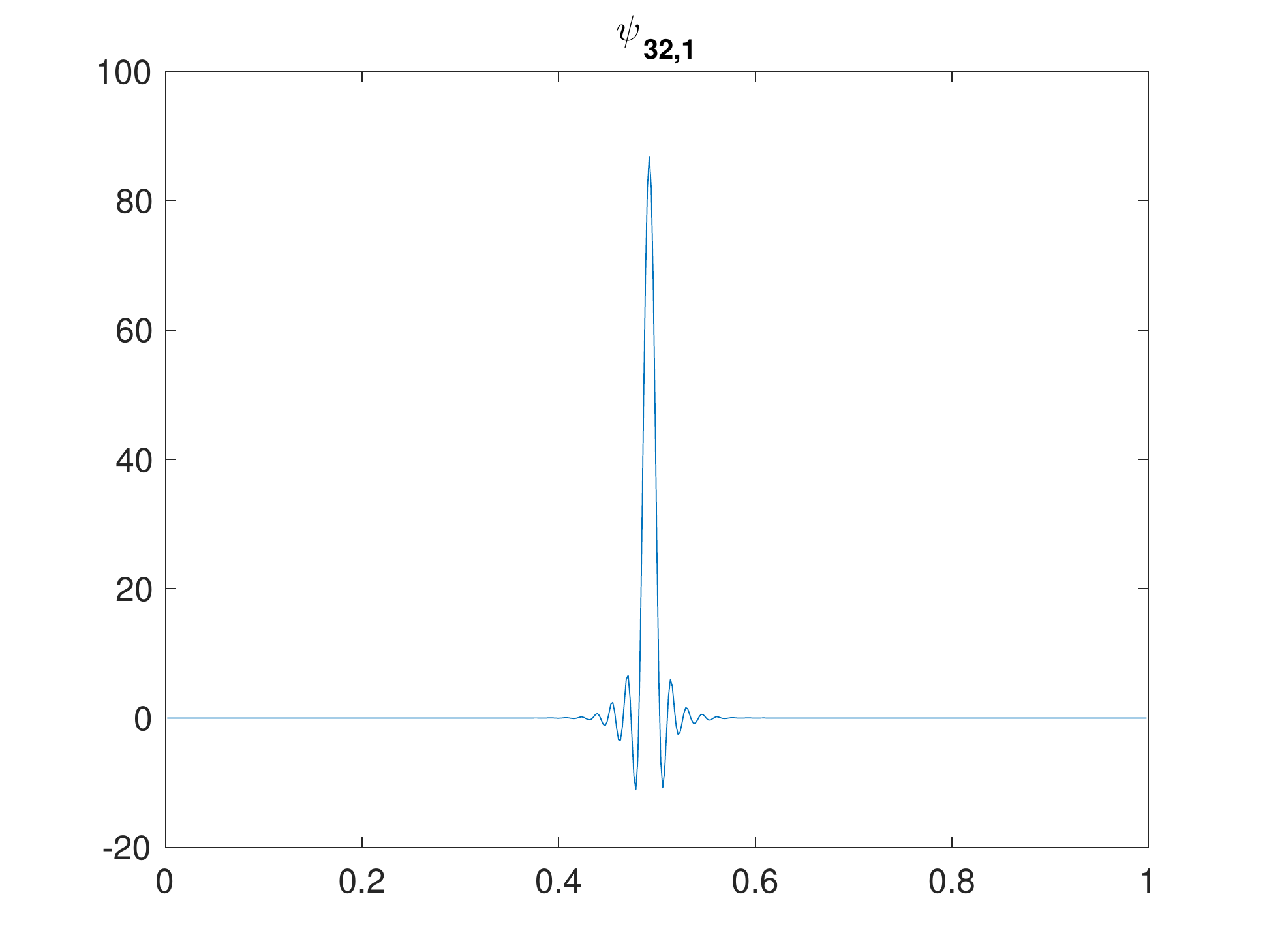}
        \includegraphics[width = 0.3\textwidth, height = 0.2\textwidth]{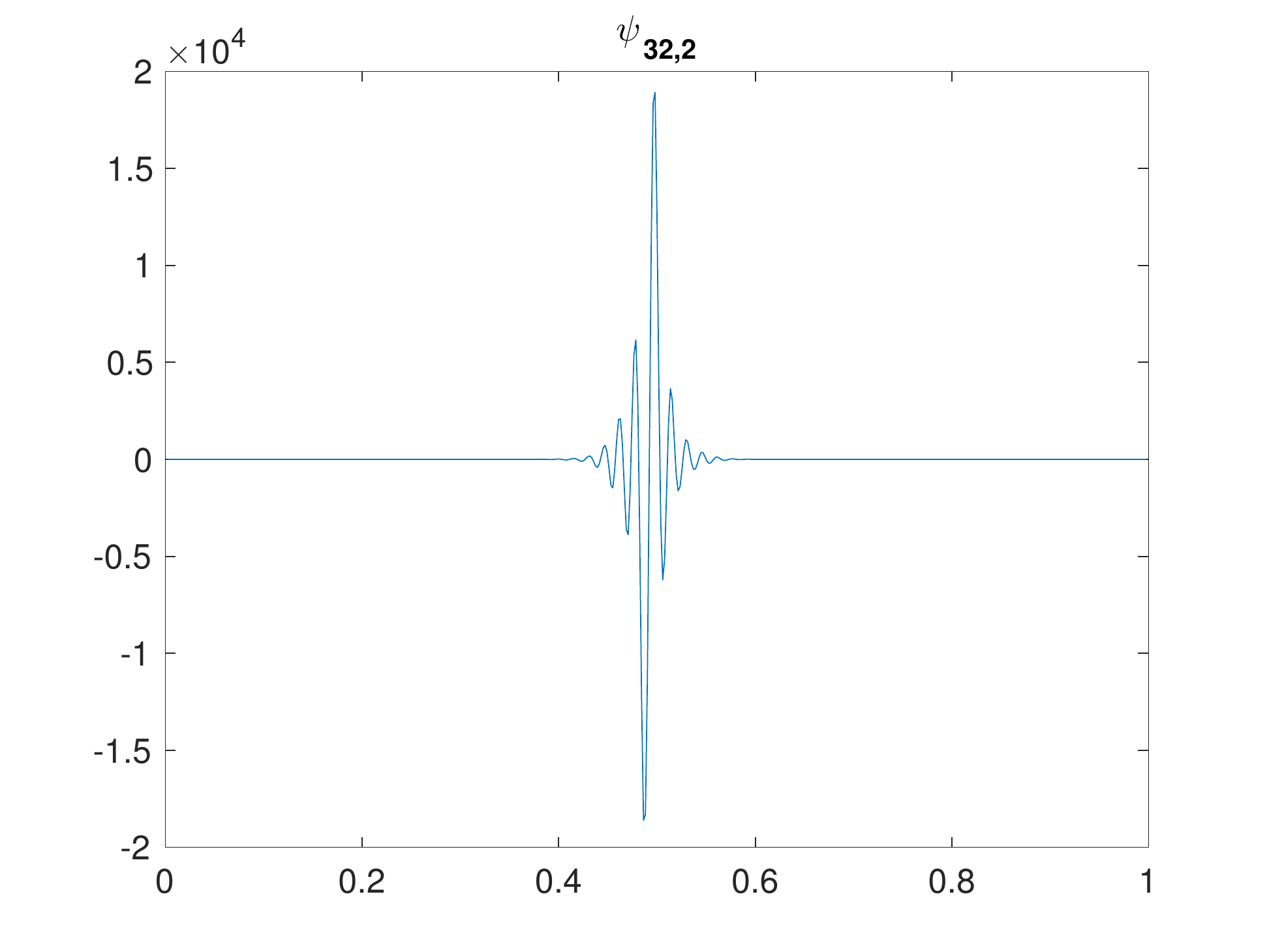}
        \includegraphics[width = 0.3\textwidth, height = 0.2\textwidth]{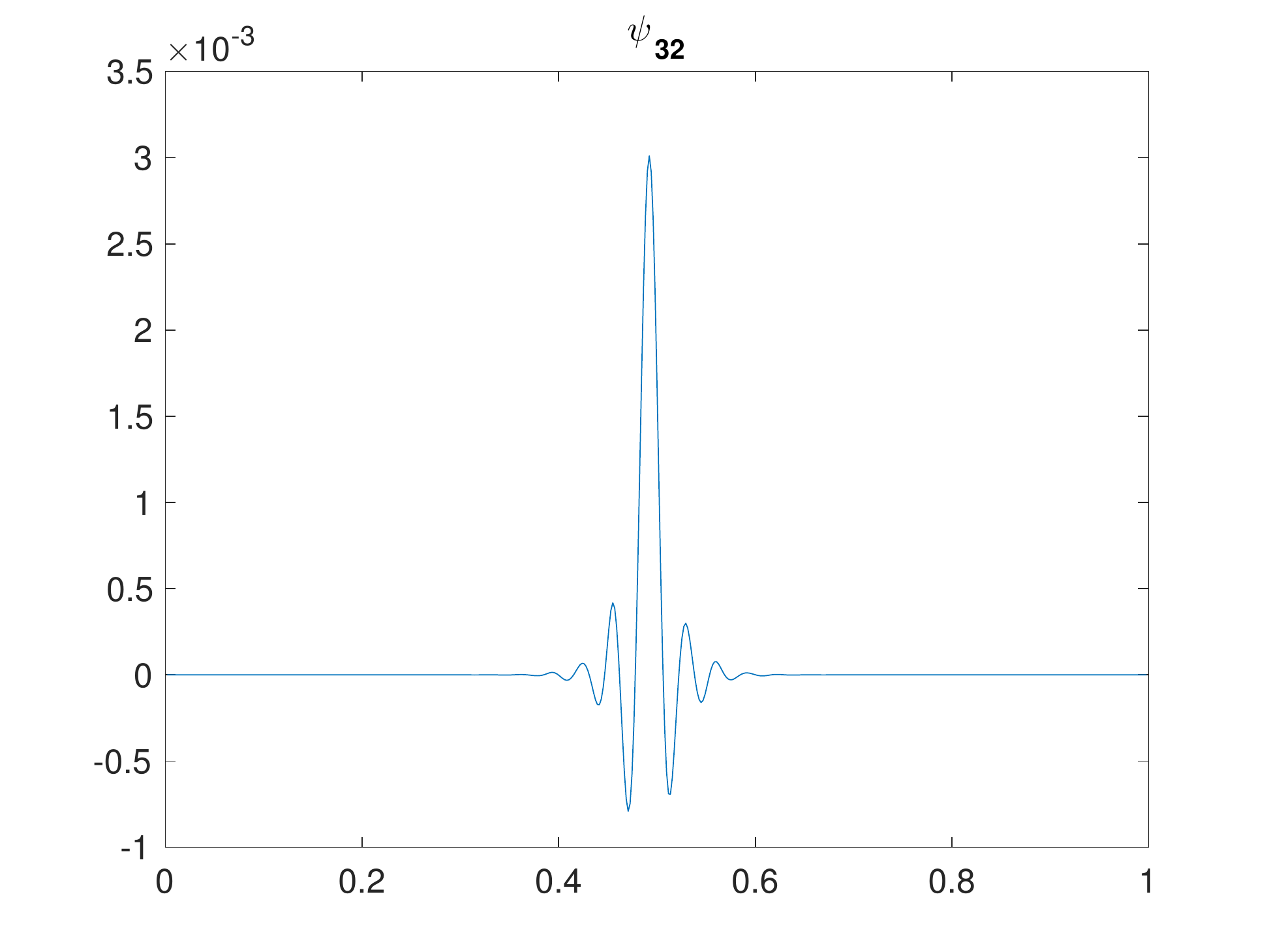}}
    
    \subfloat[In log-scale: $\psi_{32,1}, \psi_{32,2}$ for piecewise linear $\Phi$ and $\psi_{32}$ for piecewise constant $\Phi$]{
        \includegraphics[width = 0.3\textwidth, height = 0.2\textwidth]{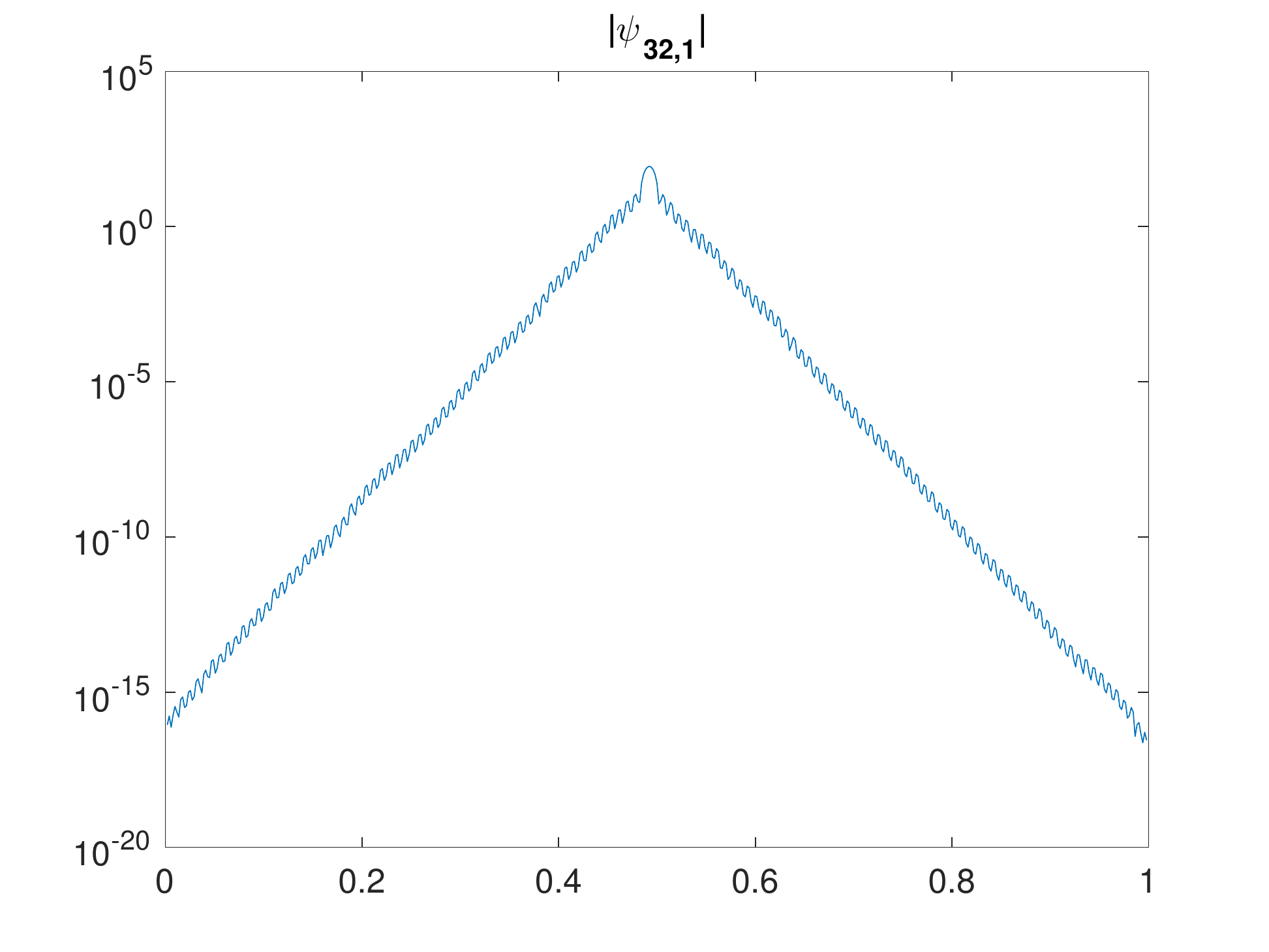}
        \includegraphics[width = 0.3\textwidth, height = 0.2\textwidth]{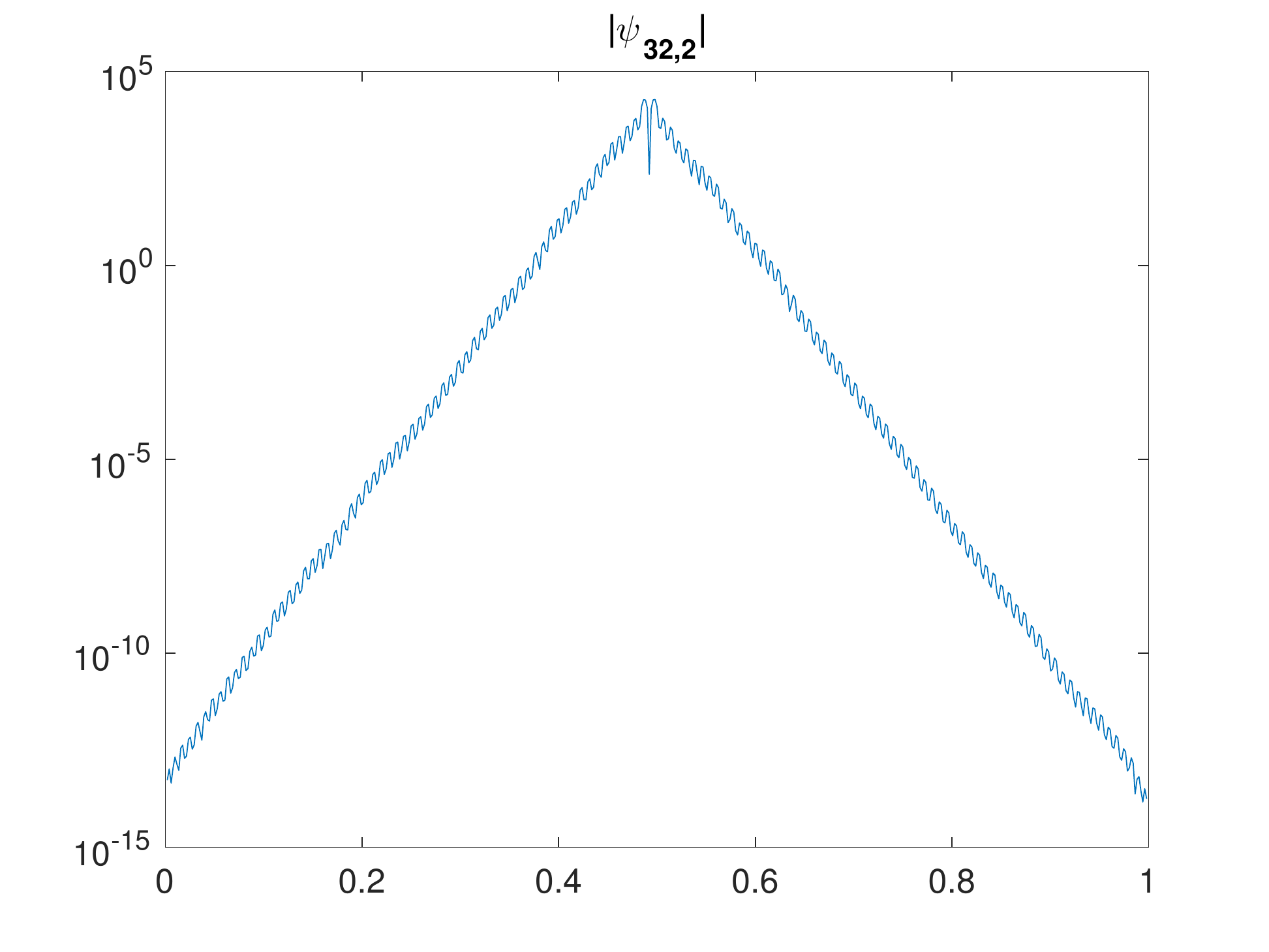}
        \includegraphics[width = 0.3\textwidth, height = 0.2\textwidth]{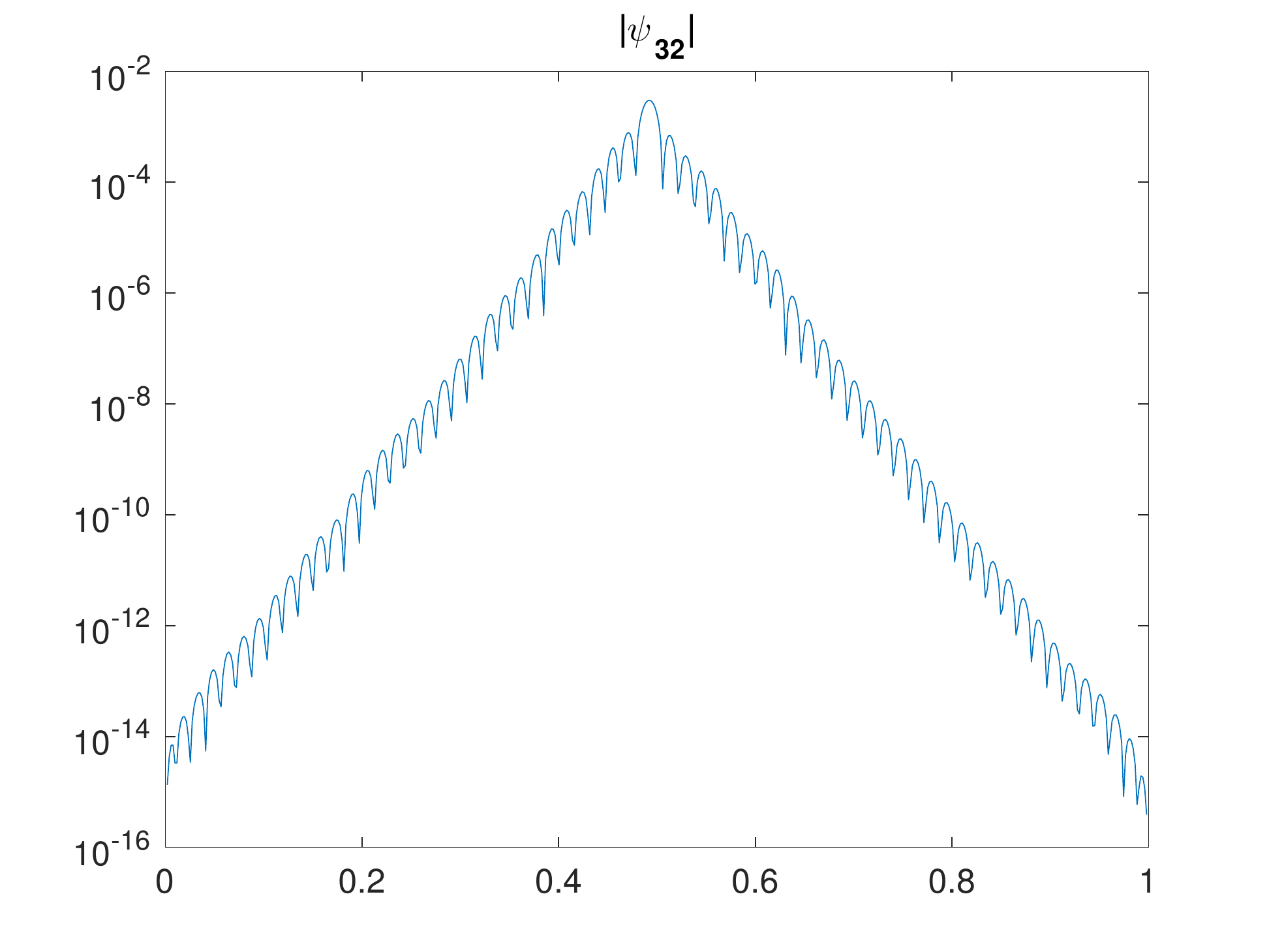}}
    
%     \subfloat[Patch-wise energy norm: $\psi_{32,1}, \psi_{32,2}$ for piecewise linear $\Phi$ and $\psi_{32}$ for piecewise constant $\Phi$]{
%        \includegraphics[width = 0.3\textwidth, height = 0.2\textwidth]{biharmonic1d/psi1_energy-eps-converted-to.pdf}
%        \includegraphics[width = 0.3\textwidth, height = 0.2\textwidth]{biharmonic1d/psi2_energy-eps-converted-to.pdf}
%        \includegraphics[width = 0.3\textwidth, height = 0.2\textwidth]{biharmonic1d/psi32_energy-eps-converted-to.pdf}}
    \caption{One-dimensional fourth-order elliptic operator~\eqref{eqn:1dbiharmonic}.}
    \label{fig:1dbiharmonicpsi}
\end{figure}

We have sampled a force $f \in L^2(D)$ from the same model~\eqref{eqn:flexural} as the flexural rigidity. Using the MsFEM, we use two different sets of basis functions $\{\psi_{i,q}\}_{i=1, q=1}^{m, 2}$ and $\{\psi_{i}\}_{i=1}^{m}$ to solve the corresponding fourth-order elliptic equation~\eqref{eqn:1dbiharmonic}, and get solutions $u_{h,1}$ and $u_{h,0}$ respectively. We show their errors in the energy norm, i.e., $\|u_{h,1}-u\|_H$ and $\|u_{h,0} - u\|_H$ in Figure~\ref{fig:exp1_errordecay}. We can see that $\|u_{h,1}-u\|_H$ decays quadratically with respect to the patch size $h$, while $\|u_{h,0}-u\|_H$ decays only linearly. Therefore, to obtain the optimal convergence rate $h^2$ in the energy norm, it is necessary to include all the piecewise linear functions in the space $\Phi$, as we have proved in Theorem~\ref{thm:conditioning1} and Eqn.~\eqref{eqn:Herror2}.
%\begin{figure}[ht]
%\centering
%\includegraphics[width = 0.4\textwidth]{biharmonic1d/load-eps-converted-to.pdf}
%\caption{Highly oscillatory load without scale separation.}\label{fig:exp1_load}
%\end{figure}
\begin{figure}[ht]
\centering
\includegraphics[width = 0.4\textwidth]{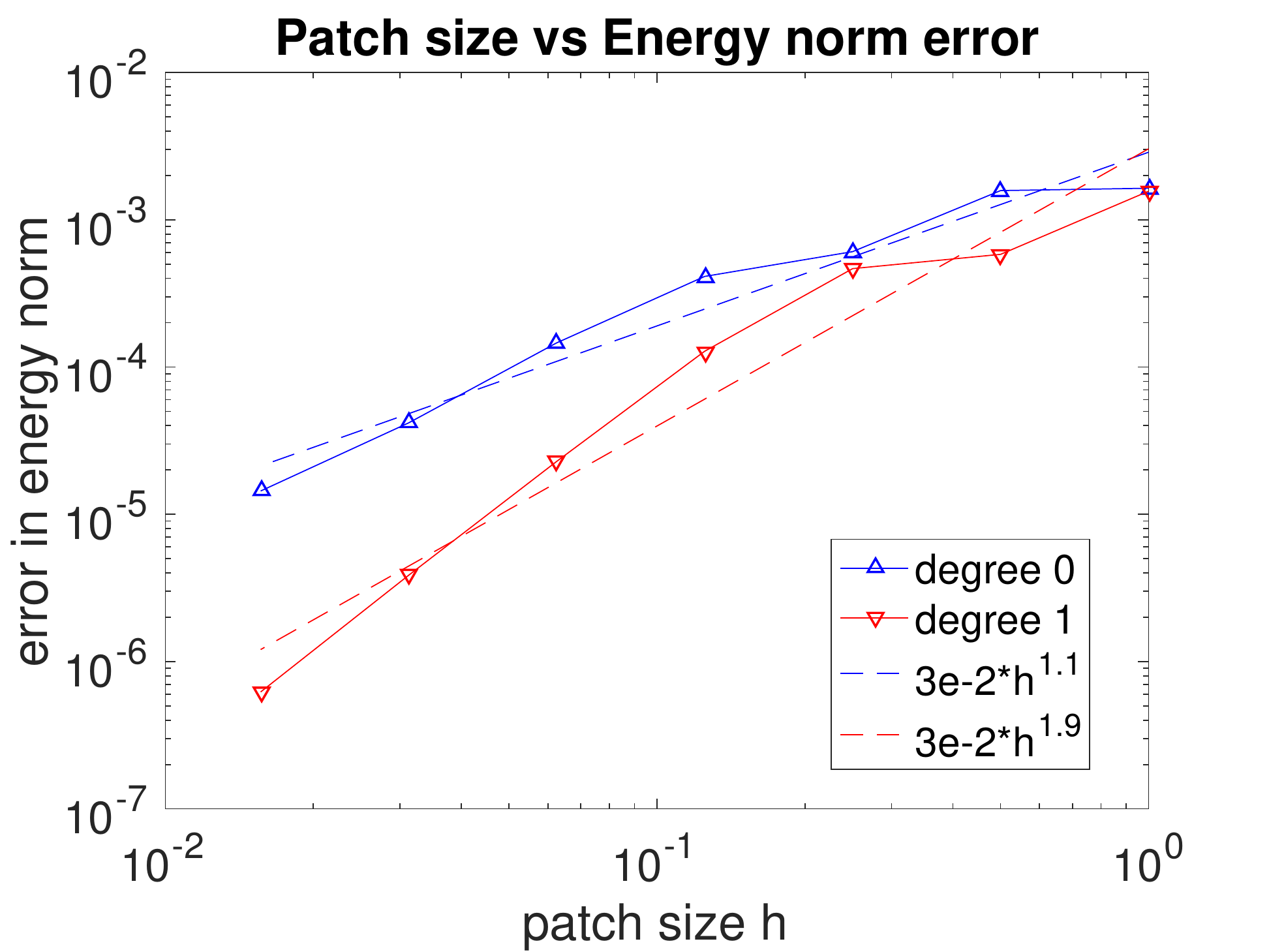}
\caption{Error of the finite element solutions: $\|u_{h,0}-u\|_H$ and $\|u_{h,1} - u\|_H$.}\label{fig:exp1_errordecay}
\end{figure}

\subsection{The 2D fourth-order elliptic operator}
\label{subsec:2dbiharmonic}
\referee{
Consider the solution operator of the 2D fourth-order elliptic equation on domain $D = (0,1)^2$
\begin{equation}\label{eqn:2dbiharmonic}
\begin{split}
	&\partial_x^2( a_{20}(x,y) \partial_x^2 u(x,y)) + \partial_y^2( a_{02}(x,y) \partial_y^2 u(x,y)) + 2 \partial_{xy}( a_{11}(x,y) \partial_{xy} u(x,y)) = f(x,y),\\
	&u \in H_0^2(D),
\end{split}
\end{equation}
which describes the vibration $u$ of a clamped plate subject to a transverse force $f\in L^2(D)$. The coefficients in the operator are given by
\begin{equation}\label{eqn:2dbiharmonic_coef}
\begin{split}
	a_{20}(x,y) = a_{02}(x,y) = &\frac{1}{6}( \frac{1.1 + \sin(2\pi x/\epsilon_1)}{1.1 + \sin(2\pi y/\epsilon_1)} + \frac{1.1 + \sin(2\pi y/\epsilon_2)}{1.1 + \cos(2\pi x/\epsilon_2)} + \\
	& \frac{1.1 + \cos(2\pi x/\epsilon_3)}{1.1 + \sin(2\pi y/\epsilon_3)} + \frac{1.1 + \sin(2\pi y/\epsilon_4)}{1.1 + \cos(2\pi x/\epsilon_4)} + \sin(4 x^2 y^2) + 1 ),\\
	a_{11}(x,y) = & 1 + \frac{1}{2} \sin\left(\sum_{k=1}^K k^{-\alpha} (\zeta_{1k} \sin(kx) + \zeta_{2k} \cos(ky))\right),
\end{split}
\end{equation}
where $\epsilon_1=\frac{1}{5}, \epsilon_2=\frac{1}{13}, \epsilon_3=\frac{1}{17}, \epsilon_4=\frac{1}{31}$, $K=20$, $\alpha = 0$, and $\{\zeta_{1k}\}_{k=1}^K$ and $\{\zeta_{2k}\}_{k=1}^K$ are two independent random vectors with independent entries uniformly distributed in $[-1/2, 1/2]$. }

\referee{
Based on the uniform partition with grid size $h_x = h_y = \frac{1}{8}$, we construct the piecewise linear function space $\Phi$, which has dimension $n = 3 m = 192$. We solve the quadratic optimization problem~\eqref{eqn:psivariationalk2} with the weighted extended B-splines (Web-splines~\cite{Hollig2005}) of degree 3 on the uniform refined grid with grid size $h_{x,f} = h_{y,f} = \frac{1}{32}$. The 2D Gaussian quadrature with 5 points on each axis is utilized to compute the integral on each fine grid cell. The three basis functions associated with the patch $[1/2-h_x, 1/2]\times [1/2-h_y, 1/2]$ are shown in Figure~\ref{fig:2dbiharmonicpsi}. We also show them in the log-scale in Figure~\ref{fig:2dbiharmonicpsi_log}. We can clearly see that the basis functions decay exponentially fast away from its associated patch, which validates our Theorem~\ref{thm:expdecay2simple}.}
\begin{figure}[ht]
\centering
\includegraphics[width = 0.3\textwidth]{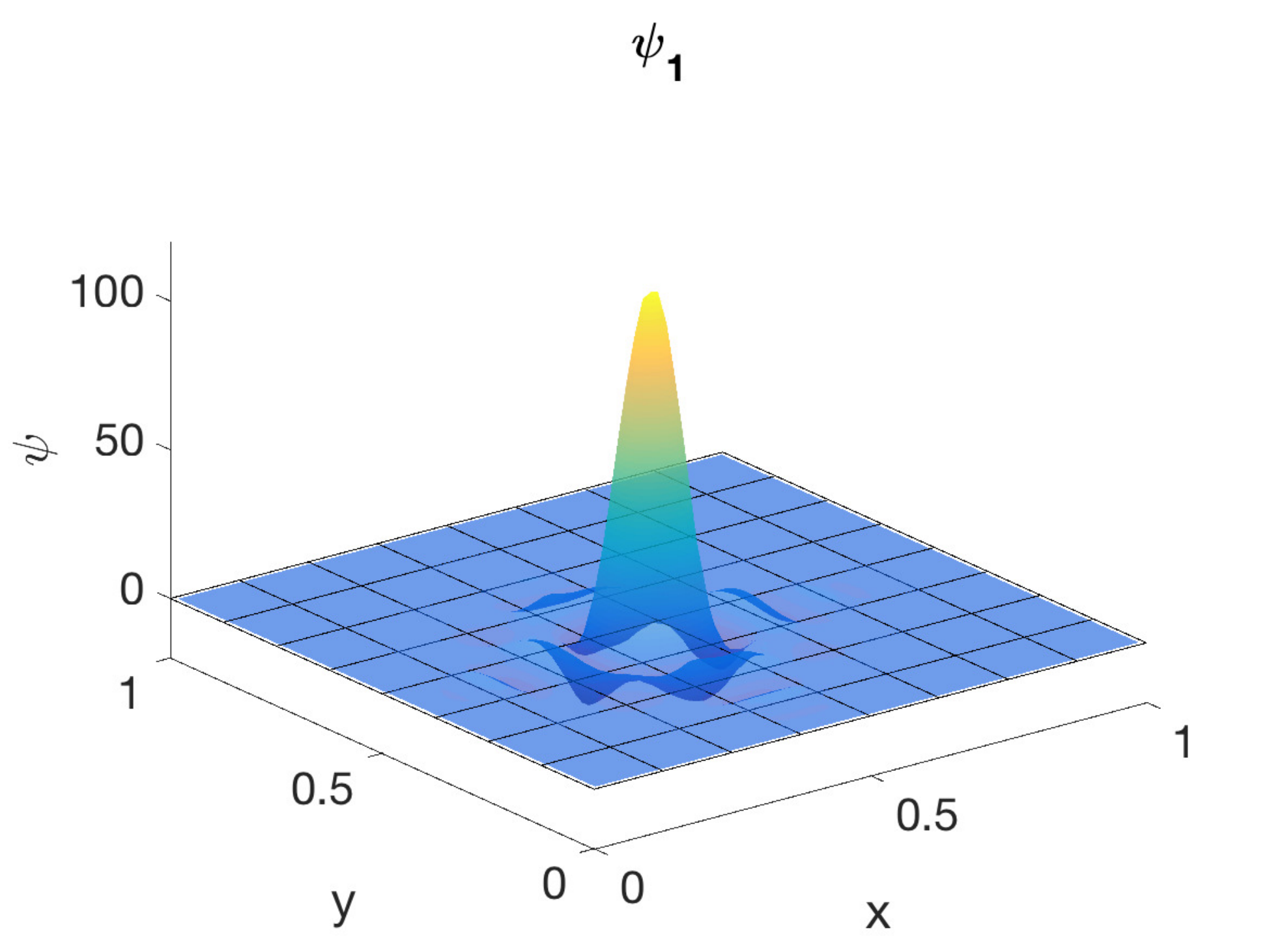}
\includegraphics[width = 0.3\textwidth]{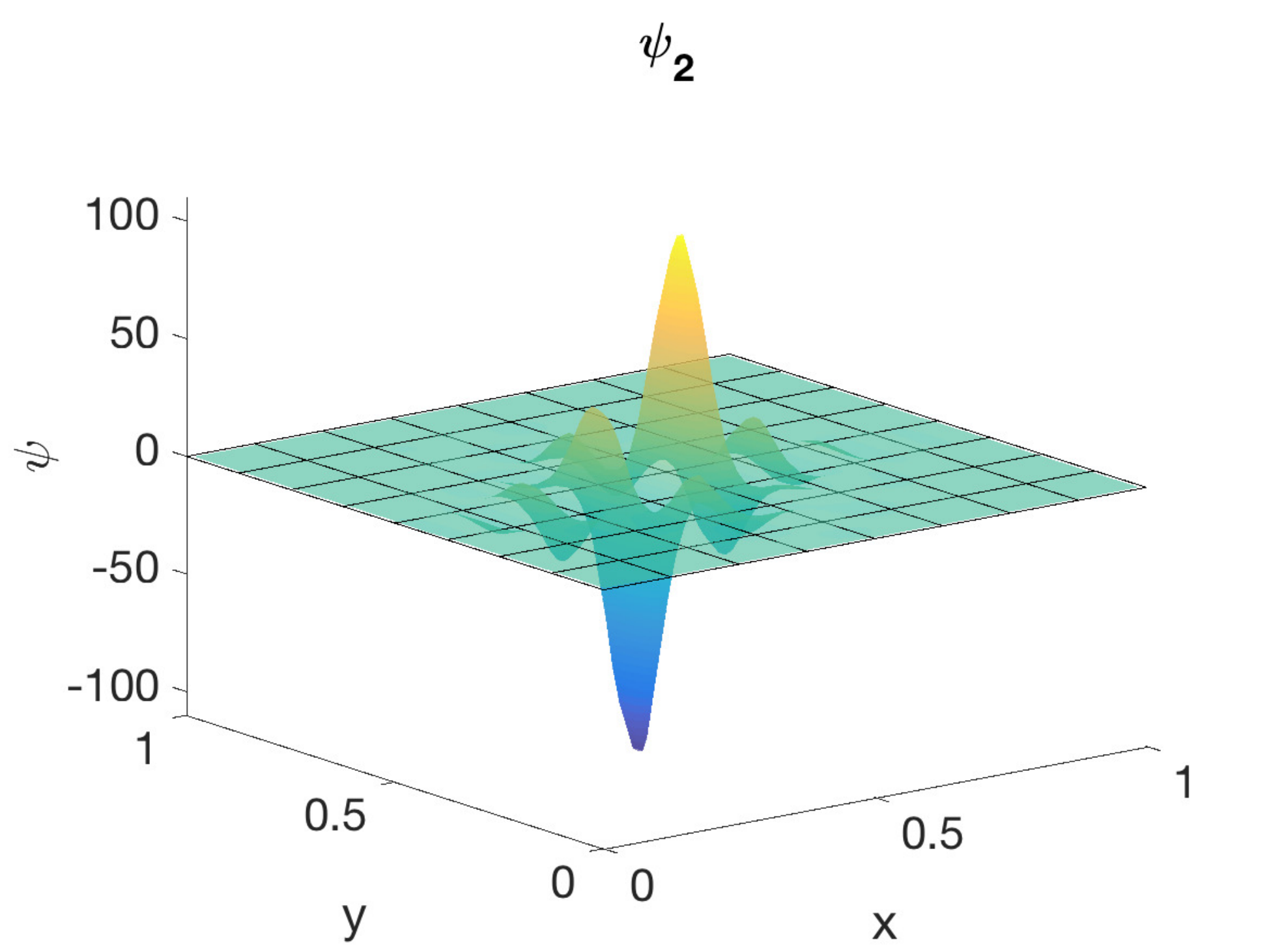}
\includegraphics[width = 0.3\textwidth]{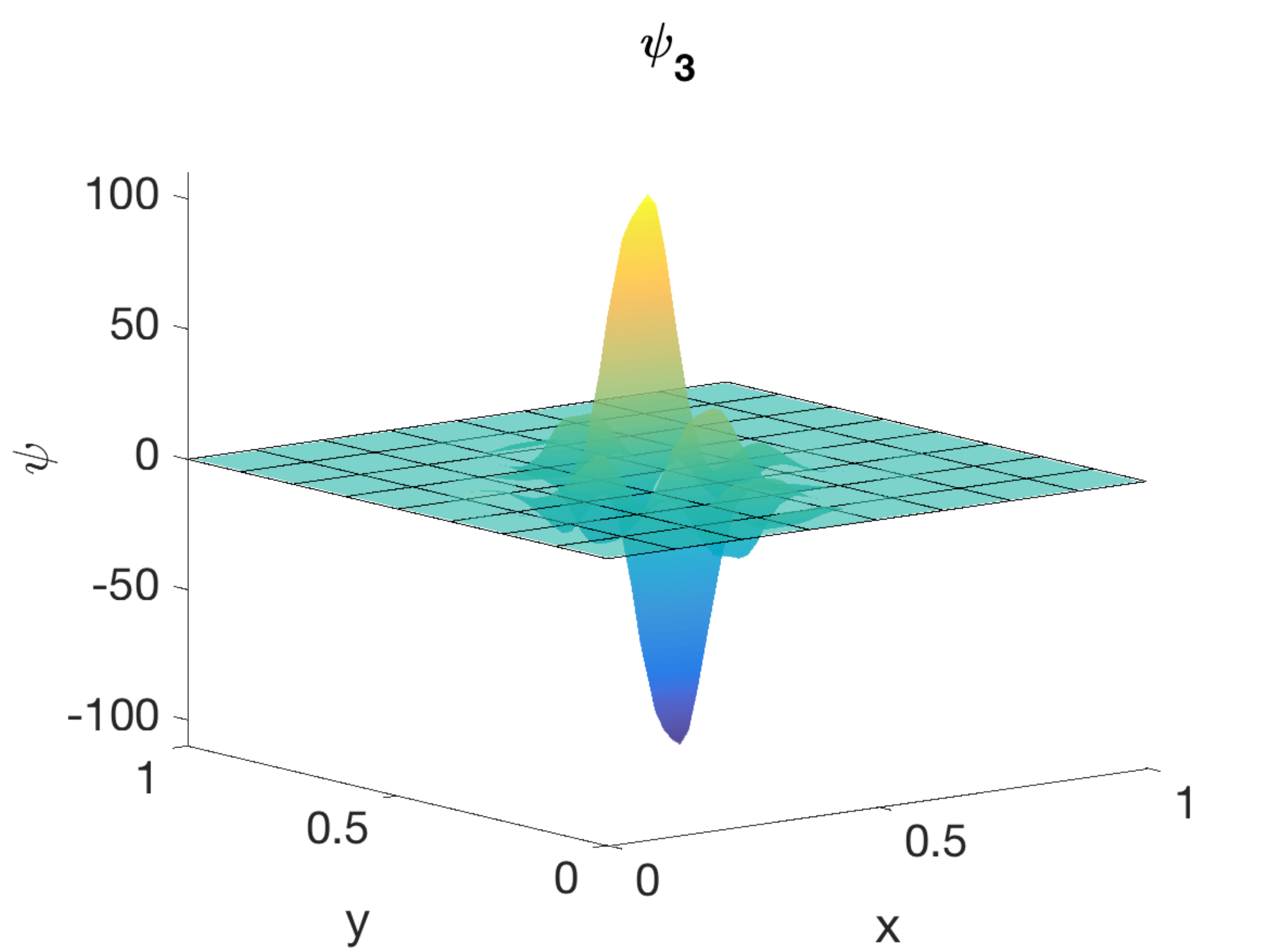}
\caption{Three basis functions associated with patch $[1/2-h_x, 1/2]\times [1/2-h_y, 1/2]$.}\label{fig:2dbiharmonicpsi}
\end{figure}
\begin{figure}[ht]
\centering
\includegraphics[width = 0.3\textwidth]{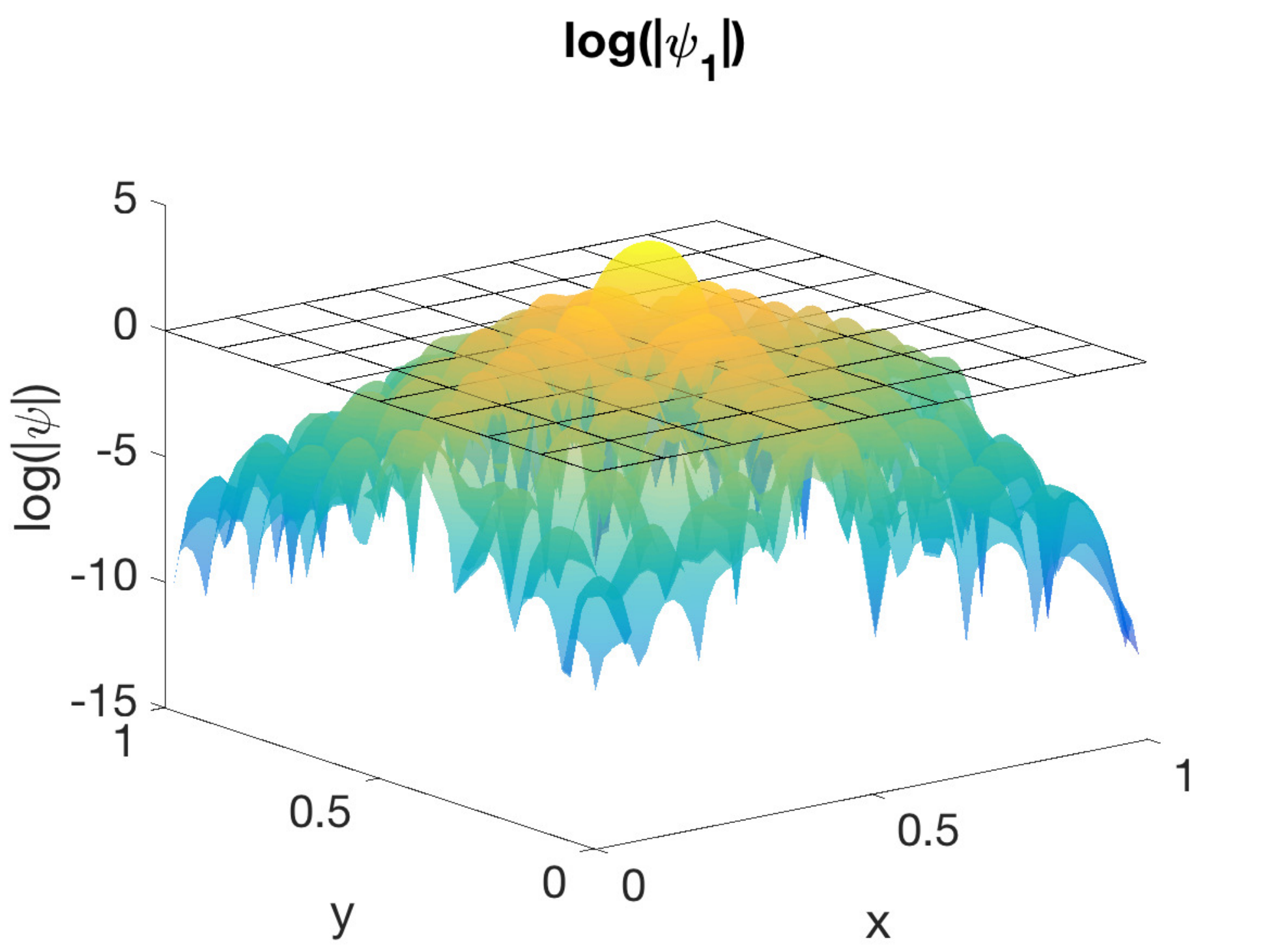}
\includegraphics[width = 0.3\textwidth]{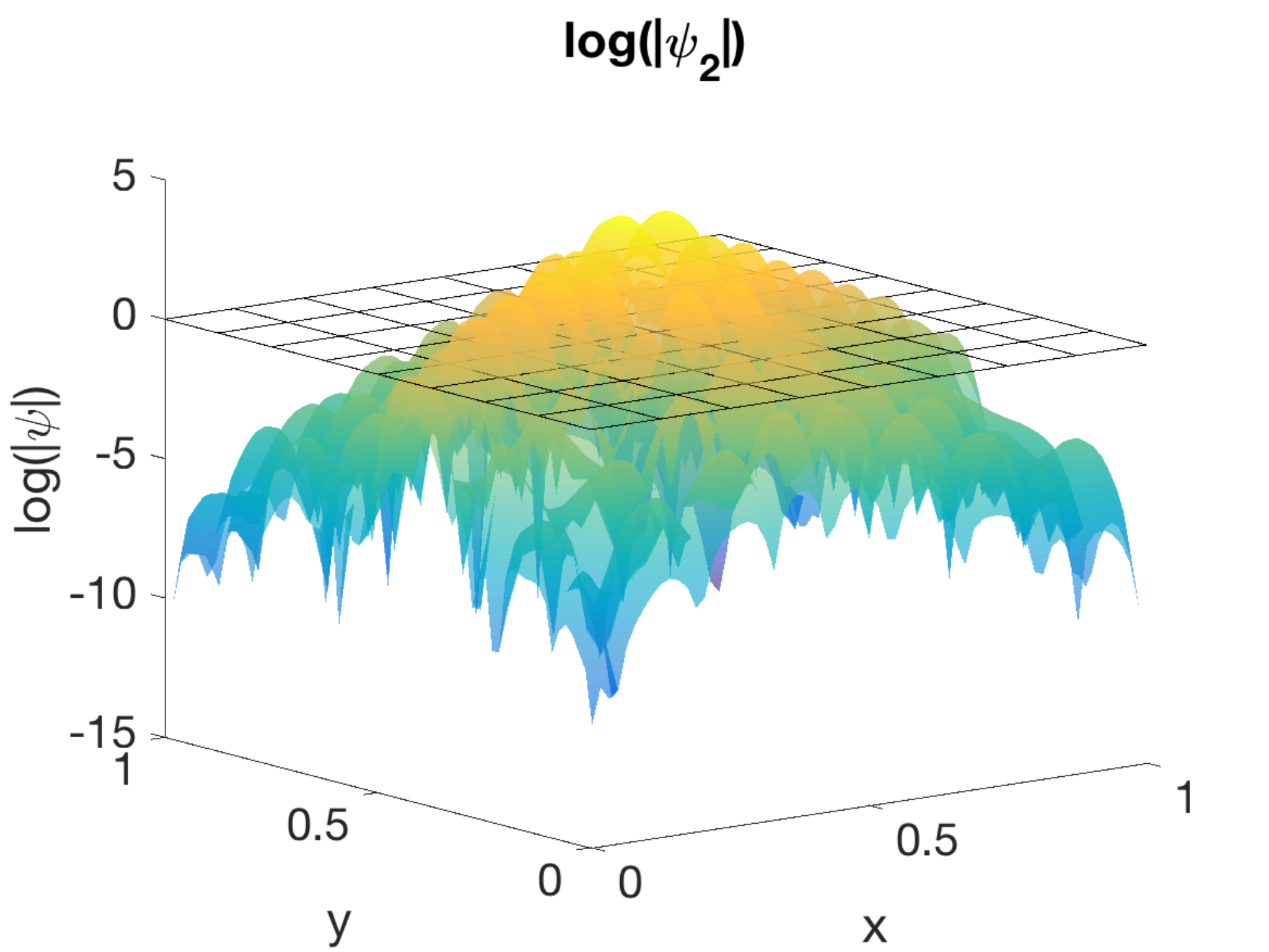}
\includegraphics[width = 0.3\textwidth]{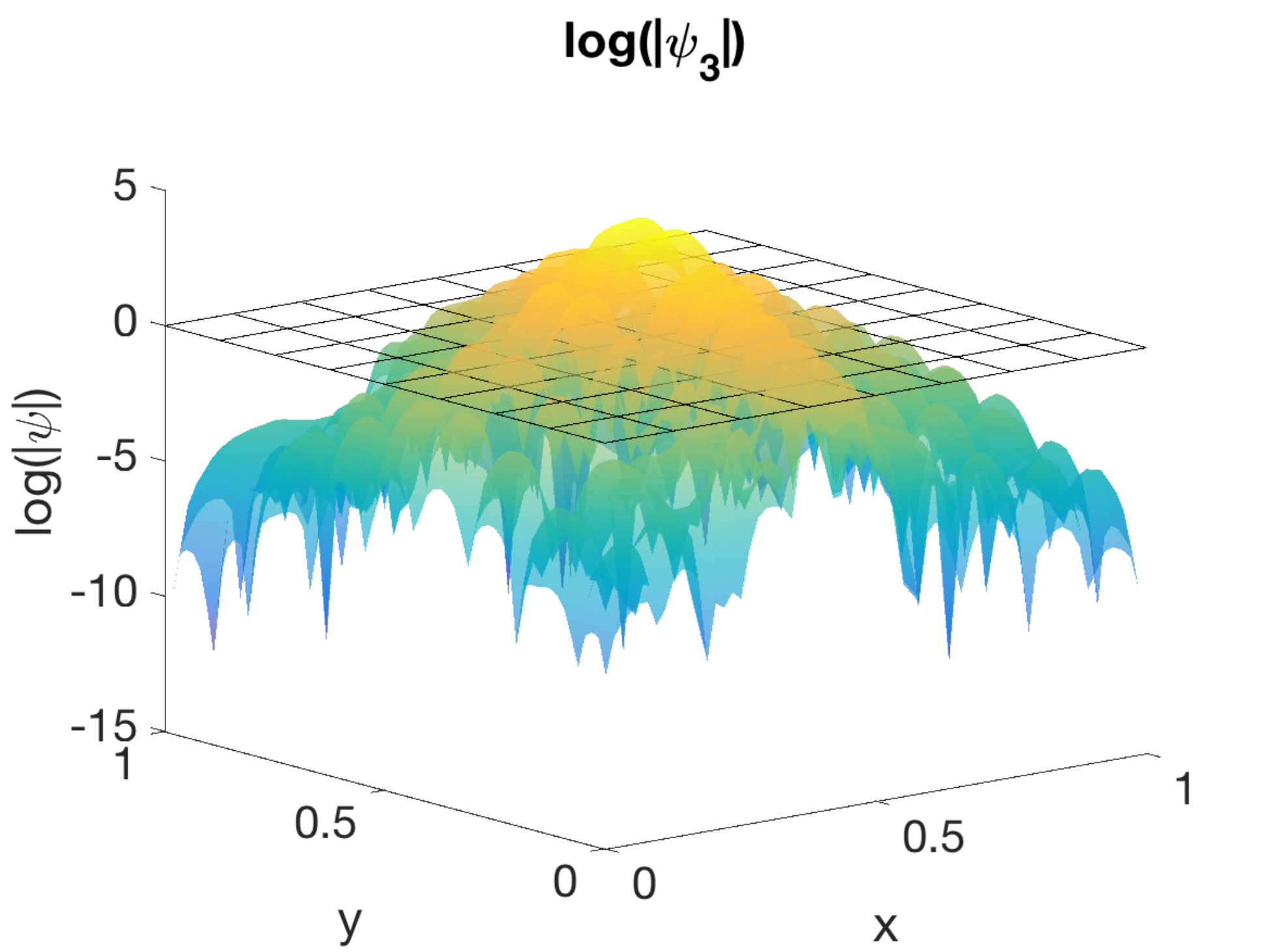}
\caption{Three basis functions associated with patch $[1/2-h_x, 1/2]\times [1/2-h_y, 1/2]$ in log-scale.}\label{fig:2dbiharmonicpsi_log}
\end{figure}

\referee{
We point out that the stiffness matrix for the fourth-order elliptic operator~\eqref{eqn:2dbiharmonic} becomes ill-conditioned very quickly when we refine the grid size. A carefully designed numerical strategy is required to validate the optimal convergence rate. We will leave this to our future work.}

\section{Concluding Remarks}
\label{k1sec:conclusions}
In this paper, we have developed a general strategy to compress a class of self-adjoint higher-order elliptic operators by minimizing the energy norm of the localized basis functions. These energy-minimizing localized basis functions are obtained by solving decoupled local quadratic optimization problems with linear constraints, and they give optimal approximation property of the solution operator. For a self-adjoint, bounded and strongly elliptic operator of order $2k$ ($k \ge 1$), we have proved that with support size $O(h \log(1/h))$, our localized basis functions can be used to compress higher-order elliptic operators with the optimal compression rate $O(h^{2k})$. We have applied our new operator compression strategy in different applications. For elliptic equations with rough coefficients, our localized basis functions can be used as multiscale basis functions, which gives the optimal convergence rate $O(h^k)$ in the energy norm. In the application of the sparse PCA, our localized basis functions achieve nearly optimal sparsity and the optimal approximation rate simultaneously when the covariance operator to be compressed is the solution operator of an elliptic operator. We remark that a number of Mat\'{e}rn covariance kernels are related to the Green's functions of some elliptic operators. 

There are several directions we can explore in the future work. First of all, the constants in both the compression error and the localization depend on the contrast of the coefficients, which makes the existing methods inefficient for coefficients with high contrast. Other methods (e.g., \cite{Grasedyck_ALbasis_2012, maalqvist2014localization, owhadi2015bayesian, owhadi2015multi}) also suffer from the same limitation. Our sparse operator compression framework can be used to deal with this high contrast case, and we will report our findings in our upcoming paper. Secondly, in the application of the sparse PCA, our current construction requires the knowledge of the underlying elliptic operator $\CalL$. We believe that it is possible to construct these localized basis functions using only the covariance function. Moreover, given any covariance operator, which may not be the solution operator of an elliptic operator, we can still define the Cameron--Martin space and the corresponding energy-minimizing basis functions. We are interested in the localization and compression properties of these energy-minimizing basis functions in this general setting. Our preliminary results show that the energy-minimizing basis functions still enjoy fast decay rate away from its associated patch, although the exponential decay may not hold true any more. Thirdly, it is interesting to apply our framework to the graph Laplacians, which can be viewed as discretized elliptic operators. Along this direction, we would like to develop an algorithm with nearly linear complexity to solve linear systems with graph Laplacians. Finally, we are also interested in applying our method to construct localized Wannier functions and to compress the Hamiltonian in quantum chemistry. Unlike the second-order elliptic operators with multiscale diffusion coefficients, all multiscale features of the Hamiltonian $\CalH = -\Delta + V(x)$ lie in its potential $V(x)$. Some adaptive domain partition strategy may prove to be useful in this application.

\appendix

\section{More on Lemma~\ref{lem:scaling}}\label{app:moreonscaling}
In this section, we prove that $C(k, s, d, \Omega_1)$ can be bounded by $C(k, s, d, \delta)$ and give an explicit formula of $C(k, s, d, \delta)$ for the case $k = s =1$. Before we do this, we need the following comparison lemma.
\begin{lemma}\label{lem:compare2}
Let $\Omega$ be a smooth, bounded, open subset of $\R^d$ and $S$ is a smooth subdomain in $\Omega$. Let $G_{\Omega}$ be the Green's function of $\CalL = (-1)^k \sum_{|\sigma| = k} D^{2 \sigma}$ with the homogeneous Dirichlet boundary condition on $\partial \Omega$ and $G_{S}$ be the Green's function of $\CalL$ with the homogeneous Dirichlet boundary condition on $\partial S$. Then, for all $f\in L^2(\Omega)$, we have
\begin{equation}\label{eqn:compare2}
	\int_{S} \int_{S} G_{S}(x,y) f(x) f(y) \rd x\, \rd y \le \int_{\Omega} \int_{\Omega} G_{\Omega}(x,y) f(x) f(y) \rd x\, \rd y.
\end{equation}
\end{lemma}
\begin{proof}
Let $f\in L^2(\Omega)$. Let $\psi_{\Omega}$ be the solution of $\CalL \psi_{\Omega} = f$ with the homogeneous Dirichlet boundary conditions on $\partial \Omega$ and $\psi_{S}$ be the solution of $\CalL \psi_{S} = f$ with the homogeneous Dirichlet boundary conditions on $\partial S$. Observe that $\psi_{\Omega}$ and $\psi_{S}$ are the unique minimizers of $I_{\Omega}(u, f) = \frac{1}{2} \sum_{|\sigma|=k}\int_{\Omega} |D^{\sigma} u|^2 - \int_{\Omega} u f$ with
\begin{equation}\label{eqn:variational2}
\begin{split}
	\psi_{\Omega} &= \argmin_{u \in H_0^k(\Omega)} I_{\Omega}(u, f), \qquad \psi_{S} = \argmin_{u \in H_0^k(S; \Omega)} I_{\Omega}(u, f)\\
	H_0^k(S; \Omega) &:= \{u \in H_0^k(\Omega): u \equiv 0 \text{ on } \Omega\backslash S\}.
\end{split}
\end{equation}
Moreover, we have
\begin{equation}\label{eqn:minimal2}
\begin{split}
	I_{\Omega}(\psi_{\Omega}, f) &= -\frac{1}{2} \int_{\Omega} \psi_{\Omega} f = - \frac{1}{2} \int_{\Omega} \int_{\Omega} G_{\Omega}(x,y) f(x) f(y) \rd x\, \rd y, \\
	I_{\Omega}(\psi_{S}, f) &= - \frac{1}{2} \int_{S} \psi_{S} f = - \frac{1}{2} \int_{S} \int_{S} G_{S}(x,y) f(x) f(y) \rd x\, \rd y.
\end{split}
\end{equation}
Since $H_0^k(S; \Omega)$ is a subset of $H_0^k(\Omega)$, we obtain
\begin{equation}\label{eqn:variationalcompare2}
	I_{\Omega}(\psi_{\Omega}, f) \le I_{\Omega}(\psi_{S}, f),
\end{equation}
which proves the lemma.
\end{proof}

Note that Lemma~\ref{lem:compare2} in fact holds true for the general operator $\sum\limits_{0 \le |\sigma|, |\gamma| \le k} (-1)^{|\sigma|} D^{\sigma} (a_{\sigma \gamma}(x) D^{\gamma} u)$ with various boundary conditions. 
%It is related to the Law of Total Variance in probability. 
Notice that $\Omega_1$ is a smooth, bounded, open subset of $\R^d$ that satisfies $B(0,\delta/2) \subset \Omega_1 \le B(0,1)$. By Lemma~\ref{lem:compare2}, we are able to bound the energy norm on $\Omega_1$ by that on $B(0,\delta/2)$ and $B(0,1)$. To simplify the notation, we omit the subscript ``1'' in the rest of this section.

\begin{proposition}\label{prop:Clowerbound}
$C(k, s, d, \Omega)$ (defined in Eqn.~\eqref{eqn:Cbyeig}) can be bounded by $C(k, s, d, \delta)$ which only depends on $k, s, d$ and $\delta$. Moreover, we can set
\begin{equation}\label{eqn:cases1}
	C(1, 1, d, \delta) = 2 \sqrt{d (d+2)} \delta^{-1-d/2}.
\end{equation}
\end{proposition}
\begin{proof}
From the definition~\eqref{eqn:Cbyeig}, we have
\begin{equation}\label{eqn:Cbyeig2}
	\left( C(k, s, d, \Omega) \right)^2 = \lambda_{\max}(M, S) = \max_{p \in \CalP_{s-1}} \frac{\int_{\Omega} p^2(x)\rd x}{ \int_{\Omega} \int_{\Omega} G(x,y)p(x)p(y)\rd x\, \rd y},
\end{equation}
where $G(x,y)$ is the Green's function of $\CalL = (-1)^k \sum_{|\sigma| = k} D^{2 \sigma}$ with the homogeneous Dirichlet boundary condition on $\partial \Omega$. Notice that $B(0,\delta/2) \subset \Omega \subset B(0,1)$. Utilizing Lemma~\ref{lem:compare2}, we have
\begin{equation*}
	\lambda_{\max}(M, S) \le \max_{p \in \CalP_{s-1}} \frac{\int_{B(0,1)} p^2(x)\rd x}{\int_{B(0,\delta/2)} \int_{B(0,\delta/2)} G_{\delta/2}(x,y)p(x)p(y)\rd x\, \rd y} := \lambda_{\max}(\hat{M}, \hat{S}),
\end{equation*}
where $G_{\delta/2}$ is the Green's function of $\CalL$ with the homogeneous Dirichlet boundary condition on $\partial B(0,\delta/2)$, $\lambda_{\max}(\hat{M}, \hat{S}) > 0$ is the largest generalized eigenvalue of $\hat{M}$ and $\hat{S}$ with
\begin{equation}\label{eqn:SMhat}
	\hat{S}(i,j) = \int_{B(0,\delta/2)} \int_{B(0,\delta/2)} G_{\delta/2} p_i p_j = \int_{B(0,\delta/2)} u_{_{\delta/2,i}} p_j, \qquad \hat{M}(i,j) = \int_{B(0,1)} p_i p_j.
\end{equation}
Here, $\{p_1, p_2, \ldots, p_Q\}$ are all the monomials defined in Lemma~\ref{lem:scaling} and $u_{_{\delta/2,i}} = \CalL^{-1} p_i$ with the homogeneous Dirichlet boundary condition on $\partial B(0,\delta/2)$. It is obvious that $\lambda_{\max}(\hat{M}, \hat{S})$ only depends on $k$, $s$, $d$ and $\delta$. Thus, we can choose
\begin{equation}\label{eqn:Cbyeig3}
	C(k, s, d, \delta) = \sqrt{\lambda_{\max}(\hat{M}, \hat{S})}.
\end{equation}
Since $\Omega$ has diameter at most 1, there exists $x_0 \in \Omega$ such that $\Omega \subset B(x_0, 1/2)$. Therefore, we have $\int_{\Omega} p^2(x)\rd x \le \int_{B(x_0, 1/2)} p^2(x)\rd x$, and we have a tighter bound for $M$ in the case $s = 1$: $M \le \hat{M} := \int_{B(x_0,1/2)} \rd x = A_{d-1} / (d 2^d)$, where $A_{d-1}$ is the surface area of the $(d-1)$-sphere of radius 1 (set $A_{0} = 2$). 

For the case $s=k=1$, $u_{_{\delta/2,1}}$ (defined as $\CalL^{-1} p_1$ with the homogeneous Dirichlet boundary condition on $\partial B(0,\delta/2)$) can be solved explicitly:
\begin{equation*}
	u_{_{\delta/2,1}} = \left( (\delta/2)^2 - r^2 \right)/(2d).
\end{equation*}
Then, we have
\begin{equation*}
\quad\hat{S} = \frac{1}{d^2 (d+2)} \left(\frac{\delta}{2}\right)^{d+2} A_{d-1}, \quad \hat{M} = A_{d-1} / (d 2^d).
\end{equation*}
Since $\lambda_{\max}(\hat{M}, \hat{S}) = \hat{M}/\hat{S}$ in the case of $s=1$, Eqn.~\eqref{eqn:cases1} naturally follows.
\end{proof}

\section{Derivations involving \texorpdfstring{$I_1$}{I1}}
\label{app:I1derivation}
\subsection{From Eqn.~\eqref{eqn:I1_1_2} to Eqn.~\eqref{eqn:I1_1_3} in the proof of Theorem~\ref{thm:expdecay2}}
\label{subapp:l1_expdecay2}
We want to prove that there exists a constant $C_1(k,d)$ such that
\begin{equation}\label{eqn:l1_expdecay2}
	\sum_{|\sigma| \le k} \int_{S^*} \left| \sum_{\stackrel{\sigma_1 + \sigma_2 = \sigma}{|\sigma_1| \ge 1}} \binom{\sigma}{\sigma_1} D^{\sigma_1} \eta D^{\sigma_2} \psi_{i,q} \right|^2 \le C_1^2 C_{\eta}^2 \sum_{s=1}^k \sum_{s'=1}^s (l h)^{-2 s'} |\psi_{i,q}|_{s-s', 2, S^*}^2.
\end{equation}
\begin{proof}
We re-arrange terms on the left-hand side with the same $|\sigma|$ and use the Cauchy inequality:
\begin{eqnarray}
	LHS &=& \sum_{s=1}^k \sum_{|\sigma| = s} \int_{S^*} \left| \sum_{\sigma_1 \le \sigma, |\sigma_1| \ge 1} \binom{\sigma}{\sigma_1} D^{\sigma_1} \eta D^{\sigma - \sigma_1} \psi_{i,q} \right|^2 \nonumber \\
	&\le& \sum_{s=1}^k \sum_{|\sigma| = s} \left(\sum_{\sigma_1 \le \sigma, |\sigma_1| \ge 1} \binom{\sigma}{\sigma_1}^2\right) \left(\sum_{\sigma_1 \le \sigma, |\sigma_1| \ge 1} \int_{S^*} |D^{\sigma_1} \eta|^2 |D^{\sigma - \sigma_1} \psi_{i,q}|^2 \right) \nonumber\\
	&\le& C_{1,1}^2 C_{\eta}^2 \sum_{s=1}^k \sum_{|\sigma| = s} \sum_{\sigma_1 \le \sigma, |\sigma_1| \ge 1} \int_{S^*} (l h)^{- 2|\sigma_1|} |D^{\sigma - \sigma_1} \psi_{i,q}|^2, \label{eqn:appl1_1}
\end{eqnarray}
where we have used $|D^{\sigma_1} \eta| \le C_{\eta} (l h)^{-|\sigma_1|}$ and $C_{1,1} := \max\limits_{|\sigma| \le k} \sum\limits_{\sigma_1 \le \sigma, |\sigma_1| \ge 1} \binom{\sigma}{\sigma_1}^2$. We re-arrange the terms in Eqn.~\eqref{eqn:appl1_1} by grouping terms with the same $|\sigma_1|$, and we get
\begin{equation*}
	\sum_{|\sigma| = s} \sum_{\sigma_1 \le \sigma, |\sigma_1| \ge 1} \int_{S^*} (l h)^{- 2|\sigma_1|} |D^{\sigma - \sigma_1} \psi_{i,q}|^2 \le \sum_{s' = 1}^s \sum_{|\sigma_1| = s'} N(s, \sigma_1) (l h)^{- 2|\sigma_1|} |D^{\sigma - \sigma_1} \psi_{i,q}|^2,
\end{equation*}
where $N(s,\sigma_1) = \sum\limits_{|\sigma| = s} \sum\limits_{\sigma_1 \le \sigma, |\sigma_1| \ge 1} 1$. Suppose that $N(s,\sigma_1) \le C_{1,2}$ for all $1\le s \le k$ and $1 \le |\sigma_1| \le s$. Then, we have
\begin{equation}\label{eqn:appl1_2}
	\sum_{|\sigma| = s} \sum_{\sigma_1 \le \sigma, |\sigma_1| \ge 1} \int_{S^*} (l h)^{- 2|\sigma_1|} |D^{\sigma - \sigma_1} \psi_{i,q}|^2 \le C_{1,2} \sum_{s' = 1}^s (l h)^{- 2 s'} |\psi_{i,q}|^2_{s-s', 2, S^*}.
\end{equation}
Combining Eqn.~\eqref{eqn:appl1_1} and \eqref{eqn:appl1_2}, and denoting $C_1 = C_{1,1} C_{1,2}^{1/2}$, we have proved Eqn.~\eqref{eqn:l1_expdecay2}.
\end{proof}

\begin{remark}
If there are no lower-order terms, we can obtain
\begin{equation}\label{eqn:l1_expdecay2simple}
	\sum_{|\sigma| = k} \int_{S^*} \left| \sum_{\stackrel{\sigma_1 + \sigma_2 = \sigma}{|\sigma_1| \ge 1}} \binom{\sigma}{\sigma_1} D^{\sigma_1} \eta D^{\sigma_2} \psi_{i,q} \right|^2 \le C_1^2 C_{\eta}^2 \sum_{s'=1}^k (l h)^{-2 s'} |\psi_{i,q}|_{k-s', 2, S^*}^2.
\end{equation}
Here, we can take $C_1 = C_{1,1} C_{1,2}^{1/2}$ with $C_{1,1} := \max\limits_{|\sigma| = k} \sum\limits_{\sigma_1 \le \sigma, |\sigma_1| \ge 1} \binom{\sigma}{\sigma_1}^2$ and $C_{1,2} = \max\limits_{1 \le |\sigma_1| \le k} N(k,\sigma_1)$. Of course, we can simply take the same $C_1$ as in Eqn.~\eqref{eqn:l1_expdecay2}.

Eqn.~\eqref{eqn:l1_expdecay2simple} is used from Eqn.~\eqref{eqn:I1_1_2simple} to Eqn.~\eqref{eqn:I1_1_3simple} in the proof of Theorem~\ref{thm:expdecay2simple}.
\end{remark}

\subsection{Estimation of $\|\eta \psi_{i,q}\|_{H(S^*)}$ in the proof of Theorem~\ref{thm:localization}}
\label{subapp:l1_localization}
In this subsection, we will prove the following result that is used in in the proof of Theorem~\ref{thm:localization}: for all $h>0$ such that $\frac{1-h^{2k}}{1-h^2} \le 2$, we have
\begin{equation}\label{eqn:hs_bound}
	\|\eta \psi_{i,q}\|_{H(S^*)} \le \frac{C}{2} |\psi_{i,q}|_{k,2,S^*} + \sqrt{\frac{C^2}{4} |\psi_{i,q}|_{k,2,S^*}^2 + C |\psi_{i,q}|_{k,2,S^*} \|\psi_{i,q}\|_{H(S^*)} + \|\psi_{i,q}\|_{H(S^*)}^2},
\end{equation}
where $C = C_1 C_{\eta} C_p \sqrt{2 k \theta_{k,\max}}$.
\begin{proof}
We begin by expressing the following integral as a sum of two terms:
\begin{equation}\label{eqn:hs_1}
\begin{split}
&\sum_{0 \le |\sigma|, |\gamma| \le k} \int_{S^*} a_{\sigma \gamma}D^{\sigma}(\eta \psi_{i,q}) D^{\gamma}(\eta \psi_{i,q}) = \underbrace{\sum_{0 \le |\sigma|, |\gamma| \le k} \int_{S^*} \eta a_{\sigma \gamma}(x) D^{\sigma}\psi_{i,q} D^{\gamma} (\eta \psi_{i,q})}_{I_3} \\
& + \underbrace{\sum_{0 \le |\sigma|, |\gamma| \le k}\sum_{\stackrel{\sigma_1 + \sigma_2 = \sigma}{|\sigma_1| \ge 1}} \binom{\sigma}{\sigma_1} \int_{S^*} a_{\sigma \gamma}(x) D^{\sigma_1} \eta D^{\sigma_2} \psi_{i,q} D^{\gamma} (\eta \psi_{i,q})}_{I_4}.
\end{split}
\end{equation}
Repeating the same argument from Eqn.~\eqref{eqn:I1_1_1} to Eqn.~\eqref{eqn:I1_1_3}, we obtain
\begin{equation}\label{eqn:I4_1}
	|I_4| \le C_1 C_{\eta} \left( \sum_{s=1}^k \sum_{s'=1}^s h^{-2 s'} |\psi_{i,q}|_{s-s', 2, S^*}^2 \right)^{1/2} \|\eta \psi_{i,q}\|_{H(S^*)} \sqrt{\theta_{k,\max}}.
\end{equation}
Since $\psi_{i,q} \perp \CalP_{k-1}$ locally in $L^2$, from Eqn.~\eqref{eqn:generalPoincare} we have 
\begin{equation*}
	|\psi_{i,q}|_{s-s', 2, S^*} \le C_p h^{s'} |\psi_{i,q}|_{s, 2, S^*}.
\end{equation*}
Repeating the same argument from Eqn.~\eqref{eqn:I1_2_1} to Eqn.~\eqref{eqn:I1_2_3}, we conclude
\begin{eqnarray}
	  I_4 &\le& C_1 C_{\eta} C_p \sqrt{\theta_{k,\max}} \left(\sum_{s=1}^k \sum_{s'=1}^s |\psi_{i,q}|_{s, 2, S^*}^2\right)^{1/2} \|\eta \psi_{i,q}\|_{H(S^*)}  \label{eqn:I4_2_1}\\
	  &\le& C_1 C_{\eta} C_p \sqrt{\theta_{k,\max}} \left(\sum_{s=1}^k s |\psi_{i,q}|_{s, 2, S^*}^2\right)^{1/2} \|\eta \psi_{i,q}\|_{H(S^*)} \label{eqn:I4_2_2}\\
	 &\le& C_1 C_{\eta} C_p \sqrt{2 k \theta_{k,\max}} |\psi_{i,q}|_{k,2,S^*} \|\eta \psi_{i,q}\|_{H(S^*)} \label{eqn:I4_2_3}. 
\end{eqnarray}
In the last inequality~\eqref{eqn:I1_2_3}, we have used the polynomial approximation property~\eqref{eqn:generalPoincare} again and take $\frac{h^2-h^{2k}}{1-h^2} \le 1/C_p^2$ to make it true.

%Repeating the same process for $\|\eta \psi_{i,q}\|_{H(S^*)}$, we have
Repeating the same process for $I_3$, we have
\begin{equation}\label{eqn:I3_1}
\begin{split}
&I_3 = \underbrace{\sum_{0 \le |\sigma|, |\gamma| \le k} \int_{S^*} \eta^2 a_{\sigma \gamma}(x) D^{\sigma}\psi_{i,q} D^{\gamma} \psi_{i,q}}_{I_5} \\
& + \underbrace{\sum_{0 \le |\sigma|, |\gamma| \le k}\sum_{\stackrel{\sigma_1 + \sigma_2 = \sigma}{|\sigma_1| \ge 1}} \binom{\sigma}{\sigma_1} \int_{S^*} \eta a_{\sigma \gamma}(x) D^{\sigma_1} \eta D^{\sigma_2} \psi_{i,q} D^{\gamma} \psi_{i,q}}_{I_6}.
\end{split}
\end{equation}
Here, we have exchanged the index $\sigma$ and $\gamma$ so that $I_6$ has a structure similar to that of $I_4$. Since \\$\sum_{0 \le |\sigma|, |\gamma| \le k} a_{\sigma \gamma}(x) D^{\sigma}\psi_{i,q} D^{\gamma} \psi_{i,q} \ge 0$ and $|\eta(x)| \le 1$ for every $x\in D$, we obtain
\begin{equation}\label{eqn:I5}
	I_5 \le \|\psi_{i,q}\|_{H(S^*)}^2
\end{equation}
Repeating the same argument from Eqn.~\eqref{eqn:I1_1_1} to Eqn.~\eqref{eqn:I1_1_3} again, we obtain
\begin{eqnarray}
	  I_6 &=& \sum_{0 \le |\sigma|, |\gamma| \le k}\sum_{\stackrel{\sigma_1 + \sigma_2 = \sigma}{|\sigma_1| \ge 1}} \binom{\sigma}{\sigma_1} \int_{S^*} a_{\sigma \gamma}(x) \eta D^{\sigma_1} \eta D^{\sigma_2} \psi_{i,q} D^{\gamma} \psi_{i,q}  \nonumber \\
	 &\le& \left( \sum_{|\sigma| \le k} \int_{S^*} \left| \sum_{\stackrel{\sigma_1 + \sigma_2 = \sigma}{|\sigma_1| \ge 1}} \binom{\sigma}{\sigma_1} \eta D^{\sigma_1} \eta D^{\sigma_2} \psi_{i,q} \right|^2 \right)^{1/2} \|\psi_{i,q}\|_{H(S^*)} \sqrt{\theta_{k,\max}} \nonumber\\
	 &\le& C_1 C_{\eta} \sqrt{\theta_{k,\max}} \left( \sum_{s=1}^k \sum_{s'=1}^s h^{-2 s'} |\psi_{i,q}|_{s-s', 2, S^*}^2 \right)^{1/2} \|\psi_{i,q}\|_{H(S^*)} \; . \label{eqn:I6_1}
\end{eqnarray}
The derivation of Eqn.~\eqref{eqn:I6_1} is nearly the same as that of Eqn.~\eqref{eqn:l1_expdecay2} and the only difference is that we need to use $|\eta D^{\sigma_1} \eta| \le C_{\eta} h^{-|\sigma_1|}$ (thanks to $|\eta| \le 1$) in Eqn.~\eqref{eqn:appl1_1}. Using exactly the same argument from Eqn.~\eqref{eqn:I4_2_1} to Eqn.~\eqref{eqn:I4_2_3}, we conclude that for all $h>0$ such that $\frac{1-h^{2k}}{1-h^2} \le 2$,
\begin{equation}\label{eqn:I6_2}
	  I_6 \le C_1 C_{\eta} C_p \sqrt{2 k \theta_{k,\max}} |\psi_{i,q}|_{k,2,S^*} \|\psi_{i,q}\|_{H(S^*)}. 
\end{equation}
Combining Eqn.~\eqref{eqn:I3_1}, \eqref{eqn:I5} and \eqref{eqn:I6_2}, we obtain
\begin{equation}\label{eqn:I3_2}
|I_3| \le \|\psi_{i,q}\|_{H(S^*)}^2 + C_1 C_{\eta} C_p \sqrt{2 k \theta_{k,\max}} |\psi_{i,q}|_{k,2,S^*} \|\psi_{i,q}\|_{H(S^*)}. 
\end{equation}
Combining Eqn.~\eqref{eqn:hs_1}, \eqref{eqn:I4_2_3} and \eqref{eqn:I3_2}, we have
\begin{equation}\label{eqn:hs_2}
	\|\eta \psi_{i,q}\|_{H(S^*)}^2 \le \|\psi_{i,q}\|_{H(S^*)}^2 + C_1 C_{\eta} C_p \sqrt{2 k \theta_{k,\max}} |\psi_{i,q}|_{k,2,S^*} (\|\psi_{i,q}\|_{H(S^*)} + \|\eta \psi_{i,q}\|_{H(S^*)}). 
\end{equation}
Solving the above quadratic inequality, we have proved the lemma.
\end{proof}

\vspace{0.1in}
\noindent
{\bf Acknowledgments.}
The research was in part supported by NSF Grants DMS 1318377 and DMS 1613861. We would like to thank Professor Lei Zhang and Venkat Chandrasekaran for several stimulating discussions, and Professor Houman Owhadi for valuable comments.

%%% bibliography
\bibliographystyle{abbrv}
\bibliography{thesis} 

\begin{thebibliography}{10}

\bibitem{babuvska1983generalized}
I.~Babu{\v{s}}ka and J.~E. Osborn.
\newblock Generalized finite element methods: their performance and their
  relation to mixed methods.
\newblock {\em SIAM Journal on Numerical Analysis}, 20(3):510--536, 1983.

\bibitem{babuska_optimal_2011}
I.~Babu\v{s}ka and R.~Lipton.
\newblock Optimal local approximation spaces for generalized finite element
  methods with application to multiscale problems.
\newblock {\em Multiscale Modeling \& Simulation}, 9(1):373--406, Jan. 2011.

\bibitem{bachmayr2016representations}
M.~Bachmayr, A.~Cohen, and G.~Migliorati.
\newblock Representations of gaussian random fields and approximation of
  elliptic pdes with lognormal coefficients.
\newblock {\em arXiv preprint arXiv:1603.05559}, 2016.

\bibitem{bolin2011spatial}
D.~Bolin and F.~Lindgren.
\newblock Spatial models generated by nested stochastic partial differential
  equations, with an application to global ozone mapping.
\newblock {\em The Annals of Applied Statistics}, pages 523--550, 2011.

\bibitem{efendiev16jcp}
E.~Chung, Y.~Efendiev, and T.-Y. Hou.
\newblock Adaptive multiscale model reduction with generalized multiscale
  finite element methods.
\newblock {\em JCP}, 320:69--95, 2016.

\bibitem{ciarlet2002finite}
P.~G. Ciarlet.
\newblock {\em The finite element method for elliptic problems}, volume~40.
\newblock Siam, 2002.

\bibitem{DAHLKE200629}
S.~Dahlke, E.~Novak, and W.~Sickel.
\newblock Optimal approximation of elliptic problems by linear and nonlinear
  mappings {I}.
\newblock {\em Journal of Complexity}, 22(1):29 -- 49, 2006.

\bibitem{dAspremont_sparsePCA}
A.~d'Aspremont, L.~El~Ghaoui, M.~Jordan, and G.~Lanckriet.
\newblock A direct formulation for sparse {PCA} using semidefinite programming.
\newblock {\em SIAM Review}, 49(3):434--448, 2007.

\bibitem{d2013coarse}
M.~D'Elia and M.~Gunzburger.
\newblock Coarse-grid sampling interpolatory methods for approximating gaussian
  random fields.
\newblock {\em SIAM/ASA Journal on Uncertainty Quantification}, 1(1):270--296,
  2013.

\bibitem{weinan2010localized}
W.~E, T.~Li, and J.~Lu.
\newblock Localized bases of eigensubspaces and operator compression.
\newblock {\em Proceedings of the National Academy of Sciences},
  107(4):1273--1278, 2010.

\bibitem{efendiev_generalized_2013}
Y.~Efendiev, J.~Galvis, and T.~Y. Hou.
\newblock Generalized multiscale finite element methods ({GMsFEM}).
\newblock {\em Journal of Computational Physics}, 251:116--135, Oct. 2013.

\bibitem{efendiev2011multiscale}
Y.~Efendiev, J.~Galvis, and X.-H. Wu.
\newblock Multiscale finite element methods for high-contrast problems using
  local spectral basis functions.
\newblock {\em Journal of Computational Physics}, 230(4):937 -- 955, 2011.

\bibitem{Houbook09}
Y.~Efendiev and T.~Y. Hou.
\newblock {\em Multiscale Finite Element Methods: Theory and Applications}.
\newblock Springer, New York, 2009.

\bibitem{gittelson2012representation}
C.~J. Gittelson.
\newblock Representation of gaussian fields in series with independent
  coefficients.
\newblock {\em IMA Journal of Numerical Analysis}, 32(1):294--319, 2012.

\bibitem{gneiting2012matern}
T.~Gneiting, W.~Kleiber, and M.~Schlather.
\newblock Matérn cross-covariance functions for multivariate random fields.
\newblock {\em Journal of the American Statistical Association},
  105(491):1167--1177, 2010.

\bibitem{goedecker1999linear}
S.~Goedecker.
\newblock Linear scaling electronic structure methods.
\newblock {\em Reviews of Modern Physics}, 71(4):1085, 1999.

\bibitem{Grasedyck_ALbasis_2012}
L.~Grasedyck, I.~Greff, and S.~Sauter.
\newblock The {AL} basis for the solution of elliptic problems in heterogeneous
  media.
\newblock {\em Multiscale Modeling \& Simulation}, 10(1):245--258, 2012.

\bibitem{guttorp2006studies}
P.~Guttorp and T.~Gneiting.
\newblock Studies in the history of probability and statistics xlix on the
  matern correlation family.
\newblock {\em Biometrika}, 93(4):989--995, 2006.

\bibitem{henning2013_oversample}
P.~Henning and D.~Peterseim.
\newblock Oversampling for the multiscale finite element method.
\newblock {\em Multiscale Modeling \& Simulation}, 11(4):1149--1175, 2013.

\bibitem{Hollig2005}
K.~H{\"o}llig, C.~Apprich, and A.~Streit.
\newblock Introduction to the web-method and its applications.
\newblock {\em Advances in Computational Mathematics}, 23(1):215--237, 2005.

\bibitem{hou_LocalModes_2014}
T.~Y. Hou, Q.~Li, and P.~Zhang.
\newblock A sparse decomposition of low rank symmetric positive semidefinite
  matrices.
\newblock {\em Multiscale Modeling \& Simulation}, 15(1):410--444, 2017.

\bibitem{houliu2016_DCDS}
T.~Y. Hou and P.~Liu.
\newblock Optimal local multi-scale basis functions for linear elliptic
  equations with rough coefficients.
\newblock {\em Discrete and Continuous Dynamical Systems, A}, 36(8):4451--4476,
  2016.

\bibitem{hou_multiscale_1997}
T.~Y. Hou and X.-H. Wu.
\newblock A multiscale finite element method for elliptic problems in composite
  materials and porous media.
\newblock {\em Journal of Computational Physics}, 134(1):169--189, 1997.

\bibitem{hou_PetrovGalerkin_2004}
T.~Y. Hou, X.-H. Wu, and Y.~Zhang.
\newblock Removing the cell resonance error in the multiscale finite element
  method via a {P}etrov-{G}alerkin formulation.
\newblock {\em Communications in Mathematical Sciences}, 2(2):185--205, 06
  2004.

\bibitem{hughes1998variational}
T.~J. Hughes, G.~R. Feij{\'o}o, L.~Mazzei, and J.-B. Quincy.
\newblock The variational multiscale method—a paradigm for computational
  mechanics.
\newblock {\em Computer methods in applied mechanics and engineering},
  166(1):3--24, 1998.

\bibitem{Jolliffe_2003}
I.~T. Jolliffe, N.~T. Trendafilov, and M.~Uddin.
\newblock A modified principal component technique based on the {LASSO}.
\newblock {\em Journal of Computational and Graphical Statistics},
  12(3):531--547, 2003.

\bibitem{LaiLuOsher_2014}
R.~{Lai}, J.~{Lu}, and S.~{Osher}.
\newblock {Density matrix minimization with $L_1$ regularization}.
\newblock {\em Communications in Mathematical Sciences}, 13(8), 2015.

\bibitem{lindgren2011explicit}
F.~Lindgren, H.~Rue, and J.~Lindstr{\"o}m.
\newblock An explicit link between gaussian fields and gaussian markov random
  fields: the stochastic partial differential equation approach.
\newblock {\em Journal of the Royal Statistical Society: Series B (Statistical
  Methodology)}, 73(4):423--498, 2011.

\bibitem{maalqvist2014localization}
A.~M{\aa}lqvist and D.~Peterseim.
\newblock Localization of elliptic multiscale problems.
\newblock {\em Mathematics of Computation}, 83(290):2583--2603, 2014.

\bibitem{marzari2012maximally}
N.~Marzari, A.~A. Mostofi, J.~R. Yates, I.~Souza, and D.~Vanderbilt.
\newblock Maximally localized {Wannier} functions: Theory and applications.
\newblock {\em Rev. Mod. Phys.}, 84:1419--1475, Oct 2012.

\bibitem{marzari1997maximally}
N.~Marzari and D.~Vanderbilt.
\newblock Maximally localized generalized {Wannier} functions for composite
  energy bands.
\newblock {\em Phys. Rev. B}, 56:12847--12865, Nov 1997.

\bibitem{matern2013spatial}
B.~Mat{\'e}rn.
\newblock {\em Spatial variation}, volume~36.
\newblock Springer Science \& Business Media, 2013.

\bibitem{MELENK2000272}
J.~Melenk.
\newblock On n-widths for elliptic problems.
\newblock {\em Journal of Mathematical Analysis and Applications}, 247(1):272
  -- 289, 2000.

\bibitem{ming2006numerical}
P.~Ming and X.~Yue.
\newblock Numerical methods for multiscale elliptic problems.
\newblock {\em Journal of Computational Physics}, 214(1):421--445, 2006.

\bibitem{Nikolskii1975}
S.~M. Nikol'skii.
\newblock {\em Imbedding Theorems for Different Metrics and Dimensions}, pages
  231--260.
\newblock Springer Berlin Heidelberg, Berlin, Heidelberg, 1975.

\bibitem{owhadi2015bayesian}
H.~Owhadi.
\newblock Bayesian numerical homogenization.
\newblock {\em Multiscale Modeling \& Simulation}, 13(3):812--828, 2015.

\bibitem{owhadi2015multi}
H.~Owhadi.
\newblock Multigrid with rough coefficients and multiresolution operator
  decomposition from hierarchical information games.
\newblock {\em SIAM Review}, 59(1):99--149, 2017.

\bibitem{owhadi2017universal}
H.~Owhadi and C.~Scovel.
\newblock Universal scalable robust solvers from computational information
  games and fast eigenspace adapted multiresolution analysis.
\newblock {\em arXiv preprint arXiv:1703.10761}, 2017.

\bibitem{owhadi2016gamblets}
H.~Owhadi and L.~Zhang.
\newblock Gamblets for opening the complexity-bottleneck of implicit schemes
  for hyperbolic and parabolic odes/pdes with rough coefficients.
\newblock {\em arXiv:1606.07686v1}, 2016.

\bibitem{owhadi_polyharmonic_2014}
H.~Owhadi, L.~Zhang, and L.~Berlyand.
\newblock Polyharmonic homogenization, rough polyharmonic splines and sparse
  super-localization.
\newblock {\em {ESAIM}: Mathematical Modelling and Numerical Analysis},
  48(02):517--552, 2014.

\bibitem{ozolins_compressed_2013}
V.~Ozoli\c{n}\v{s}, R.~Lai, R.~Caflisch, and S.~Osher.
\newblock Compressed modes for variational problems in mathematics and physics.
\newblock {\em Proceedings of the National Academy of Sciences},
  110(46):18368--18373, 2013.

\bibitem{Peterseim2016}
D.~Peterseim.
\newblock {\em Variational Multiscale Stabilization and the Exponential Decay
  of Fine-Scale Correctors}, pages 343--369.
\newblock Springer International Publishing, Cham, 2016.

\bibitem{renardy2006introduction}
M.~Renardy and R.~C. Rogers.
\newblock {\em An introduction to partial differential equations}, volume~13.
\newblock Springer Science \& Business Media, 2006.

\bibitem{stein2012interpolation}
M.~L. Stein.
\newblock {\em Interpolation of spatial data: some theory for kriging}.
\newblock Springer Science \& Business Media, 2012.

\bibitem{strouboulis2001generalized}
T.~Strouboulis, K.~Copps, and I.~Babu{\v{s}}ka.
\newblock The generalized finite element method.
\newblock {\em Computer methods in applied mechanics and engineering},
  190(32):4081--4193, 2001.

\bibitem{vu2013fantope}
V.~Q. Vu, J.~Cho, J.~Lei, and K.~Rohe.
\newblock Fantope projection and selection: A near-optimal convex relaxation of
  sparse {PCA}.
\newblock In C.~Burges, L.~Bottou, M.~Welling, Z.~Ghahramani, and
  K.~Weinberger, editors, {\em Advances in Neural Information Processing
  Systems 26}, pages 2670--2678, 2013.

\bibitem{witten2009penalized}
D.~M. Witten, R.~Tibshirani, and T.~Hastie.
\newblock A penalized matrix decomposition, with applications to sparse
  principal components and canonical correlation analysis.
\newblock {\em Biostatistics, 10}, pages 515--534, 2009.

\bibitem{Pengchuan-thesis-2017}
P.~Zhang.
\newblock {\em Compressing Positive Semidefinite Operators with
  Sparse/Localized Bases}.
\newblock PhD thesis, California Institute of Technology, 2017.

\bibitem{Zou_PCA_06}
H.~Zou, T.~Hastie, and R.~Tibshirani.
\newblock Sparse principal component analysis.
\newblock {\em Journal of Computational and Graphical Statistics},
  15(2):265--286, 2006.

\end{thebibliography}

\end{document}